\documentclass[12pt, reqno]{amsart}

\usepackage{fullpage}
\usepackage{amssymb, amsmath}
\usepackage{mathrsfs}
\usepackage{enumerate}
\usepackage[colorlinks,breaklinks, allcolors=blue]{hyperref}
\usepackage{xcolor}
\usepackage{longtable}
\usepackage{relsize}
\usepackage{cleveref}
\usepackage{tikz-cd}

\newcommand{\mute}[1] {}

\newtheorem{theorem}[equation]{Theorem}
\newtheorem{lemma}[equation]{Lemma}
\newtheorem{claim}[equation]{Claim}
\newtheorem{proposition}[equation]{Proposition}
\newtheorem{corollary}[equation]{Corollary}
\newtheorem{conjecture}[equation]{Conjecture}
\newtheorem{noTitle}[equation]{}

\theoremstyle{definition}
\newtheorem{definition}[equation]{Definition}
\newtheorem{remark}[equation]{Remark}
\newtheorem{example}[equation]{Example}

\newtheorem*{assertion*}{Assertion}
\newtheorem{Setting}[equation]{Setting}
\newtheorem{Notation}[equation]{Notation}

\numberwithin{equation}{section}

\newcommand{\C}{\mathbb C}
\newcommand{\Z}{\mathbb Z}

\newcommand{\R}{\mathbb R}

\newcommand{\PP}{\mathbb P}

\newcommand{\CP}{\mathbb{C}P}

\newcommand{\SO}{\mathrm{SO}}

\newcommand{\U}{\mathrm{U}}

\newcommand{\ball}{\mathcal{B}}

\DeclareMathOperator{\PU}{PU}

\DeclareMathOperator{\Lie}{Lie}

\DeclareMathOperator{\id}{Id}
\DeclareMathOperator{\Morse}{Morse}
\DeclareMathOperator{\iso}{iso}
\DeclareMathOperator{\swap}{swap}
\DeclareMathOperator{\northpole}{northpole}
\DeclareMathOperator{\southpole}{southpole}
\DeclareMathOperator{\pt}{pt}
\DeclareMathOperator{\sph}{sph}

\newcommand{\acts}{\circlearrowleft}

\newcommand{\eps}{\varepsilon}

\begin{document}

\title[Semi-free Hamiltonian $S^1$-manifolds and fixed point data]{On isomorphisms of semi-free Hamiltonian $S^1$-manifolds and fixed point data}

\author{Liat Kessler}
\address{Department of Mathematics, Physics, and Computer Science, University of Haifa,
at Oranim, Tivon 36006, Israel}
\email{lkessler@math.haifa.ac.il}

\author{Nikolas Wardenski}
\address{Department of Mathematics, University of Haifa, Haifa 3498838, Israel}
\email{wardenski.math@gmail.com}

\begin{abstract}
   Following Gonzales, we answer the question of whether the isomorphism type of a semi-free Hamiltonian $S^1$-manifold of dimension six is determined by certain data on the critical levels.  
   We first give counter examples showing that Gonzales' assumptions are not sufficient for a positive answer. Then we prove that it is enough to further assume that the reduced spaces of dimension four are symplectic rational surfaces and the interior fixed surfaces are restricted to at most one level. The additional assumptions allow us to use results proven by $J$-holomorphic methods. Gonzales' answer was applied by Cho in proving that if the underlying symplectic manifold is positive monotone then the space is isomorphic to a Fano manifold with a holomorphic $S^1$-action. We show that our variation is enough for Cho's application.
\end{abstract}

\maketitle

\setcounter{secnumdepth}{1}
\setcounter{tocdepth}{1}
\tableofcontents
\section{Introduction}\label{sec:intro}

 An effective action of $S^1$ on a symplectic manifold $(M,\omega)$ is \textbf{Hamiltonian} if it admits a \textbf{momentum map}: a smooth function $\mu\colon M\to \R \cong (\Lie(S^1))^{*}$ with 
\begin{equation} \label{eq:mu}
d\mu(\cdot)=-\omega(\xi,\cdot),
\end{equation} where $\xi$ is the fundamental vector field
of the action. 
The space $(M,\omega,\mu)$ 
is called a \textbf{Hamiltonian $S^1$-manifold}.
An \textbf{isomorphism} between Hamiltonian $S^1$-manifolds is an equivariant symplectomorphism that intertwines the momentum maps. 
The $S^1$-action is \textbf{semi-free} if all stabilizers are connected, i.e., they are either the circle or the trivial group.\\

Let $(M,\omega,\mu)$  be a connected semi-free Hamiltonian $S^1$-manifold.
 Assume that $\mu$ is {proper} and that 
the momentum image is bounded, so $\mu$ has finitely many critical values $\lambda_0<\hdots < \lambda_k$.
 By \eqref{eq:mu}, the set of critical points of $\mu$ coincides with the fixed point set $M^{S^1}$.
    The assumption that the $S^1$-action is semi-free implies that it is free on
    \[
    \mu^{-1}((-\infty,\lambda_0)) \cup \mu^{-1}((\lambda_0,\lambda_1))\cup \hdots \cup \mu^{-1}((\lambda_{k-1},\lambda_k)) \cup \mu^{-1}((\lambda_k,\infty)). 
    \]
    Therefore, one can think of $M$ as the union of the above sets with
    $\mu^{-1}((\lambda_0-\eps_0,\lambda_0+\eps_0))$, $\mu^{-1}((\lambda_1-\eps_1,\lambda_1+\eps_1))$, and so on, for positive $\eps_i$, $i=0,\hdots,k$, small enough such that $\lambda_i$ is the only critical value in 
    $(\lambda_i-\eps_i,\lambda_i+\eps_i)$. 
        This observation motivated Gonzales' definition in \cite{Go11} of {\bf local data}: an atlas of compatible Hamiltonian charts; see Appendix \ref{Local Data}.\\
 Gonzales observed that a 'rigidity assumption' is needed in order to recover the local data from the {\bf fixed point data}, as defined below. 
  The {rigidity assumption} 
     will ensure that the equivariant symplectomorphism type 
of $\mu_1^{-1}((\lambda,\lambda'))$, for $\lambda<\lambda'$ two consecutive critical values, is determined by the equivariant symplectomorphism type of $\mu_1^{-1}((\lambda,t))$ for arbitrary $t\in (\lambda,\lambda')$. 
 We will use the following notation and definitions.

\begin{Notation}\label{not:semifreeHamiltonian}
 For a regular value $t$ of the momentum map $\mu \colon M \to \R$,
   the level set
   $P_t:=\mu^{-1}(t)$
   is a manifold of dimension $\dim M-1$, by the implicit function theorem; compact since $\mu$ is proper. It is  connected because $\mu$
 is Morse-Bott with even indices and $M$ is connected \cite{At82}.
  Since the $S^1$-action on $P_t$ is free, the {\em orbit space} $M_t:=P_t/S^1$ is a manifold 
   of dimension $\dim M-2$ and 
   $S^1\to P_t\to M_t$ is a principal $S^1$-bundle. Since $P_t$ is compact, there is a unique closed form $\omega_t$ on $M_t$ such that $\pi_{P_t \to M_t}^{*}{\omega_t}=\iota_{P_t \hookrightarrow M}^{*}\omega$. The form $\iota_{P_t \hookrightarrow M}^{*}\omega$ is basic, since $\omega$ is invariant and ${\iota_{P_t \hookrightarrow M}^{*}\omega}(\xi,\cdot)=d \mu \circ {\iota_{P_t \hookrightarrow M}}=0$. Moreover, the {\em reduced form} $\omega_t$ is symplectic by \cite{MW74}.

    If $M$ is of dimension six  and the action is semi-free, it turns out that even for non-extremal critical values $\lambda$, the orbit space $\mu^{-1}(\lambda)/S^1$, which we also call $M_{\lambda}$, can be given a smooth structure such that the symplectic form on $M$ descends to a symplectic form on the four-dimensional manifold $M_{\lambda}$.
    The case in which the fixed points at $\lambda$ are isolated is proven in \cite[Section 3.2]{Mc09}. The case in which there are also fixed surfaces 
    at $\lambda$ is in \cite[Section 3.3.1]{Go11}. 

    If $\lambda$ is an extremal critical value, then $M_{\lambda}$ coincides with the fixed point set $F$ at level $\lambda$; it is again connected. The symplectic form $\omega_{\lambda}$ is then the restriction of the symplectic form on $M$ to $F$.

    We endow the manifold $M_t$ with the orientation induced by the symplectic form $\omega_t$, for $t$ regular or critical.

    \end{Notation}

 \begin{definition}\label{def:rigid}\cite[Definition 1.4]{Go11}.
        Let $B$ be a smooth manifold and $\{\omega_t\}$ a smooth family of symplectic forms on $B$, parametrized by real values $t\in I=[t_0,t_1]$ in a closed interval. We say that $(B,\{\omega_t\})$ is \textbf{rigid} if
        \begin{itemize}
            \item Symp$(B,\omega_{t})\cap \text{Diff}_{0}(B)$ is path-connected for all $t\in I$, where $\text{Diff}_0(B)$ is the identity component of the diffeomorphism group of $B$.
            \item Any deformation between any two cohomologous symplectic forms  that are symplectic deformation equivalent to $\omega_{t_0}$
             on $B$ may be homotoped through symplectic deformations with fixed endpoints into an isotopy, i.e., a symplectic deformation through cohomologous forms.
        \end{itemize}
    \end{definition}
   Let $I=[t_0,t_1]\subset \mu(M)$ be an interval of regular values. Using the normalized gradient flow of $\mu$, we obtain a smooth family of diffeomorphisms $M_{t_0}\cong M_t$, $t\in I$, and use this family to view all reduced forms $\omega_t$ of $M_t$ to be defined on $M_{t_0}$.
    \begin{definition}\label{def:rigidityassumption}
        We say that $M$ satisfies the \textbf{rigidity assumption} if for all closed intervals $I=[t_0,t_1]\subset \R$ of regular values, $(M_{t_0},\{\omega_t\})$ is rigid in the sense of \Cref{def:rigid}.
    \end{definition}
     
     \begin{definition}\label{def:fixedpointdata}
   \cite[Definition 3.9]{Go11}.
 Assume that $\lambda$ is a common critical value of $M^1=(M^1,\omega^1,\mu_1)$ and $M^2=(M^2,\omega^2,\mu_2)$, non-extremal for both or extremal for both.  If $\lambda$ is non-extremal, assume moreover that $\lambda$ is \emph{simple}, meaning that all fixed point components at $\lambda$ have the same index.
         $M^1$ and $M^2$ have the \textbf{same fixed point data at a non-extremal critical value $\lambda$} if there is a symplectomorphism $f\colon M^1_{\lambda}\to M^2_{\lambda}$  between the reduced spaces at $\lambda$ that
        \begin{itemize}
            \item[(i)] sends the fixed point set of $M^1$ at $\lambda$ to the fixed point set of $M^2$ at $\lambda$;
            \item[(ii)] intertwines the index function on the fixed point sets; 
           \item[(iii)] intertwines $e(P_{\lambda}^-)$ \footnote{As we will see, there is a map $M_{\lambda-\eps}\to M_{\lambda}$ for $\eps>0$ small, under which the Euler class of the $S^1$-bundle over $\mu^{-1}(\lambda-\eps)/S^1=M_{\lambda-\eps}$ has a unique preimage, called $e(P_{\lambda}^-)$. For details, see \Cref{not:elambda}.}.
        \end{itemize}
        $M^1$ and $M^2$ have the \textbf{same fixed point data at an extremal critical value $\lambda$} if there is a symplectomorphism between the
        the corresponding extremal fixed point sets with the restrictions of the symplectic forms that intertwines the symplectic normal bundles in $M^1$ and $M^2$.\\
        We say that $M^1$ and $M^2$ have the \textbf{same fixed point data}\footnote{This data is different from the fixed point data that Hui Li defines in \cite[Definition 2]{Li03}. In Li's definition, and article, the emphasis is placed on the diffeomorphism type of the underlying symplectic manifold as opposed to its equivariant symplectomorphism type.} if they have the same critical values and the same fixed point data at each critical value.
    \end{definition}
 
Moreover, Gonzales argued \cite[Theorem 1.6]{Go11} that, in dimension six, 
the fixed point data can be further reduced to the {\bf small fixed point data}, if the fixed point sets at each critical level are either surfaces or isolated fixed points, 
assuming rigidity as above.
For $M^1$ and $M^2$ to have the {\bf same small fixed point data}, the requirement at a common non-extremal critical value $\lambda$ is that there is a diffeomorphism $f\colon M^1_{\lambda}\to M^2_{\lambda}$ that
        preserves the fixed point sets, the index function and the symplectic form on them (see \cite[Definition 3.11]{Go11}); the requirement at the maximum is that the maxima are symplectomorphic (\cite[Definition 1.3]{Go11}); the requirement at the minimum is as in the definition of $M^1$ and $M^2$ having the same fixed point data (\cite[Definition 1.3]{Go11}).\\

In this paper we follow Gonzales' vision. However, we show that for the fixed point data, or local data,  to determine the isomorphism type of a semi-free Hamiltonian $S^1$-manifold, more assumptions are required. 
We first highlight the subtleties in gluing Hamiltonian charts when more than one non-extremal critical level occurs. We give an example of closed, semi-free Hamiltonian $S^1$-manifolds of dimension six,  with only one fixed component at each critical level,  
that have the same local and fixed point data, and satisfy the rigidity assumption, but are not isomorphic. Thus, it is a counter example to \cite[Theorems 1.5, 1.6, and 2.6]{Go11}. Indeed, these manifolds are not even {\bf $\mu-S^1$-diffeomorphic}, meaning that there is no equivariant diffeomorphism between them that intertwines the momentum maps.
See \Cref{examp:main} and the following discussion. 

We then give an example of semi-free Hamiltonian $S^1$-manifolds of dimension six that satisfy the rigidity assumption, have the same small fixed point data, are isomorphic below a critical level $\lambda$, but for which there is no isomorphism of preimages of neighborhoods of the critical value.
This contradicts \cite[Lemma 3.13]{Go11}, used prominently in the proof of \cite[Theorem 1.6]{Go11}. See \Cref{ex:openDpoly}. The counter example shows that, even under Gonzales' assumptions, the small fixed point data do not necessarily determine the local data.

\begin{noTitle}\label{nt:proof-problem}
We sketch the problem in the proof of \cite[Lemma 3.13]{Go11}. Assume that we have an isomorphism $f$ between two semi-free Hamiltonian manifolds $M^1$ and $M^2$ below a common critical level $\lambda$. In order to extend the isomorphism over the critical level $\lambda$, Gonzales removes neighborhoods $U^1$ and $U^2$ around the fixed points of $M^1$ resp.\ $M^2$ at level $\lambda$ such that the flow of the gradient vector field of the momentum map induces a well defined map on $\mu_1^{-1}((\lambda-\eps,\lambda+\eps))\smallsetminus U^1$ and on $\mu_2^{-1}((\lambda-\eps,\lambda+\eps))\smallsetminus U^2$. If $f$ maps the neighborhoods $U^1_t:=U^1\cap \mu_1^{-1}(t)$ and $U^2_t:=U^2\cap \mu_2^{-1}(t)$ into each other, one could indeed obtain a $\mu-S^1$-diffeomorphism
\[
\mu_1^{-1}((\lambda-\eps,\lambda+\eps))\smallsetminus U^1 \cong
\mu_2^{-1}((\lambda-\eps,\lambda+\eps)) \smallsetminus U^2.
\]
However, $f$ does not need to map $U^1_t$ and $U^2_t$ into each other, nor does there need to exist an isotopy from $f$ through $\mu-S^1$-diffeomorphisms to a map $f'$ that intertwines the neighborhoods.
There might be topological obstructions to do so. For example, $U^1_t/S^1$ and $U^2_t/S^1$
might be 
neighborhoods of embedded 2-spheres in $M^1_t$ and $M^2_t$,
as
happens if some fixed point component at $\lambda$ is a $2$-sphere  but also when there is an isolated fixed point at $\lambda$ with  Morse index $4$, and $U^1_t/S^1$ and $f^{-1}(U^2_t/S^1)$ 
represent different homology classes in $M^1_t$.
\end{noTitle}

We prove a variation of \cite[Lemma 3.13]{Go11}. 
We assume that reduced spaces below $\lambda$ are symplectic rational surfaces, as defined in \Cref{not:rationalmanifold}. In such symplectic manifolds, we have a characterization of exceptional classes using the theory of $J$-holomorphic curves, as in  \cite[Lemma 2.12 and Theorem 3.12]{KK17}.  
 We apply these results to show that $f \colon {\mu_1^{-1}(t)}/{S^1} \to {\mu_2^{-1}(t)}/{S^1}$ intertwines the sets of classes of the spheres that are sent to fixed points of Morse index $4$ at $\lambda$.
Further assuming that the fixed points at a non-extremal critical level are isolated, we deduce that $f$
can be isotoped through $\mu-S^1$-diffeomorphisms to a $\mu-S^1$-diffeomorphism that maps $U_t$ into $U'_t$, thus avoiding the problem described in \ref{nt:proof-problem}.\\
However, even if it is possible to extend $f$ as a $\mu-S^1$-diffeomorphism, it is not clear why it can be extended as an isomorphism. For that, the assumption that reduced spaces are symplectic rational surfaces also comes in handy: we will find an isomorphism $g$ between neighborhoods of the critical sets at $\lambda$ whose induced map on homology of the reduced spaces right below $\lambda$ agrees with the induced map of $f$. Then, we apply results on symplectic rational surfaces and the rigidity assumption on $M$ to conclude that these symplectomorphisms are isotopic through symplectomorphisms, which allows us to piece $g$ and $f$ together.

\begin{Notation}\label{not:rationalmanifold}
We consider the smooth manifold $\C \PP^2 \# k \overline{\C \PP^2}$ as the
manifold obtained from the complex projective plane $\C \PP^2$ by complex blowups at distinct points $q_1, \ldots , q_k$ in $\C \PP^2$. 
We have a
decomposition
$$H_2(\C \PP^2 \# k \overline{\C \PP^2};\Z) = \Z L \oplus \Z E_1 \oplus \cdots \oplus \Z E_k$$
where $L$ is the image of the homology class of a line $\C \PP^1$ in $\C \PP^2$ under the inclusion
map $H_2(\C \PP^2;\Z) \hookrightarrow H_2(\C \PP^2 \# k \overline{\C \PP^2};\Z)$ and $E_1, \ldots , E_k$ are the homology classes of the
exceptional divisors. 
A {\bf blowup form} on $\C \PP^2 \# k \overline{\C \PP^2}$ is a symplectic form for which there
exist pairwise disjoint embedded symplectic spheres in the classes 
$L, E_1, \ldots , E_k$.

If a symplectic manifold $(B,\omega)$ is symplectomorphic to $S^2\times S^2$ endowed with a positive multiply of the form $\omega^{\lambda}_{S^2 \times S^2}:=(1+\lambda)\omega_{\operatorname{SF}}\oplus \omega_{\operatorname{SF}}$ for $\lambda \geq 0$\footnote{By 
\cite{Gr85,LL95,LM96,Mc90,Ta95,Ta00}, every symplectic form on $S^2 \times S^2$  is of this form, up to rescaling.} or to some $\C \PP^2 \# k \overline{\C \PP^2}$ endowed with a blowup form, we say that $(B,\omega)$ is a \textbf{symplectic rational surface}. 
\end{Notation}
We relate isolated fixed points of Morse index $4$ and fixed spheres in the reduced space of an interior critical value $\lambda$ to spheres in the reduced space at a regular level below $\lambda$. For that, we use the \textit{Morse flow} $f_{\Morse}\colon M_{\lambda-r}\to M_{\lambda}$ induced from the flow of the normalized gradient vector field of the momentum map, defined in \S \ref{rem:criticalvalue}. We denote by $C'\subset M^i_{\lambda-r}$ the preimage of a fixed sphere $C\subset M^i_{\lambda}$ under $f^{i}_{\Morse}$. Also, we define $\mathcal{D}^i_{\sph}\subset H_2(M^i_{\lambda-r})$ to be the set \footnote{There will be no double count of classes due to our assumption that all fixed spheres are exceptional.} of homology classes corresponding to the set of spheres $C'\subset M^i_{\lambda-r}$.\\

We weaken the "same small fixed point data" (at a critical value $\lambda$) of $M^1$ and $M^2$ to the {\bf same $*$-small fixed point data} (at a critical value $\lambda$) of $M^1$ and $M^2$, as follows:
\begin{itemize}
    \item If $\lambda$ is extremal, we assume that the dimensions of the corresponding fixed point sets are the same.
    \item If $\lambda$ is not extremal, we require that there is a diffeomorphism $\eta_{\lambda}$ between the fixed point sets at level $\lambda$ that intertwines the index function (at level $\lambda$).
\end{itemize}
Moreover, we no longer assume that a non-extremal $\lambda$ is simple.
\begin{Setting}
\label{set:intro}

Let $(M,\omega,\mu)$ be a connected semi-free Hamiltonian manifold of dimension six whose momentum map $\mu$ is proper and with a bounded image, and $\lambda$ a critical value of $\mu$.
 We assume that
 \begin{itemize}
    \item for all $t$  below or equal to $\lambda$, the reduced space $(M_t,\omega_t)$ is a symplectic rational surface
    whenever it is of dimension four;

    \item   for any interval $I=[t_0,t_1]$ of regular values below $\lambda$,
    the family $(M_{t_0},\omega_{t\in I})$ is rigid.
 \end{itemize}

\end{Setting}
\begin{theorem}\label{thm:extending-g}
    For $i=1,2$, let $M^i=(M^i,\omega^i,\mu_i)$ and $\lambda$
 be as in 
    \Cref{set:intro}. Assume that $M^1$ and $M^2$ have the same $*$-small fixed point data at the critical value $\lambda$, and that $\lambda$ is the only critical value of $\mu_i$ for $i=1,2$.
    Assume that
 \begin{itemize}       
     \item[(i)] if $\lambda$ is {non-extremal}, then 
there are only isolated fixed points and exceptional spheres in $M^i_{\lambda}$.
\item[(ii)] if $\lambda$ is maximal and the fixed point set  at $\lambda$ in $M^i$ is of $\dim<4$, then it is either a point or a sphere.
    \end{itemize}
  Let $r>0$ be such that there is no critical value in $[\lambda-r,\lambda)$ with respect to both $\mu_1$ and $\mu_2$. Consider an isomorphism 
 $$f \colon \mu_1^{-1}((-\infty,\lambda-r]) \to \mu_2^{-1}((-\infty,\lambda-r])$$
 such that $f_{\lambda-r}$ sends $\mathcal{D}^1_{\sph}$ bijectively into $\mathcal{D}^2_{\sph}$.\\
 Then there is $\eps'>0$ that can be chosen arbitrarily small, such that $f$ restricted to $\mu_1^{-1}((-\infty,\lambda-(\eps'+r)])$ extends over the level $\lambda$ as an isomorphism, meaning that there is $\delta>0$ and an isomorphism
 $$h\colon \mu_1^{-1}((-\infty,\lambda+\delta)) \to \mu_2^{-1}((-\infty,\lambda+\delta))$$
 such that $h=f$ on $\mu_1^{-1}((-\infty,\lambda-(r+\eps')])$.\\
 The statement is true if we replace $-\infty$ everywhere with $\alpha$ for any $\alpha<\lambda-r$.
\end{theorem}

 Note that in our counter example, \Cref{ex:openDpoly}, all the assumptions but (i) hold.\\
 
We deduce a variation of \cite[Theorem 1.6]{Go11}.

\begin{theorem}\label{thm:mainresult}
 Let $M^1=(M^1,\omega^1,\mu_1)$ and $M^2=(M^2,\omega^2,\mu_2)$ 
 be compact, simply-connected semi-free Hamiltonian manifolds of dimension six. Assume that $M^1$ and $M^2$ have the same $*$-small fixed point data, and that for every critical value $\lambda$ and $i=1,2$, $M^i$ and $\lambda$ are as in \Cref{set:intro}.
  Suppose that one of the following is true. 
 \begin{itemize}
    \item[(i)] $M^1$ and $M^2$ contain non-extremal fixed surfaces,  these surfaces are all mapped to the same $\lambda_S$ by $\mu_1$ and $\mu_2$, and $M^1$ and $M^2$ have the same fixed point data at $\lambda_S$. In that case, assume that $\lambda_S$ is simple. 
   \item[(ii)] All non-extremal fixed points of $M^1$ and $M^2$ are isolated, and $M^1$ and $M^2$ have the same fixed point data at some non-extremal critical value $\lambda_S$. In that case, $\lambda_S$ does not have to be simple.
   \item[(iii)] All non-extremal fixed points of $M^1$ and $M^2$ are isolated, and neighborhoods of the minima of $M^1$ and $M^2$ are equivariantly symplectomorphic. In that case, we call the minimal level $\lambda_S$.
   \end{itemize}
   Then $M^1$ and $M^2$ are isomorphic.
\end{theorem}
The additional assumptions in \Cref{thm:mainresult}, compared to \cite[Theorem 1.6]{Go11}, are that reduced spaces are symplectic rational surfaces  and the restriction on interior fixed surfaces.  However, we dropped the assumption of \cite[Theorem 1.6]{Go11} that the fixed point sets at each critical
level are either surfaces or isolated fixed points; we allow the extrema to be four-dimensional. We note that in our counter example, \Cref{examp:main}, 
all the theorem's assumptions except for the restriction on the interior fixed surfaces hold. 
 \begin{proof}[Proof of \Cref{thm:mainresult}, assuming 
  \Cref{thm:extending-g}]
In either case (i) or (ii) or (iii), we find $\eps>0$ and an isomorphism $$f \colon \mu_1^{-1}((\lambda_S-\eps,\lambda_S+\eps)) \to \mu_2^{-1}((\lambda_S-\eps,\lambda_S+\eps))$$
of open sets in $(M^1,\omega^1,\mu_1)$ and $(M^2,\omega^2,\mu_2)$. 
This is clear by assumption in case (iii), is by \cite[Lemma 3.4]{Mc09} in case (ii), and by \cite[Theorem 13.1]{GS89} whenever $\lambda_S$ is simple.\\
    
    In any case, it is enough to extend the isomorphism iteratively over the critical values of $\mu_i$ above $\lambda_S$ to an isomorphism
    \[
    \mu_1^{-1}((\lambda_S-\eps,\infty))\to \mu_2^{-1}((\lambda_S-\eps,\infty)).
    \]
    This is since extending the isomorphism over the critical values of $\mu_i$ below $\lambda_S$, if they exist, amounts to extending $f$ as an isomorphism 
    $$(-\mu_1)^{-1}((-\lambda_S-\eps,-\lambda_S+\eps)) \to (-\mu_2)^{-1}((-\lambda_S-\eps,-\lambda_S+\eps))$$
    of open sets in $(M^1,-\omega^1,-\mu_1)$ and $(M^2,-\omega^2,-\mu_2)$
    over the critical values of $-\mu_i$ above $-\lambda_S$ to an isomorphism 
   \[
    (-\mu_1)^{-1}((-\lambda_S-\eps,\infty))\to (-\mu_2)^{-1}((-\lambda_S-\eps,\infty)).
    \]
    
  So let $\lambda>\lambda_S$ be a critical value such that there is no critical value in $(\lambda_S,\lambda)$. By 
    \Cref{thm:extending-g},
    we find $\delta>0$ and an isomorphism
    \[
     \mu_1^{-1}((\lambda_S-\eps,\lambda+\delta))\to \mu_2^{-1}((\lambda_S-\eps,\lambda+\delta))
    \]
    extending the isomorphism $f$.
    If $\lambda$ is maximal, we are done. If $\lambda$ is not maximal, we let $\lambda'$ be a critical value such that there is no critical value in $(\lambda,\lambda')$ and apply \Cref{thm:extending-g} to extend the isomorphism to $\mu_1^{-1}((\lambda_S-\eps,\lambda'+\delta'))\to \mu_2^{-1}((\lambda_S-\eps,\lambda'+\delta'))$ with $\delta'>0$. We repeat this argument till we reach the maximal value. 
    \end{proof}

 \subsection{Application to Cho's classification of positive monotone semi-free Hamiltonian $S^1$-manifolds}
 Gonzales' results were pivotal in Cho's classification 
 of six-dimensional positive monotone symplectic manifolds admitting semi-free Hamiltonian circle actions. 
 A compact symplectic manifold $(M,\omega)$ is called \textbf{positive monotone} if the cohomology class $[\omega]$ is a multiple by a positive real number of the first Chern class of $TM$ with respect to an almost complex structure compatible with $\omega$.
A positive monotone symplectic manifold is the symplectic analogue of a \textbf{Fano manifold}: a compact complex manifold $X$ whose anticanonical line bundle $K_X^{-1}$ is ample.
The ampleness of $K_X^{-1}$ means that there is a holomorphic embedding
$ i \colon X \hookrightarrow \mathbb{C}P^N $  such that $(K_X^{-1})^k=i^* \mathcal{O} (1)$ for some $N>0$ and $k>0$.
The almost complex structure induced 
by the complex analytic atlas on $X$ is compatible with the symplectic form $i^*(\omega_{FS})$.
Since $c_1(TX)=c_1(K_X^{-1})$ and $c_1(i^* \mathcal{O} (1))=[i^*(\omega_{FS})]$,  the symplectic manifold $(X, i^*(\omega_{FS}))$ is positive monotone. 
 
In dimensions two and four, it was proven that a  positive monotone symplectic manifold is symplectomorphic to a Fano manifold (with a positive multiple of $i^{*}\omega_{FS}$) \cite{Gr85, Mc90, Ta00}.
In dimension greater than or equal to twelve, there are examples of positive monotone symplectic manifolds that are not simply connected \cite{FP10}. 
Since Fano manifolds are simply connected \cite[Corollary 6.2.18]{IP99}, 
these examples are not even homotopy equivalent to Fano manifolds. 
In dimensions six, eight, and ten, it is not known if any positive monotone symplectic manifold is diffeomorphic, or homotopy equivalent, to a Fano manifold.
However, if the {\bf complexity} $\frac{1}{2} \dim M-\dim T$ of a positive monotone Hamiltonian $T$-space is $0$, then it
is equivariantly symplectomorphic to a Fano manifold with a holomorphic torus action, as follows from \cite{De88}.
In higher complexity the question is still open.

\begin{conjecture}Fine-Panov 2015 \cite{FP15}.
Let $(M, \omega)$ be a positive monotone symplectic manifold of dimension six that admits a Hamiltonian circle action.
Then $M$ is diffeomorphic to a Fano manifold.
\end{conjecture}

  In \cite{Ch19}, \cite{Ch21.1} and \cite{Ch21.2}, Cho classified the so-called 'topological fixed point data' of {positive monotone} semi-free Hamiltonian $S^1$-spaces of dimension six. Cho showed that these data determine the fixed point data as in \Cref{def:fixedpointdata}, and then used  \cite[Theorem 1.5]{Go11} in order to conclude that all of these spaces  are isomorphic as Hamiltonian $S^1$-manifolds to Fano manifolds with holomorphic $S^1$-actions. 
   The fact that \cite[Theorem 1.5]{Go11} is not correct, as we show in \Cref{examp:main}, raises question on the validity of Cho's result.
    Nevertheless, we deduce from \Cref{thm:mainresult} that in the examples that occur in Cho's classification, the fixed point data \textbf{do} determine the isomorphism type of the underlying Hamiltonian $S^1$-manifold. \\

The key is the implication of the positive monotone assumption on the distribution of fixed point components.
We assume that the positive monotone symplectic manifold 
is \textbf{normalized}, i.e., $[\omega]=c_1(TM)$.
For a normalized positive monotone Hamiltonian $S^1$-space,  there is a momentum map
$\mu \colon M \rightarrow \R$ such that 
$$\mu(p)=-(\alpha^1_p+\alpha^2_p+\alpha^3_p)$$ for any fixed point $p$, where the $\alpha^i_p$ 
are the weights of the $T$-representation 
on the tangent space $T_pM$ \cite[Proposition 3.5]{CSS23}. 
Combining with the local normal form for a semi-free Hamiltonian $S^1$-action in dimension six, recalled in \Cref{nt:local6}, we conclude the following lemma.
    \begin{lemma} \label{lem:keym}
        Let $M$ be a normalized positive monotone semi-free Hamiltonian $S^1$-manifold of dimension six.
        Then \begin{itemize}
        \item all non-extremal fixed point components of dimension zero are located at level $\pm 1$;
        \item there are at most two critical values larger than $0$;
        \item all non-extremal fixed point components of dimension two are located at level $0$ and have the same index.
    \end{itemize}
    \end{lemma}

 We will also use the fact that rigidity holds for a big family of symplectic manifolds in dimension four. The next theorem is a collection of many results, see (\cite{Gr85}, \cite{AM00}, \cite{LP04}, \cite{Pin08}, \cite{Ev11},  \cite{LLW15}).
    \begin{theorem}\label{thm:rigid}
        Let $(B,\{\omega_t\}_{t \in I})$ be $S^2\times S^2$ or a $k$-fold blowup of $\C \PP^2$ with $0\leq k\leq 4$, endowed with a family of symplectic forms smoothly parametrized by $t$. Then $(B,\{\omega_t\}_{t \in I})$ is rigid.
    \end{theorem}
    We deduce that Cho's theorem \cite[Theorem 1.2]{Ch19} still holds. 
    
\begin{theorem} \cite[Theorem 1.2]{Ch19}.
    Any  positive monotone, compact,  connected six-dimensional symplectic manifold $M$ with a semi-free Hamiltonian $S^1$-action is equivariantly symplectomorphic to a Fano manifold $M'$ with a positive multiple of $i^{*}\omega_{FS}$ and a Hamiltonian $S^1$-action induced from a holomorphic $\C^*$-action.
\end{theorem}
\begin{proof}
    By \cite[Section 6,7,8]{Ch19}, given a positive monotone, compact, connected six-dimensional symplectic manifold $M$ with a semi-free Hamiltonian $S^1$-action, there is a Fano manifold $M'$ with a positive multiple of $i^{*}\omega_{FS}$ and a Hamiltonian $S^1$-action induced from a holomorphic $\C^{*}$-action such that $M$ and $M'$  have the same {topological fixed point data}, as defined in \cite[Definition 5.7]{Ch19}. Moreover, the topological fixed point data already determine the fixed point data (see \cite[Lemma 9.7]{Ch19} and the proof of \cite[Theorem 1.2]{Ch19}). If, in Cho's notation, $M$ is of type \textbf{(II-1-4.$k$)} with $k\geq 5$, then the proof of \cite[Theorem 1.2]{Ch19} shows that $M$ and $M'$ are isomorphic without using the result of Gonzales.

    If $M$ is not of type \textbf{(II-1-4.$k$)} with $k\geq 5$, then each four-dimensional reduced space of $M$ is either $S^2\times S^2$ or an $m$-fold blowup of $\C \PP^2$ with $0\leq m\leq 4$ (see the statement right after \cite[Theorem 9.1]{Ch19}). It remains to check that $M$ and $M'$ satisfy the assumptions needed to apply \Cref{thm:mainresult}, since then $M$ and $M'$ would be equivariantly symplectomorphic.
    Indeed, that the rigidity assumption on $M$ (or $M'$) holds follows from \Cref{thm:rigid}; the fact that all fixed surfaces not corresponding to an extremal critical value are mapped to the same value $\lambda_S$ is by \Cref{lem:keym}. 
\end{proof}

{\bf The structure of the paper.} 
In \Cref{sec:examples}, we construct two counter examples showing that \cite[Theorem 1.5]{Go11} and \cite[Lemma 3.13]{Go11} are incorrect as stated. We then proceed towards proving \Cref{thm:extending-g}.  In \Cref{sec:pre} we describe the effect on the smooth and symplectic structures on the reduced spaces when 'flowing into' a critical level by the map induced from the gradient flow of the momentum map. In \Cref{sec:almostsymplectic} we describe piecing  together of isomorphisms, and, for that, introduce the notion of almost symplectic $\mu-S^1$-diffeomorphism. In \Cref{sec:neighborhoods} we establish the implications
of having the same $*$-small fixed point data on isomorphisms of neighborhoods of critical
levels, assuming that reduced spaces of dimension four are symplectic rational surfaces. Finally, in \Cref{sec:extend} we prove the theorem.
In general, we try to avoid repetition of Gonzales' arguments and proofs, but we fill in details when they are required. 

\subsection*{Funding} 
This work was supported by the National Science Foundation and the Binational Science Foundation [grant number 2021730].
\subsection*{Acknowledgements} 
We were motivated by Sue Tolman's objection to the assertion that the global isomorphism type of a semi-free Hamiltonian $S^1$-manifold is determined  by the local data; we are grateful to Sue for her guidance towards a counter example. We thank Martin Pinsonnault for answering our questions about configurations of exceptional spheres in symplectic rational surfaces, and Yael Karshon and Isabelle Charton for helpful discussions. 
\section{Non-isomorphic manifolds with the same fixed point data}\label{sec:examples}

We give an example of non-isomorphic closed semi-free Hamiltonian $S^1$-manifolds with the same fixed and local data. We further show that two semi-free Hamiltonian six-dimensional $S^1$-manifolds that have the same small fixed point data at a critical level and are isomorphic below it might not be isomorphic near that level. In both examples, the rigidity assumption is satisfied.

\subsection*{An example of non-isomorphic closed semi-free Hamiltonian $S^1$-manifolds with the same fixed and local data} \label{subsec:counterexample}

\begin{example} \label{examp:main}
Let $\eps>0$ and consider the standard product action of $T^3=(S^1)^3$ on $$M:=S^2\times S^2\times S^2 \text{ with }\omega^M:=((1+\eps)\omega, \omega, \omega),$$
where $\omega$ is the Fubini-Study form on $S^2$. Consider the diagonal embedding
$S^1\hookrightarrow S^1\times S^1 \times \{e\}$ and its induced (semi-free!) action $\rho_{S^1}$ on $M$ with momentum map $\mu$. The fixed point components of this action are the four spheres $\{a\} \times \{b\} \times S^2$ for $a,b \in \{\southpole,\northpole\}$.

One can think of $S^1 \acts (M,\omega^M)$ as the product of $(S^2,\omega)$ endowed with the trivial action and $$N:=(S^2\times S^2,((1+\eps)\omega, \omega)$$  endowed with the diagonal action $\rho_{S^1}$
and momentum map $\mu_N$. 
Since $\eps$ is positive, the fixed point components of the $S^1$-action on $N$ are mapped to different values under $\mu_N$. We denote by $\lambda$ the lowest interior critical value and by $\lambda'$ the other one. The decorated graph of $\rho_{S^1}$ on $N$ is given on the left of \Cref{fig:1}. 
We will also view $N$ as a symplectic toric manifold with the canonical $S^1\times S^1$-action $\rho_1$ with momentum map $\mu_1$ whose image is given on the right of  \Cref{fig:1}. The momentum map $\mu_N$ of the $S^1$-action on $N$ is obtained by composing the projection on $(1+\eps)y=x$ with the momentum map $\mu_1$ of the $S^1 \times S^1$-action. 
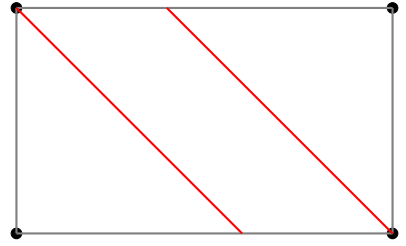
\begin{figure}[h]
	\begin{center}
 \begin{tikzpicture}
 \hspace{-1cm}
                \filldraw[black] (0,0) circle (2pt);
			\filldraw[black] (0,1.3) circle (2pt);
			\filldraw[black] (0,1.7) circle (2pt);
			\filldraw[black] (0,3) circle (2pt);
   \node (m) at (2.6,0) {$\mu_N(\southpole,\southpole)$};
   \node (t) at (3,1.3) {$\lambda{=\mu_N(\southpole,\northpole)}$};
   \node (t') at (3,1.7) {$\lambda'{=\mu_N(\northpole,\southpole)}$};
    \node (M) at (2.6,3) {$\mu_N(\northpole,\northpole)$};
 \end{tikzpicture}
		\begin{tikzpicture}
  \hspace{2.5cm}
			\filldraw[black] (0,0) circle (2pt);
			\filldraw[black] (5,0) circle (2pt);
			\filldraw[black] (5,3) circle (2pt);
			\filldraw[black] (0,3) circle (2pt);
			\draw[gray, thick] (0,0) -- (5,0);
			\draw[gray, thick] (0,0) -- (0,3);
			\draw[gray, thick] (0,3) -- (5,3);
			\draw[gray, thick] (5,0) -- (5,3);
   \draw[red, thick] (0,3) -- (3,0);
   \draw[red, thick] (2,3) -- (5,0);
		\end{tikzpicture}
	\end{center}
 \caption{On the left: the images of the fixed points of $N$ under $\mu_N$.\\
 On the right: the toric momentum image of $N$, where the red lines represent
 the level sets $\lambda$ and $\lambda'$ of $\mu_N$.\label{fig:1}}
 \end{figure}
 
We will construct a new semi-free Hamiltonian $S^1$-manifold $M'$ by gluing $\mu^{-1}((-\infty,\lambda'))$ and $\mu^{-1}((\lambda,\infty)$ along $\mu^{-1}((\lambda,\lambda'))$ by a gluing map that is an equivariant symplectomorphism on $\mu^{-1}((\lambda,\lambda'))$ and is not the identity. To define the gluing map, we will first reinterpret $\mu^{-1}((\lambda,\lambda'))$ and the symplectic form and $S^1$-action on it. 

For that, consider $N':=\mu_N^{-1}((\lambda,\lambda'))$ with the symplectic form and $S^1 \times S^1$-action induced from the symplectic form and action $\rho_1$ on $N$. The open symplectic toric manifold $N'$ is the $\mu_1$-preimage of the interior of the red lines in \Cref{fig:1}.
Using the automorphism of $T^2$ given by
\[
A=\begin{pmatrix}
1 & 0 \\
1 & 1 
\end{pmatrix},
\]
we define another $T^2$-action $\rho_2$ on $N'$ by $\rho_2:=\rho_1\circ A$. See \Cref{fig:2} for the effect of composing with $A$ on the momentum map image.
\begin{figure}[h]
	\begin{center}
		\begin{tikzpicture}
  \hspace{1cm}
			\filldraw[black] (5,0) circle (2pt);
			\filldraw[black] (0,3) circle (2pt);
			\draw[gray, thick] (1.5,0) -- (5,0);
			\draw[gray, thick] (0,1.5) -- (0,3);
			\draw[gray, thick] (0,3) -- (3.5,3);
			\draw[gray, thick] (5,0) -- (5,1.5);
   \draw[red, thick] (0,3) -- (3,0);
   \draw[red, thick] (2,3) -- (5,0);
   \draw[->, thick] (6,1.5) -- (8,1.5);
   \hspace{-2.5cm}
			\filldraw[black] (13,3) circle (2pt);
			\filldraw[black] (15,0) circle (2pt);
			\draw[gray, thick] (11.5,0) -- (15,0);
			\draw[gray, thick] (15,0) -- (16.5,1.5);
			\draw[gray, thick] (16.5,3) -- (13,3);
			\draw[gray, thick] (11.5,1.5) -- (13,3);
   \draw[red, thick] (13,3) -- (13,0);
   \draw[red, thick] (15,3) -- (15,0);
		\end{tikzpicture}
	\end{center}
 \caption{\label{fig:2} On the left, the momentum polytope of the con-compact symplectic toric manifold we started with, on the right the momentum polytope of the same manifold after twisting the action.}
 \end{figure}
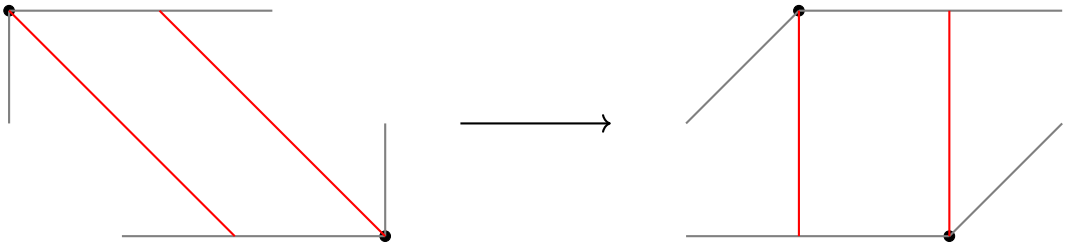
 
By the classification of non-compact symplectic toric manifolds in \cite[Theorem 1.3.2]{KL15},
$\rho_2$ on $N'$ is $T^2$-equivariantly symplectomorphic to $((\lambda,\lambda')\times S^1)\times S^2$ endowed with the standard $T^2$-action, and with the standard symplectic form
$(dt\wedge d\varphi,\omega)$, since the two have the same momentum map image.  The circle action of $S^1\times \{e\}$ on $N'$ induced by $\rho_2$ acts on the $(\lambda,\lambda')\times S^1$-factor only. Note that it corresponds to the action $\rho_{S^1}$ on $N'$.

As a consequence, the open symplectic toric manifold $$U:=\mu^{-1}((\lambda,\lambda'))=N'\times S^2$$  is, up to automorphism of $T^3$,
$T^2$-equivariantly symplectomorphic to $((\lambda,\lambda')\times S^1)\times S^2\times S^2)$ endowed with the 
standard $T^3$-action, and with the standard symplectic form $(dt\wedge d\varphi,\omega, \omega)$. Under that identification, the $S^1$-action $\rho_{S^1}$ is the standard action on the first factor. 
An ordinary symplectomorphism $f$ on $(S^2\times S^2,(\omega,\omega))$ induces naturally a symplectomorphism on $U$ that is equivariant w.r.t.\ the latter $S^1$-action and preserves the momentum map, and therefore also a symplectomorphism that is equivariant w.r.t.\ the original $S^1$-action and preserves $\mu$. We call the latter map $f_{U}$.

Now consider the two sets $M_-:=\mu^{-1}((-\infty,\lambda'))$ and $M_+:=\mu^{-1}((\lambda,\infty))$
endowed with the symplectic forms  ${\omega^M}|_{M_-}$ and ${\omega^M}|_{M_{+}}$, resp.\
Their union is equal to $M$, and their overlap is precisely $U$. Take the
symplectomorphism $g^{\swap}$  of $(S^2\times S^2,(\omega,\omega))$ given by $(p_1,p_2)\mapsto (p_2,p_1)$. Define
\[
M':= M_-\cup_{(g^{\swap}_U,U)} M_+,
\]
 that is 
$$M_{-} \sqcup M_{+}/\sim \text{ where }x \sim y \text{ iff }g^{\swap}_{U}(x)=y.$$
The tangent bundle of $TM'$ is then given by $$TM_{-} \sqcup TM_{+}/\sim \text{ where }(x,v) \sim (y,w) \text{ iff }g^{\swap}_{U}(x)=y \text{ and }dg^{\swap}_{U}(v)=w.$$
Endow $M'$ with the gluing form $\omega_{M'}$, that is, the unique form whose restrictions to $M_-$ and $M_+$ are  ${\omega^M}|_{M_-}$ and ${\omega^M}|_{M_{+}}$. The form $\omega_{M'}$ is well defined since $${\omega^M}_{x}(v_1,v_2)={\omega^M}_{g^{\swap}_{U}(x)}(dg^{\swap}_{U}(v_1),dg^{\swap}_{U}(v_2)).$$ It is also symplectic. The $S^1$-action $\rho_{S^1}$ on $M$ induces an $S^1$-action on $M'$, well defined since $g^{\swap}_{U}$ is equivariant. The obtained action is again semi-free; the set of fixed points of $M$ does not intersect the overlap $U$ of $M_{-}$ and $M_{+}$. Moreover, since $g^{\swap}_{U}$ preserves the momentum map, the map $\mu'([x])=\mu(x)$ is well defined; it is a momentum map for the $S^1$-action on $M'$.
\end{example}

\begin{remark}\label{rem:homology}
As in \Cref{examp:main}, consider the symplectic toric manifold $N=(S^2\times S^2,((1+\eps)\omega, \omega))$ with the canonical $S^1 \times S^1$-action and momentum map $\mu_1$ whose image is the rectangle in \Cref{fig:2.3}.
The symplectic manifold $N$ is also endowed with the diagonal $S^1$-action and its momentum map $\mu_N$.
As before, we identify the open symplectic toric manifold $N':=\mu_N^{-1}((\lambda,\lambda'))$ with 
$((\lambda,\lambda')\times S^1)\times S^2$ with the standard symplectic form and $T^2$-action. This allows us to consider any slice $\{s\}\times S^1\times S^2$ as a subset of $N$ via the inclusion $N'\hookrightarrow N$: its image under $\mu_1$ is the red line in \Cref{fig:2.3}. Such a slice can be seen as the boundary of tubular neighborhoods 
$$N_-:=\mu_N^{-1}((-\infty,s]) \text{ and }N_+:=\mu_N^{-1}([s,\infty))$$ of the spheres $\{\southpole\}\times S^2$ and $\{\northpole\}\times S^2$ corresponding to the the left-most and right-most vertical lines in \Cref{fig:2.3}.
Both $N_-$ and $N_+$ are diffeomorphic to $S^2\times D^2$. Indeed, we can view $N_{\pm}$ as the standard tubular neighborhood (the green area in \Cref{fig:2.3}) with the $D^2$-fiber rescaled according to the height of the base point in $S^2$.

   We deduce results on the maps on $H_{2}$ coming from the inclusions $N' \hookrightarrow N$ and $N_{\pm}\hookrightarrow N$. 
    \begin{itemize}
        \item The map $H_2(N')\to H_2(N_{\pm})$ is an isomorphism.
        \item The map $H_2(N_{\pm})\to H_2(N)$ is injective.
    \end{itemize}
    Since, for the momentum map $\mu$ of the $S^1$-action $\rho_{S^1}$ on $M$ and any subset $O$ of $\R$, we have $\mu^{-1}(O)=\mu_N^{-1}(O)\times S^2$, it holds that $U=\mu^{-1}((\lambda,\lambda'))=N'\times S^2$ and $M_{\pm}=N_{\pm}\times S^2$. Therefore for the maps on $H_{2}$ coming from the inclusions $M_{\pm}\hookrightarrow M$ and $U \hookrightarrow M_{\pm}$,
    \begin{itemize}
        \item the map $\iota_*^{\pm}\colon H_2(U)\to H_2(M_{\pm})$ 
        is an isomorphism,
        \item The map $H_2(U)\to H_2(M)$ is injective.
    \end{itemize}
    The Mayer Vietoris sequence
    $$\hdots \to H_2(U)\to H_2(M_-)\oplus H_2(M_+)\to H_2(M')\to \hdots $$
    then gives that $H_2(M_+)\to H_2(M')$ is injective, too, because both maps $H_2(U)\to H_2(M_{\pm})$ are isomorphisms. Note, however, that only one of the maps $H_2(U)\to H_2(M_{\pm})$ in the above sequence is given by $\iota_*^{\pm}$; the other one is given by $\iota_*^{\pm}$ precomposed with $(g_U^{\swap})_*\colon H_*(U)\to H_*(U)$.
\end{remark}

\begin{figure}[h]
	\begin{center}
		\begin{tikzpicture}
  \draw [fill=green] (0,0) rectangle (2.5,3);
			\filldraw[black] (0,0) circle (2pt);
			\filldraw[black] (5,0) circle (2pt);
			\filldraw[black] (5,3) circle (2pt);
			\filldraw[black] (0,3) circle (2pt);
			\draw[gray, thick] (0,0) -- (5,0);
			\draw[gray, thick] (0,0) -- (0,3);
			\draw[gray, thick] (0,3) -- (5,3);
			\draw[gray, thick] (5,0) -- (5,3);
   \draw[red, thick] (1,3) -- (4,0);
		\end{tikzpicture}
	\end{center}
 \caption{\label{fig:2.3}}
 \end{figure}

We will show that $(S^1 \acts M,\mu)$ and $(S^1 \acts M',\mu')$ are not \textbf{$\mu-S^1$-diffeomorphic}, meaning that there is no equivariant diffeomorphism between the spaces that respects the momentum maps. In particular, $(S^1 \acts M,\omega^M,\mu)$ and $(S^1 \acts M',\omega_{M'},\mu')$ can not be isomorphic as Hamiltonian $S^1$-manifolds. 
The main point is that the two fixed spheres at level $\lambda$ and $\lambda'$ in $M$ represent the same homology in $H_2(M)$, whereas they do not in $M'$.

\begin{lemma}\label{lem:counterexample}
    $M=(S^1 \acts M,\omega^M,\mu)$ and $M'=(S^1 \acts M',\omega_{M'},\mu')$ are not equivariantly symplectomorphic.
    In fact, they are not even $\mu-S^1$-diffeomorphic.
\end{lemma}
\begin{proof}
    Denote by $S^2_{+}$ the unique interior  $\rho_{S^1}$-fixed $S^2$ at level set $\lambda$ and by $S^2_{-}$ the unique interior  $\rho_{S^1}$-fixed $S^2$ at level set $\lambda'$. Any 
    $\mu-S^1$-diffeomorphism between $M$ and $M'$ would  send the sphere $S^{2}_{\pm}$ in $M$ 
    to the sphere $S^{2}_{\pm}$ in $M'$, respectively.
    Clearly, both the $S^2_{\pm}$ in $M$ share the same homology class. We show that this is not true for the $S^2_{\pm}$ in $M'$.
 
    From now on, we identify $H_2(M_{\pm})$ with $H_2(U)$ using the embeddings
    $U\hookrightarrow M_{\pm}$ and the induced isomorphism on the second homology groups (see \Cref{rem:homology}). Under this identification, the class of $S^{2}_{\pm}$ in $H_2(M_{\pm})$ corresponds to the class
    $$[\{\pt\}\times \{\pt\}\times S^2]\in H_2(((\lambda,\lambda')\times S^1)\times S^2\times S^2)=H_2(U).$$
    The gluing map $g^{\swap}_U$ used to construct $M'$ sends the class
    $[\{\pt\}\times S^2\times \{\pt\}]\in H_2(U)$
    to the class
    $[\{\pt\}\times \{\pt\} \times S^2]\in H_2(U)$ (and the other way around). Therefore, the class 
$[\{\pt\}\times S^2\times \{\pt\}]\in H_2(M_-)$
    is sent by the map 
    $$H_2(M_-)= H_2(U)\overset{(g^{\swap}_U)_*}{\to} H_2(U)= H_2(M_+)$$
    to the class of $S^{2}_+$ in $H_2(M_+)$. Thus, by the Mayer Vietoris sequence
    $$\hdots \to H_2(U) \overset{k}{\to} H_2(M_-)\oplus H_2(M_+)\to H_2(M')\to \hdots $$
    (where $k=\iota_*^{-} \; \oplus \; \iota_*^{+}\circ (g^{\swap}_U)_*$,  see \Cref{rem:homology}),
    the classes of $S^2_-$ and $S^2_+$ in $H_2(M')$ are in the image of $H_2(M_-)\to H_2(M')$, and their preimages can be chosen to be different, namely $[\{\pt\}\times S^2\times \{\pt\}]$ for $S^2_+$ and $[\{\pt\}\times \{\pt\} \times S^2]$ for $S^2_-$. 
    Since $H_2(M_-)\to H_2(M')$ is injective (see \Cref{rem:homology}), it follows that the classes of $S^2_-$ and $S^2_+$ in $H_2(M')$ are different as well.
\end{proof}
However, the manifolds have the same fixed point and local data.
\begin{lemma}\label{lem:mm'local}
 $M$ and $M'$ have the same fixed point data. Moreover, 
 $M$ and $M'$  have the same local data (see \Cref{def:localdata}).
    \end{lemma}
    \begin{proof}
    It is clear that $M$ and $M'$ have the same critical levels, and that the fixed point data of $M$, for example, at any critical level $\tilde{\lambda}$ is determined by the isomorphism type of $\mu^{-1}((\tilde{\lambda}-\eps,\tilde{\lambda}+\eps))$ for any $\eps>0$. Since $M'$ was obtained by gluing $M_-$ and $M_+$ together along $\mu^{-1}((\lambda,\lambda'))$, it is immediate that for each critical level $\tilde{\lambda}$ there is $\eps>0$ such that $\mu^{-1}((\tilde{\lambda}-\eps,\tilde{\lambda}+\eps))$ and $(\mu')^{-1}((\tilde{\lambda}-\eps,\tilde{\lambda}+\eps))$ are isomorphic. So $M$ and $M'$ have the same fixed point data.

     Let us now show that they have the same local data. Again, the way we obtained $M'$ from $M$ makes it clear that we only have to check that the gluing map 
$$g^{\swap}_U \colon (\mu^{-1}(\lambda'-\delta,\lambda'),\omega^M,\mu) \to (\mu^{-1}(\lambda'-\delta,\lambda'),\omega^M,\mu)$$
between the overlap of $U=\mu^{-1}(\lambda,\lambda')$
and $\mu^{-1}(\lambda'-\delta,\lambda'+\delta)$  belongs to the same gluing class as the identity
$$\id \colon (\mu^{-1}(\lambda'-\delta,\lambda'),\omega^M,\mu) \to (\mu^{-1}(\lambda'-\delta,\lambda'),\omega^M,\mu).$$
That is, we need to find isomorphisms
        \[
        f\colon \mu^{-1}(\lambda'-\delta,\lambda'+\delta) \to \mu^{-1}(\lambda'-\delta,\lambda'+\delta), \quad g\colon U\to U
        \]
        such that $g=g^{\swap}_U\circ f$ on $\mu^{-1}(\lambda'-\eps,\lambda')$. We may choose $f$ to be the identity and $g$ to be $g^{\swap}_U$,
        since $g^{\swap}_U$ was initially defined on $U$. 
    \end{proof}

    By \Cref{lem:mm'local} and \Cref{lem:counterexample}, the closed semi-free Hamiltonian $S^1$-manifolds $M$ and $M'$ constructed in \Cref{examp:main} have the same local data but are not isomorphic. This contradicts \cite[Theorem 2.6]{Go11}.
  Furthermore, \Cref{examp:main} is a counter example to the following assertion.

    \begin{assertion*}\label{thm:fixedpointdata}
    \cite[Theorem 1.5]{Go11}.
 Let $M^1$ and $M^2$ be compact, connected semi-free Hamiltonian $S^1$-manifolds of dimension six.
   Assume that for each non-extremal critical value $\lambda$ of $M^1$ and $M^2$ the corresponding fixed point components have the same index.
        Assume that $M^1$ and $M^2$ have the same fixed point data, and that for any two consecutive critical values $\lambda$ and $\lambda'$, any closed interval $I\subset (\lambda,\lambda')$ and any $t'\in I$, the pair $(M^1_{t'},\{\omega_t\}_{t \in I})$ is rigid. Then $M^1$ and $M^2$ are equivariantly symplectomorphic.
    \end{assertion*}
  
Indeed,  there is only one fixed point component  at each critical level, and, by  \Cref{lem:mm'local}, $M$ and $M'$ have the same fixed point data.  Moreover, since every reduced space at a regular level is diffeomorphic to $S^2 \times S^2$, \Cref{thm:rigid} implies that the rigidity assumption of \cite[Theorem 1.5]{Go11} holds for the  manifolds $M$ and $M'$ as well. 
    However, by \Cref{lem:counterexample}, the manifolds are not equivariantly symplectomorphic.

\begin{remark}\label{rem:proofthm2.6}
In the proof of \cite[Theorem 2.6]{Go11},
it is argued that the uniqueness of the isomorphism type of  a Hamiltonian $S^1$-manifold obtained from the set of local data of $M$
follows in the same way as the uniqueness of the isomorphism type of $Y \cup_{(\phi,\delta)} Z$ was deduced 
in \cite[Lemma 2.5]{Go11}. This argument is not taking into account that the situation changes when more than one critical level occurs.
Assume, for the sake of simplicity, that $(M,\omega,\mu)$ has exactly two critical levels,
at $0$ and at $1$. Let $Y_0=\mu^{-1}(0-\delta_0,0+\delta_0)$ and $Y_1=\mu^{-1}(1-\delta_1,1+\delta_1)$ be the cobordisms around these critical levels.
Let $Z=\mu^{-1}(0,1)$
be the regular slice corresponding to $(0, 1)$. Denote by $\phi_0(\colon Y_0 \to Z)$ the gluing map
corresponding to $Y_0$ and by $\phi_1$ the gluing map corresponding to $Y_1$. 
We denote by $Y'_0$, and so on, other choices representing the same local data. See Appendix \ref{Local Data} for the definitions of the terms cobordism, regular slice, gluing map, and equivalence.
Now,  the fact that $\phi_0$ and $\phi'_0$ are in the same equivalence class gives us, by definition, isomorphisms
\[
f_0\colon Y_0\to Y'_0, \quad g_0\colon Z\to Z'
\]
such that $g_0 \circ \phi_0=\phi'_0 \circ f_0$ on $\mu^{-1}(0,0+\delta_0'')$.
We obtain an isomorphism
\[
(f_0,g_0)\colon Y_0\cup_{(\phi_0,\delta_0)} Z\to Y'_0\cup_{(\phi'_0,\delta'_0)} Z'
\]
precisely as in \cite[Lemma 2.5]{Go11}.
If we now glue $Y_1$ and $Y_1'$ into those spaces using $\phi_1$ and $\phi'_1$, we could try to extend $(f_0,g_0)$ to
be an isomorphism between the spaces $M= Y_0\cup_{(\phi_0,\delta_0)} Z \cup_{(\phi_1,\delta_1)} Y_1$ and $M'= Y'_0\cup_{(\phi'_0,\delta'_0)} Z' \cup_{(\phi'_1,\delta'_1)} Y'_1$.
But there is no clear way
to do this. While,  since the gluing maps $\phi_1$ and $\phi'_1$ are equivalent, we do know that some isomorphism $g_1: Z \to Z'$ extends to 
$Z \cup_{(\phi_1,\delta_1)} Y_1$,
we do not know that \emph{specifically} $g_0$ extends.

Of course, that does not necessarily mean that $M$ and $M'$ cannot be isomorphic, but it indicates that more information than local data is required to determine if they are.
\end{remark}

\begin{remark}
In \cite{Ka99}, Karshon defines the decorated graph associated to a closed Hamiltonian $S^1$-manifold of dimension four and shows that it determines the isomorphism type. It follows from the definition of the decorated graph (see \cite[p.6-7]{Ka99}) that in the semi-free case
the small fixed point data
determine the 
decorated graph: the semi-free assumption implies that there are no edge-labels, so the graph is determined by the critical levels and the genus and size of the fixed surfaces, if exist; see \Cref{fig:1} on the left for an example of a decorated graph in that case. Hence, dimension six is the lowest dimension in which there can be non-isomorphic closed Hamiltonian $S^1$-manifolds with the same small fixed point data.
\end{remark}

\subsection*{An example of manifolds with the same small fixed point data at a critical level and no isomorphism between neighborhoods of that level}

The example that we give contradicts \cite[Lemma 3.13]{Go11}. 

\begin{assertion*}\label{lem:smallfixedpointdata}
\cite[Lemma 3.13]{Go11}. 
Assume that $M^1$ and $M^2$ are semi-free Hamiltonian $S^1$-manifolds  whose momentum maps are proper and have bounded images. Assume that $\lambda$ is a common critical value of $\mu_1$ and $\mu_2$ with only fixed point components of index $1$ \footnote{For us, the index of a fixed point (component) is the number of negative weights.}.
    Suppose further that
    \begin{itemize}
        \item $M^1$ and $M^2$ have the same small fixed point data at  $\lambda$.
        \item there is a regular $t_0$ right below $\lambda$ and a symplectomorphism $f_0\colon M^1_{t_0}\to M^2_{t_0}$ that respects the Euler classes of the principal bundles $S^1\to \mu_i^{-1}(t_0)\to M^i_{t_0}$.
        \item for any closed interval $t_0\in I\subset [t_0,\lambda)$, the pair $(M^1_{t_0},\{\omega_t\}_{t\in I})$ is rigid.
    \end{itemize}
    Then there is $\eps>0$ such that $\mu^{-1}_1((\lambda-\eps,\lambda+\eps))$ and $\mu^{-1}_2((\lambda-\eps,\lambda+\eps))$ are isomorphic.
\end{assertion*}

\begin{example} \label{ex:openDpoly}
We give two open Delzant polytopes in $\R^3$, which are the momentum images of non-compact symplectic toric manifolds, whose vertices are located in the plane $y=0$. We will make sure that these polytopes agree in the open half space $y<0$, but differ for $y>0$.

    For that, consider the points $(0,0,0)$, $(0,0,2)$, $(2,0,2)$, $(2,0,1)$ and $(1,0,0)$ at the plane $y=0$ in $\R^3$. Let $P_{\eps}$, $\eps>0$, be the convex hull of the following lines:
    \begin{enumerate}
    \item $L_1=s(1,1,0)+(0,0,0)$, $s\in (-\eps,\eps)$.
    \item $L_2=s(1,1,1)+(0,0,2)$, $s\in (-\eps,\eps)$.
    \item $L_3=s(0,1,1)+(2,0,2)$, $s\in (-\eps,\eps)$.
    \item $L_4=s(0,1,0)+(2,0,1)$, $s\in (-\eps,\eps)$.
    \item $L_5=s(0,1,0)+(1,0,0)$, $s\in (-\eps,\eps)$.
\end{enumerate}
For $\eps$ sufficiently small, $P_{\eps}$ is an open convex polytope with edges $L_1,\hdots,L_5$ as specified above. 
We will now obtain two different open polytopes by chopping $P_{\eps}$ with two different hyperplanes. We define $H_1$ and $H_2$ by
\[
H_1=(0,0,2)+s_1(0,0,1)+s_2(0,1,0), \quad s_1,s_2\in \R
\]
\[
H_2=(0,0,2)+s_1(2,1,0)+s_2(0,0,1), \quad s_1,s_2\in \R.
\]
Denote by $P^1_{\eps},P^2_{\eps}$ the open sets obtained from $P_{\eps}$ by chopping along $H^1,H^2$. Again, for $\eps>0$ sufficiently small, $P^1:=P^1_{\eps},P^2:=P^2_{\eps}$ are open convex polytopes. They have the following properties:
\begin{enumerate}
    \item $P^1_{\eps}$ has vertices $v_1$ and $v_2$ located at $(0,0,2)$ and $(2,0,2)$. The edges adjacent to $v_1$ are $(1,0,0)$, $(1,1,0)$ and $(-1,-1,-1)$, the edges adjacent to $v_2$ are $(-1,0,0)$, $(0,1,0)$ and $(0,-1,-1)$. The remaining edges of $P^1_{\eps}$, not adjacent to any vertex, are $L_1$, $L_4$ and $L_5$.
    \item $P^2_{\eps}$ has vertices $v_1$ and $v_2$ located at $(0,0,0)$ and $(0,0,2)$. The edges adjacent to $v_1$ are $(0,0,1)$, $(-1,-1,0)$ and $(2,1,0)$, the edges adjacent to $v_2$ are $(0,0,-1)$, $(-1,-1,-1)$ and $(2,1,1)$. The remaining edges of $P^2_{\eps}$, not adjacent to any vertex, are $L_3$, $L_4$ and $L_5$.
\end{enumerate}
The Delzant condition for both these open convex polytopes hold, that is, at each vertex the edges form a basis for $\Z^3$. Therefore, both $P^1$ and $P^2$ represent open symplectic toric manifolds $Y_1$ and $Y_2$. Moreover, $Y_1$ and $Y_2$ are $T^3$-equivariantly symplectomorphic below level $y=0$ \footnote{Without precomposing any of the actions with an automorphism of $T^3$.}, as follows from the fact that their momentum images coincide below $y=0$ and  \cite[Theorems 1.3.1 and 1.3.2]{KL15}. When restricting the $T^3$-action to the circle corresponding to the $y$-coordinate, we obtain two $S^1$-manifolds $M^1$ and $M^2$ that agree below level $0$. Both $S^1$-actions are semi-free, because an entry in the second coordinate of any edge is contained in $\{0,-1,1\}$.\\
Note further that each non-empty cross section $y=const.$ (including $y=0$) looks like the $T^2$-momentum image of one blowup of $S^2\times S^2$, which is a two blowup of $\C \PP^2$, see \Cref{fig:cropped-two-cuts}. So the family $(S^2\times S^2 \# \overline{\C \PP}^2,\omega^1_t)$, $t<0$ is rigid by \Cref{thm:rigid}.\\
The only fixed point component of this $S^1$-action on $M^1$ is a sphere; it belongs to the edge between $(0,0,2)$ and $(2,0,2)$ in the toric $Y_1$. Similarly, the only fixed point component of the $S^1$-action on $M^2$ is a sphere, belonging to the edge between $(0,0,0)$ and $(0,0,2)$ in the toric $Y_2$. 
Accordingly, the index of both fixed spheres is $1$, and the only critical level, $0$, is simple. See \Cref{fig:reducedspace0} for the $T^2$-momentum images of the reduced spaces; it becomes clear that the self-intersection of both fixed spheres in their respective reduced spaces is $0$.\\
Moreover, there is a diffeomorphism (which even preserves the symplectic form) between the reduced spaces at level $0$ that maps the fixed point components into each other. Indeed, by flipping their $T^2$-momentum images at a vertical line (which corresponds to precomposing the $T^2$-action with the automorphism $(z_1,z_2)\mapsto (\bar{z}_1,z_2)$ of $T^2$), it becomes clear that the reduced spaces are each a blowup of the symmetric $S^2\times S^2$ by the same size at the point $\{\southpole\}\times \{\southpole\}$ performed in the embedded closed ball $\mathcal{B}$ indicated in \Cref{fig:ball}. The symplectomorphism of $S^2\times S^2$ given by $(p_1,p_2)\mapsto (p_2,p_1)$ preserves $\mathcal{B}$ and hence induces the desired symplectomorphism between the reduced spaces.\\
We conclude that $M^1$ and $M^2$ have the same small fixed point data.
However, 
although we can separately identify the reduced spaces at level $0$ and the $S^1$-manifolds below level $0$, 
there is no 
isomorphism $f$ of the momentum-map preimages in $M^1$ and $M^2$ of $(-\eps_0,\eps_0)$ for some $\eps_0>0$. 
For $0<\delta<\eps_0$, denote by $S_i$ the
fundamental class of the sphere that is the preimage of the fixed sphere under the map 
$M^{i}_{-\delta} \to M^{i}_{0}$ induced from the flow of the gradient vector field of the momentum map w.r.t.\  an invariant metric, see \S \ref{rem:criticalvalue} for the description of the map. 
The isomorphism $f$ would need to send 
$S_1$ to $S_2$, because it maps the fixed spheres at $y=0$ into each other. However, the symplectic form evaluates on the class $S_1$ differently than it does on $S_2$ (see \Cref{fig:reducedspacebefore0}). We get a contradiction to the assertion of \cite[Lemma 3.13]{Go11}.

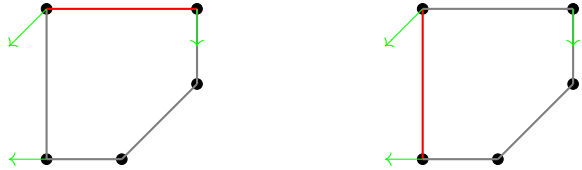
\begin{figure}[h]
	\begin{center}
		\begin{tikzpicture}
  \hspace{-2.5cm}
			\filldraw[black] (0,0) circle (2pt);
			\filldraw[black] (0,2) circle (2pt);
			\filldraw[black] (2,2) circle (2pt);
			\filldraw[black] (2,1) circle (2pt);
                \filldraw[black] (1,0) circle (2pt);
			\draw[gray, thick] (0,0) -- (0,2);
			\draw[red, thick] (0,2) -- (2,2);
			\draw[gray, thick] (2,2) -- (2,1);
			\draw[gray, thick] (1,0) -- (2,1);
                \draw[gray, thick] (1,0) -- (0,0);
                \draw[green, ->] (0,0) -- (-.5,0);
                \draw[green, ->] (0,2) -- (-.5,1.5);
                \draw[green, ->] (2,2) -- (2,1.5);
		\hspace{5cm}
			\filldraw[black] (0,0) circle (2pt);
			\filldraw[black] (0,2) circle (2pt);
			\filldraw[black] (2,2) circle (2pt);
			\filldraw[black] (2,1) circle (2pt);
                \filldraw[black] (1,0) circle (2pt);
			\draw[red, thick] (0,0) -- (0,2);
			\draw[gray, thick] (0,2) -- (2,2);
			\draw[gray, thick] (2,2) -- (2,1);
			\draw[gray, thick] (1,0) -- (2,1);
                \draw[gray, thick] (1,0) -- (0,0);
                \draw[green, ->] (0,0) -- (-.5,0);
                \draw[green, ->] (0,2) -- (-.5,1.5);
                \draw[green, ->] (2,2) -- (2,1.5);
		\end{tikzpicture}
	\end{center}
 \caption{The $T^2$-momentum images of the reduced spaces at $y=0$ for both open symplectic toric manifolds. The red lines correspond to the fixed spheres in the respective symplectic toric manifolds and are of the same length. The arrows indicate how the cross section changes as $y$ decreases.}\label{fig:reducedspace0}
 \end{figure}
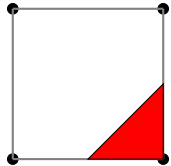
\begin{figure}[h]
	\begin{center}
		\begin{tikzpicture}
			\filldraw[black] (0,0) circle (2pt);
			\filldraw[black] (0,2) circle (2pt);
			\filldraw[black] (2,2) circle (2pt);
			\filldraw[black] (2,0) circle (2pt);
			\draw[gray, thick] (0,0) -- (0,2);
			\draw[gray, thick] (0,2) -- (2,2);
			\draw[gray, thick] (2,2) -- (2,0);
			\draw[gray, thick] (2,0) -- (0,0);
            \draw[fill=red]  (1,0) -- (2,1) -- (2,0) -- cycle;
		\end{tikzpicture}
	\end{center}
 \caption{The preimage of the red region under the momentum map from the toric symmetric $S^2\times S^2$ to $\R^2$ is the embedded closed ball $\mathcal{B}$.}\label{fig:ball}
 \end{figure}
 
\begin{figure}[h]
	\begin{center}
		\begin{tikzpicture}
			\filldraw[black] (-.3,0) circle (2pt);
			\filldraw[black] (-.3,1.6) circle (2pt);
			\filldraw[black] (2,1.6) circle (2pt);
			\filldraw[black] (2,1) circle (2pt);
                \filldraw[black] (1,0) circle (2pt);
			\draw[blue, thick] (-.3,0) -- (-.3,1.6);
			\draw[red, thick] (-.3,1.6) -- (2,1.6);
			\draw[gray, thick] (2,1.6) -- (2,1);
			\draw[gray, thick] (1,0) -- (2,1);
                \draw[gray, thick] (1,0) -- (-.4,0);
		\end{tikzpicture}
	\end{center}
 \caption{The $T^2$-momentum images of the reduced space at $y=-0.4$ for both open symplectic toric manifolds. The red line corresponds to the class $S_1$, on which the symplectic form evaluates to $1+0.4=1.4$, and the blue line corresponds to the class $S_2$, on which the form evaluates to $1-0.4=0.6$.}\label{fig:reducedspacebefore0}
 \end{figure}
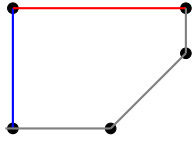
\end{example}

\newpage

\section{Preliminaries: the Morse flow}\label{sec:pre}
We describe the map from a reduced space at a regular value to a reduced space at a non-extremal critical value above it, induced from the flow of the gradient vector field of the momentum map w.r.t.\ an invariant metric. We call it the Morse flow. It will play an important role in extending an isomorphism beyond a non-extremal critical level.\\

 Let $(M,\omega,\mu)$ be a connected, semi-free Hamiltonian $S^1$-manifold. Assume that $\mu$ is proper and its image is bounded.

\begin{noTitle}\label{nt:morse1}
     Let $\lambda$ be a critical value of the momentum map $\mu$. Let $\eps>0$ be such that there is no critical value in $[\lambda-\eps,\lambda)$. The normalized flow $\Phi_t$ of the gradient vector field of $\mu$ with respect to some invariant metric gives an equivariant diffeomorphism
     \begin{equation} \label{eq:phiti}
     \mu^{-1}(\lambda-\eps)\times [\lambda-\eps,\lambda)\to \mu^{-1}([\lambda-\eps,\lambda)), \quad (p,t)\mapsto \Phi_t(p)
     \end{equation}
     under which $\mu$ pulls back to $
     (p,t)\mapsto t.$  

     Choosing a different invariant metric does not change the equivariant isotopy type of the  equivariant diffeomorphism \eqref{eq:phiti}.
\end{noTitle}  

To define the Morse flow, we first review the implications of the local normal form for a Hamiltonian $S^1$-action in case  the action is semi-free and the manifold is of dimension six. 
Recall that we use the convention that the \textbf{index} of $\mu$ at a fixed point is the number of negative weights \footnote{This convention differs from that of \cite{Go11}, where the index is the usual index of a Morse-Bott function, that is, double our index.}. The {\bf co-index} is the number of positive weights.

\begin{noTitle}\label{nt:localnormalform}{\bf Local normal form for a Hamiltonian $S^1$-action on $(M^{2n},\omega)$.} For an integer $w$, denote by $\C_w$ the $S^1$-representation $\chi$ on $\C$ given by $\chi(t)(z)=t^w\cdot z$. By the local normal form theorem for Hamiltonian group actions, there is a local chart around any fixed point $p$ given by $n$ complex coordinates $(z_1,\hdots,z_n)$ defined on a suitably small ball ${\ball}^{2n} \subset \C^n$ centered at the origin, and $n$ integers $w_1,\hdots,w_n$ called \textbf{the weights of the action at $p$}, such that the pullback of the $S^1$-action is given by
       \[
      {\ball}^{2n}\subset \C_{w_1}\oplus \hdots \oplus \C_{w_n},
      \]
the pullback of $\omega$ is given by the standard form on  
 ${\ball}^{2n}$, and the pullback of the momentum map is the corresponding standard momentum map of said representation, up to adding a constant.
      \end{noTitle}

      \begin{noTitle}\label{nt:local6}{\bf Local normal form for a semi-free Hamiltonian $S^1$-action on $(M^6,\omega)$.}
      Using the local normal form, we see that for an effective Hamiltonian circle action in dimension six, fixed point components are symplectic submanifolds of dimension $0$, $2$, or $4$.
      If the action is semi-free, weights can only be $-1$, $0$ or $1$. We list the possible triples of weights, up to ordering.
      \begin{itemize}
          \item For isolated fixed points, the weights are $(1,1,1)$ for a minimal, $(-1,-1,-1)$ for a maximal, or $\pm (1,1,-1)$ for a non-extremal fixed point.
          \item For fixed surfaces, the weights are either $\pm (0,1,1)$ corresponding to an extremal fixed point, or $\pm (0,1,-1)$ corresponding to non-extremal fixed points.
          \item For fixed four-manifolds, the weights are given by $\pm(0,0,1)$. In particular, a fixed four-manifold is always extremal.
      \end{itemize}
      \end{noTitle}

   In the rest of the section, $\lambda$ is a non-extremal critical value and $M$ is of dimension six.
\begin{noTitle}\label{rem:criticalvalue}
Since $\mu$ is a Morse-Bott function, for a fixed point component $C$ at the level $\lambda$, we can describe the corresponding \textbf{stable submanifold}, defined by
$$\{p\in M \,|\, \lim_{t \to \infty} \Phi_{t}(p) \in C\}.$$
Consider an isolated fixed point $p$ with index $1$ and the local normal form ${\ball}^6\subset \C_{-1}\oplus \C_{1}\oplus \C_1$, with the standard momentum map
    $$(z_1,z_2,z_3) \mapsto -|z_1|^2+|z_2|^2+|z_3|^2$$
    and the standard metric $g$, around it.
    The time-$T$-flow $\Phi_T$ beginning in $\mu^{-1}(\lambda-\eps)\cap {\ball}^6$, $T\geq \eps$, is well-defined in
    $$U:={\ball}^6\cap (\mu^{-1}(\lambda-\eps)\smallsetminus \{\mu^{-1}(\lambda-\eps)\cap (\C_{-1}\oplus \{0\}\oplus \{0\})\}).$$
    We define $f_{\Morse}(\eps)\colon {\ball}^6\cap \mu^{-1}(\lambda-\eps)\to {\ball}^6\cap \mu^{-1}(\lambda)$ by $\Phi_{\eps}$ on $U$ and by $z\mapsto p$ on $\C_{-1}\oplus \{0\}\oplus \{0\}$.\\
    In other words, we define $f_{\Morse}(\eps)$ as the limit of $\Phi_{\delta}$ as $\delta$ goes to $\eps$. This converges uniformly, since for all $p$ in the domain, $d(\Phi_{\delta}(p),f_{\Morse}(\eps)(p))\leq \eps-\delta$, where $d$ is the metric on $M$ defined by geodesic distance with respect to $g$. Therefore $f_{\Morse}(\eps)$ is continuous and clearly equivariant.

    Similarly, for an isolated fixed point $p$ with index $2$, and the local normal form ${\ball}^{6} \subset \C_{-1}\oplus \C_{-1}\oplus \C_1$, with the standard momentum map
    $$(z_1,z_2,z_3) \mapsto -|z_1|^2-|z_2|^2+|z_3|^2$$
    and the standard metric $g$, around it,
    we set 
     $$U:={\ball}^6\cap (\mu^{-1}(\lambda-\eps)\smallsetminus \{\mu^{-1}(\lambda-\eps)\cap (\C_{-1}\oplus \C_{-1} \oplus \{0\})\}).$$
    We define $f_{\Morse}(\eps)\colon {\ball}^6\cap \mu^{-1}(\lambda-\eps)\to {\ball}^6\cap \mu^{-1}(\lambda)$ by $\Phi_{\eps}$ on $U$ and by $z\mapsto p$ on $\C_{-1}\oplus \C_{-1} \oplus \{0\}$. 
    Again, $f_{\Morse}(\eps)$ is continuous and equivariant.\\
    
    In case there is a fixed surface $\Sigma$ at $\lambda$, an explicit description like that is not possible anymore. But here, by \cite[Proposition 3.2]{AB95}, the unstable submanifold belonging to $\Sigma$ is indeed a smooth submanifold. This is necessarily $S^1$-invariant if the metric is. In fact, if there is a symplectic $G$-action on a neighborhood of $\Sigma$, for any compact Lie group $G$, then the unstable submanifold is also (locally) $G$-invariant. Further, the map $u(\Sigma)\colon \mathcal{U}(\Sigma)\to \Sigma$ mapping a point to the limit of its flow under $\Phi$ is smooth and defines a fiber bundle near $\Sigma$ whose fiber is a ball of dimension twice the index of $\Sigma$; the fiberwise $S^1$-action gives a fiberwise complex structure on that bundle.\\
    In our case, the fiber is of dimension two, so we call this bundle the \textbf{negative normal bundle}; similarly, we define the positive normal bundle.
    In both cases, we define $e({\Sigma})_-$ resp.\ $e(\Sigma)_+$ to be the Euler class of the negative resp.\ positive normal bundle of $\Sigma$.
    The intersection $S_{-\eps}$ of $\mathcal{U}(\Sigma)$ with a level set $\lambda-\eps$ is then an $S^1$-bundle over $\Sigma$, so that its orbit space $S_{-\eps}/S^1$ is diffeomorphic to $\Sigma$. A diffeomorphism is given by restricting $u(\Sigma)$ to $S_{-\eps}/S^1$, so that $S_{-\eps}/S^1$ is a symplectic submanifold of $M_{\lambda-\eps}$ for $\eps$ small enough. Note that then $e(\Sigma)_-$ pulls back to the restriction of the Euler class of $S^1\to \mu^{-1}({\lambda-\eps})\to M_{\lambda-\eps}$ to $\Sigma_{-\eps}$ under $\Sigma_{-\eps}\to \Sigma$.\\
    For the same reasons as above, we can define (near $\Sigma$) the continuous, equivariant map $f_{\Morse}(\eps)\colon \mu^{-1}({\lambda-\eps})\to \mu^{-1}(\lambda)$ to be the limit of $\Phi_{\delta}$ as $\delta$ goes to $\eps$.\\
    
    All in all, for a choice of metric that, near the fixed components, is as specified above and $\eps>0$ small enough, we get a well-defined, equivariant continuous map 
    by    
    \begin{equation} \label{eq:morse}
         f_{\Morse}=f_{\Morse}(\eps)\colon \mu^{-1}({\lambda-\eps})\to \mu^{-1}({\lambda}), \quad f_{\Morse}(p)=\lim\limits_{\delta \to \eps_-} \Phi_{\delta}(p).
    \end{equation}
    We call this map the {\bf Morse flow}. Since this map is equivariant, it descends to a continuous map $M_{\lambda-\eps}\to M_{\lambda}$ that we will also call $f_{\Morse}$.\\
    For a fixed point of index $1$ considered to be in $M_{\lambda}$, the preimage under $f_{\Morse}(\eps)$ in $M_{\lambda-\eps}$ is a point, whereas for a fixed point $p$ of index $2$, the preimage in $M_{\lambda-\eps}$ is an embedded symplectic 2-sphere $S_{-\eps}$ of size $\eps$ and self-intersection $-1$. 
    As the preimages of different points, the spheres corresponding to the fixed points of index $2$ are pairwise disjoint.
    The preimage of a fixed surface $\Sigma\subset M_{\lambda}$ is an embedded symplectic surface $\Sigma_{-\eps}$ of the same genus. 
    \end{noTitle}

\begin{Notation}\label{not:FF'}
    Denote by $F$ the set of fixed points in $M_{\lambda}$, by $F_{\iso}$ its subset of isolated fixed points, and by $F_{\iso, 2}$ its subset of isolated fixed points of index $2$. 
Denote by $F'$, $F'_{\iso}$ and $F'_{\iso,2}$ the preimages of $F$, $F_{\iso}$ and $F_{\iso,2}$ in $M_{\lambda-\eps}$ under $f_{\Morse}(\eps)$.
\end{Notation}

Since the restriction of $f_{\Morse}$ to $M_{\lambda-\eps}\smallsetminus F'$ coincides with the gradient flow $\Phi_t$ of $\mu$, we get the following corollary.
\begin{corollary}\label{cor:restdiffeo}
    The restriction of $f_{\Morse}(\eps)$ to $M_{\lambda-\eps}\smallsetminus F'$ induces a diffeomorphism
      \begin{equation} \label{eq:morsediffeo}
        M_{\lambda-\eps}\smallsetminus F'\hookrightarrow M_{\lambda}\smallsetminus F.
    \end{equation}
\end{corollary}

Moreover, the definition of $f_{\Morse}$ also implies the following corollary.

\begin{corollary}\label{cor:reshomeo}
The restrictions of $f_{\Morse}(\eps)$ to $M_{\lambda-\eps}\smallsetminus F'_{\iso}$ and to $M_{\lambda-\eps}\smallsetminus F'_{\iso,2}$ induce homeomorphisms 
     \begin{equation} \label{eq:morsehomeo1}
        M_{\lambda-\eps}\smallsetminus F'_{\iso}\to M_{\lambda}\smallsetminus F_{\iso}
    \end{equation}
    and
    \begin{equation} \label{eq:morsehomeo}
        M_{\lambda-\eps}\smallsetminus F'_{\iso,2}\to M_{\lambda}\smallsetminus F_{\iso,2}.
    \end{equation}
\end{corollary}

\begin{proof}
It is clear that the maps are bijective and continuous.
    The inverse restricted to $M_{\lambda}\smallsetminus F$ is continuous since  \eqref{eq:morsediffeo} is a diffeomorphism. Moreover, for a small open neighborhood $U=U'_{\iso}$ or $U'=U'_{\iso,2}$ around $F'_{\iso}$ and $F'_{\iso,2}$, resp., the restriction of $f_{\Morse}(\eps)$ to $M_{\lambda-\eps}\smallsetminus U' \to M_{\lambda}\smallsetminus f_{\Morse}(\eps)(U')$
    is a bijective, continuous map from a compact space into a Hausdorff space, hence a homeomorphism. In particular the inverse of  each of the maps \eqref{eq:morsehomeo1} and \eqref{eq:morsehomeo} is continuous at each point in $F \smallsetminus F_{\iso}$ and in  $F \smallsetminus F_{\iso,2}$, respectively.
\end{proof}
    We call the topological connected sum of a topological manifold $M$ of dimension four with $\overline{\C\PP}^2$ the \textbf{blowup in the topological category} at a point. 
    The following claim is immediate from the definitions.
    \begin{claim}\label{cl:top-bup}
     The Morse flow $f_{\Morse} \colon M_{\lambda-\eps} \to M_{\lambda}$ induced by  \eqref{eq:morse} 
      coincides in the topological category with the map $M_{\lambda-\eps} \to M_{\lambda}$ of the blowup at the isolated fixed points of co-index $1$. 
      \end{claim}
     
    \begin{remark} 
    The definition and properties of $f_{\Morse}$ presented above 
    still hold if we let $\mu$ be an $S^1$-invariant Morse-Bott function whose critical set is precisely the set of $S^1$-fixed points and assume that in a neighborhood of the fixed point set, $\mu$ is a momentum map for the $S^1$-action on $(M,\omega)$. Of course, $M_t$ is to be understood as the 'reduced space' with respect to the Morse-Bott function $\mu$. 
    To justify this, we note that for a critical level $\lambda$ of $\mu$, $M_{\lambda}$ can still be given a smooth structure using the same arguments as for an actual momentum map. This is because $M_{\lambda}\smallsetminus F$ is clearly smooth, and $\mu$ is a momentum map near $F$.
\end{remark}

     Note that $f_{\Morse}(\eps) \colon M_{\lambda-\eps} \to M_{\lambda}$ is orientation-preserving w.r.t.\ the orientation induced by the symplectic forms $\omega_{\lambda-\eps}$ and $\omega_{\lambda}$.
      We will further look at the effect of the Morse flow on the symplectic form in case $\mu$ is a momentum map. First we strengthen \Cref{cor:restdiffeo}; again, this is since the
   restriction of $f_{\Morse}$ to $M_{\lambda-\eps}\smallsetminus F'$ coincides with the gradient flow of the momentum map.

\begin{corollary}\label{cor:restdiffeo1}
For the diffeomorphism \eqref{eq:morsediffeo} induced by the
   restriction of $f_{\Morse}(\eps)$ to $M_{\lambda-\eps}\smallsetminus F'$, the form $(f_{\Morse}(\eps)|_{M_{\lambda-\eps}\smallsetminus F'})_{*}{\omega_{\lambda-\eps}}$, as a form on $M_{\lambda}\smallsetminus F$, converges to $\omega_{\lambda}$ as $\eps$ goes to $0$.

       In particular, 
        the symplectic volume of $(M_{\lambda-\eps},\omega_{\lambda-\eps})$ approaches that of $(M_{\lambda},\omega_{\lambda})$ as $\eps$ approaches $0$.
      
\end{corollary}

   Next, recall that the Morse flow $f_{\Morse}(\eps)$ is a homeomorphism onto its image $M_{\lambda} \smallsetminus F_{\iso,2}$ when restricted to $M_{\lambda-\eps}\smallsetminus F'_{\iso,2}$, by \Cref{cor:reshomeo}. We will also use the following fact, which is a simple application of Stokes.
   \begin{lemma}\label{lem:volumedisk}
       Let $B$ be a closed symplectic ball of radius $\delta$ in $\C^2$, that is, the set $\{(z_1,z_2)\in \C^2\colon |z_1|^2+|z_2|^2\leq \delta^2\}$ endowed with the standard symplectic form $\omega$. Then for any smooth embedding $\iota\colon D^2\to B$ such that $\partial D^2\subset \partial B$ is an orbit of the diagonal $S^1$-action, we have that $\int_{D^2}\iota^*\omega=\pm \pi\delta^2$.
   \end{lemma}
        \begin{lemma}\label{lem:cohomologyconverges}
           The class $(f_{\Morse}|_{M_{\lambda-\eps}\smallsetminus F'_{\iso,2}})_{*}[\omega_{\lambda-\eps}]$ converges to $[\omega_{\lambda}]$ on $M_{\lambda} \smallsetminus F_{\iso,2}$ as $\eps$ goes to $0$.
        \end{lemma}
\begin{proof}
For a primitive homology class $A$ in $H_2(M_{\lambda}\smallsetminus F_{\iso,2})$, and $t \in (\lambda-\eps,\lambda)$, denote by $A_t$ the preimage of $A$  under  $$f_{\Morse}\colon M_t\smallsetminus F'_{\iso,2}\to M_\lambda\smallsetminus F_{\iso,2}.$$
We will show that  as $t<\lambda$ goes to $\lambda$, the evaluation $\omega_t(A_t)$ goes to $\omega_{\lambda}(A)$.

As a primitive homology class in $H_2(M_{\lambda}\smallsetminus F_{\iso,2})$, the class $A$ is represented by an immersed sphere $\iota_A(S^2)$. Up to homotopy, $\iota_A$ is transverse to $F$ (see \cite[Theorem 2.4]{MH76}). In particular, $\iota_A(S^2)$ is disjoint from the isolated fixed points and meeting the fixed surfaces only a finite number of times in $m$ distinct points $p_1,\hdots,p_m$. For $\delta>0$ small enough, there exists an embedding of $m$ disjoint, closed balls $\mathcal{B}_i$, all of radius $\delta/\sqrt{\pi}$, such that $p_i$ is only contained in $\mathcal{B}_i$ and $\mathcal{B}_i \cap \iota_A(S^2)$ is a disk, implying that $\partial \mathcal{B}_{i} \cap \iota_A(S^2) \cong S^1$ is the unknot in $\partial \mathcal{B}_i\cong S^3$. Therefore, up to isotopy, we may assume that $\partial \mathcal{B}_{i} \cap \iota_A(S^2)$ is an orbit of the diagonal $S^1$-action on $S^3$. We denote by $\mathcal{D}_i\subset S^2$ the preimage of $\mathcal{B}_i\cap\iota_A(S^2)$ under $\iota_A$.\\
 For $\mathcal{B}:= \mathcal{B}_1\cup \hdots \cup \mathcal{B}_m$, we then have
    $$\int\limits_{A\smallsetminus \mathcal{B}} \omega_{\lambda}\in [\omega_{\lambda}(A)-m\delta^2,\omega_{\lambda}(A)+m\delta^2]$$
    by \Cref{lem:volumedisk}.
    For $t\in (\lambda-\eps,\lambda)$, denote by $\mathcal{B}^t,\mathcal{B}^t_i$ the preimage of $\mathcal{B},\mathcal{B}_i$ under
    $$f_{\Morse}\colon M_t\smallsetminus F'_{\iso,2}\to M_\lambda\smallsetminus F_{\iso,2}.$$
    Denote by $\iota_{A_t}=f_{\Morse}^{-1}\circ \iota_A\colon S^2\to M_t$ the induced map. This might not be smooth near $f_{\Morse}^{-1}(\{p_1,\hdots,p_m\})$, but we may 
    redefine $\iota_{A_t}$ by replacing $(\iota_{A_t})_{|\mathcal{D}_i}\colon \mathcal{D}_i\to \mathcal{B}^t_i$ with a smooth embedding $j_i\colon \mathcal{D}_i\to 
    \mathcal{B}^t_i$ such that $j_i=\iota_{A_t}$ near $\partial \mathcal{D}_i$; this is possible because $\iota_{A_t}(\partial \mathcal{D}_i)$ is the unknot in $\partial 
    \mathcal{B}^t_i$. Since $j_i$ and $(\iota_{A_t})_{|\mathcal{D}_i}$ are certainly homotopic rel boundary, we do not change $\omega_t([\iota_{A_t}(S^2)])$ when doing so.\\
    Also, we may again assume that $\iota_{A_t}(\partial \mathcal{D}_i)$ is an $S^1$-orbit of the diagonal action on $S^3$.\\
    
    Then, in $M_t$, for $t \in (\lambda-\eps,\lambda)$ with $\eps$ small enough, $\mathcal{B}_t$ is contained in a symplectic ball of radius $2\delta/\sqrt{\pi}$, so the symplectic volume of $\mathcal{B}_t\cap A_t$ with respect to $\omega_t$ is in $[-4m\delta^2,4m\delta^2]$ by \Cref{lem:volumedisk}. This implies that $\int\limits_{A_t\smallsetminus \mathcal{B}_t} \omega_t$ is in $[\omega_t(A_t)-4m\delta^2,\omega_t(A_t)+4m\delta^2]$. We now have
    \begin{equation}\label{eq:limit}
        \lim\limits_{t\to \lambda} \int\limits_{A_t\smallsetminus \mathcal{B}_t} \omega_t=
    \int\limits_{A\smallsetminus \mathcal{B}} \omega_{\lambda},
    \end{equation}
    because of 
    \Cref{cor:restdiffeo1}.
    Therefore, if $\omega_t(A_t)$ did not approach $\omega_{\lambda}(A)$, then we could choose $\delta$ and $\theta>0$ small enough such that
    $$[\omega_{\lambda}(A)-m\delta^2,\omega_{\lambda}(A)+m\delta^2] \cap [\omega_t(A_t)-4m\delta^2,\omega_t(A_t)+4m\delta^2]=\emptyset$$
    whenever $t\in (\lambda-\theta,\lambda)$, contradicting \cref{eq:limit}.
\end{proof}

         A similar statement to
         \Cref{lem:cohomologyconverges} with index $1$ replacing co-index $1$ 
         holds for $(M_{\lambda+\eps},\omega_{\lambda+\eps})$.

\begin{Notation}\label{not:elambda}

We denote by $e(P_t)\in H^2(M_t;\Z)$ the Euler class of the principal $S^1$-bundle $S^1\to P_t\to M_t$. The equivariant diffeomorphism  \eqref{eq:phiti} gives an identification of $e(P_{t_0})$ and $e(P_{t_1})$ for $t_0,t_1\in [\lambda-\eps,\lambda)$ independent of the metric chosen, allowing us to write $e(P)$ for $e(P_t)$ when it is clear in which interval of regular values $t$ is contained.

The 'Euler class' $e_-$ at $M_{\lambda}$ is defined as follows: 
First, the pushforward of the usual Euler class $e(P_{\lambda-\eps})$ as a class on $M_{\lambda-\eps} \smallsetminus F'_{\iso}$ by the homeomorphism \eqref{eq:morsehomeo1} 
 is a class $\tilde{e}(P_{\lambda})$ in $H^{2}(M_{\lambda}\smallsetminus F_{\iso};\Z)$.
 Now, the inclusion $M_{\lambda}\smallsetminus F_{\iso}\hookrightarrow M_{\lambda}$ induces an isomorphism in the second cohomology groups, so we define $e_-$  to be the image of $\tilde{e}(P_{\lambda})$ under that isomorphism.\\
That way, it is clear that the restriction of $e(P^{-}_{\lambda})$ to $M_{\lambda}\smallsetminus F$ is the actual Euler class of the principal bundle $S^1\to P_{\lambda}\smallsetminus F \to M_{\lambda}\smallsetminus F$.\\
Similarly, replacing $M_{\lambda-\eps}$ with $M_{\lambda+\eps}$, we can define $e_+$. In general, $e_-\neq e_+$.
\end{Notation}

\section{Almost symplectic $\mu-S^1$-diffeomorphisms of free Hamiltonian $S^1$-manifolds}\label{sec:almostsymplectic} 

In the proof of \Cref{thm:extending-g}, we will piece an isomorphism below a critical level with an isomorphism of 
neighborhoods of the critical level. 
In this section, we describe this operation. We will also use piecing of isomorphisms to show that the rigidity assumption allows to extend an equivariant symplectomorphism below a critical level to arbitrarily close to the critical level. 

The result of piecing together equivariant symplectomorphisms might no longer be an equivariant symplectomorphism. 
However, we will show that it is an almost symplectic $\mu-S^1$-diffeomorphism. 
Recall that a {\bf $\mu-S^1$-diffeomorphism} is an equivariant diffeomorphism that respects the momentum maps. 

\begin{definition}
We call a diffeomorphism $\psi$ from a symplectic manifold $(X,\omega^{X})$ to a symplectic manifold $(Y,\omega^{Y})$ \textbf{almost symplectic} if $\psi^*{\omega^Y}$ and $\omega^X$ are \emph{isotopic} under the standard homotopy
\begin{equation*}\label{eq:omegas}
\omega(s):=s \psi^*{\omega^Y}+(1-s){\omega^X}, \,\,\,s\in [0,1].
\end{equation*}
That is, all $\omega(s)$ are symplectic forms that represent 
the cohomology class $[\omega^X]$ in 
$H^2(X;\R)$.
\end{definition}
\begin{remark}\label{rem:almostsymplecticopen}
   We will use frequently that the non-degeneracy of $\omega(s):=s \psi^*{\omega^Y}+(1-s){\omega^X}$ is an open condition in the following sense: if $\omega^X_t$, $\omega^Y_t$ and $\psi^t$ are smooth paths with $\omega^X_0=\omega^X$, $\omega^Y_0=\omega^Y$ and $\psi^0=\psi$, then there is $\eps>0$ such that ${\psi^s}^{*}{\omega^Y_{s'}}$ and $\omega^X_{s''}$ are isotopic under the standard homotopy, provided that $|s|,|s'|,|s''|<\eps$.
   
   That is to say, working with almost symplectic diffeomorphisms as opposed to symplectomorphisms gives us more flexibility. This will be more concrete in the next lemmata.
\end{remark}

Since the piecing of the isomorphisms will take place at a regular level, we first study almost symplectic $\mu-S^1$-diffeomorphisms
between \emph{free} Hamiltonian $S^1$-manifolds with proper momentum maps. 
Recall the notation $P_t,\,M_t,\,e(P_t),e(P)$ from Notations \ref{not:semifreeHamiltonian} and \ref{not:elambda}. 
\begin{Notation}
    For any $\mu-S^1$-diffeomorphism $f$ on $S^1 \acts (M,\mu)$ (into $M$ or another Hamiltonian $S^1$-manifold), we denote by $f^t$ the map it induces on $P_t$, and by $f_t$ the map it induces on $M_t$.
\end{Notation}

Let $(M,\omega=\omega^M,\mu=\mu_M)$ be a connected Hamiltonian $S^1$-manifold; assume that 
the circle action is {free} and the 
momentum map is a proper surjective map $\mu\colon M \to [0,1]$. Through this section, as in \eqref{eq:phiti}, we use the flow of the gradient vector field of the momentum map (w.r.t.\ some invariant metric) to identify $M\cong P_0\times [0,1]$ as $S^1$-manifolds such that the momentum map becomes projection onto the $[0,1]$-factor.
Recall the Duistermaat-Heckman formula
\begin{equation}\label{eq:dh}
[\omega_t]=[\omega_{t'}]+(t-t')e(P),
    \end{equation}
where $t,t'$ are abritrary values in $[0,1]$.\\

By \cite[Section 2]{GS89}, for each $t\in [0,1]$ there is $\delta>0$  and a one-form $\alpha^M$ that is a connection on the principal $S^1$-bundle $S^1\to P_t\to M_t$ such that $\mu^{-1}((t-\delta,t+\delta))$ is isomorphic to 
\begin{equation}\label{Udelta}
U=U_{\delta}:=P_t\times (t-\delta,t+\delta)
\end{equation}
endowed with the symplectic form 
\begin{equation}\label{omegaU}
\omega^{U,M}=(\omega_t+(t'-t)d\alpha^M)+\alpha^M\wedge dt',
\end{equation}
where  $t'$ is the coordinate in the $(t-\delta,t+\delta)$-factor\footnote{Here, it is understood that $(-\delta,\delta)=[0,\delta)$, for example.}
 in $U$.
We make the same assumptions and notation for $(N,\omega^N,\mu_N)$.\\

A $\mu-S^1$-diffeomorphism $f \colon M \to N$ gives an identification of $P_t^M=\mu_M^{-1}(t)$ and $P_t^N=\mu_N^{-1}(t)$.
We show that for a $\mu-S^1$-diffeomorphism $f$ such that $[f^{*}\omega^N]=[\omega^M]$, being almost symplectic can be checked at the reduced spaces. 
\begin{lemma}\label{lem:isotopybase}
    Consider a $\mu-S^1$-diffeomorphism $f$ from $(M,\omega^M)$ to $(N,\omega^N)$ 
    such that the induced diffeomorphism $M_t\to N_t$ is almost symplectic for every $t$ and $[f^*\omega^N]=[\omega^M]$ in $H^2(M;\R)$. Then $f$ is almost symplectic: the family
    \begin{equation}\label{eq:iso1}
    \omega(s):=sf^*\omega^N+(1-s)\omega^M, \,\, s \in [0,1]
    \end{equation}
    is an isotopy between $f^{*}\omega^N$ and $\omega^M$.
\end{lemma}
\begin{proof}
Since $[f^{*}{\omega^N}]=[\omega^M]$,
we only need to show that each $\omega(s)$ is a symplectic form on $M$.
This is a local statement, so we may restrict to a neighborhood 
$U=U^{\delta}$ with the symplectic form $\omega^{U,M}=(\omega^M_t+(t'-t)d\alpha^M)+\alpha^M\wedge dt'$ (and similarly for $\omega^{U,N}$), as in \eqref{Udelta} and \eqref{omegaU}.
 
 Since $f$ is a $\mu-S^1$-diffeomorphism between $M$ and $N$, we have $f_{*}(\xi_M)=\xi_N$, where $\xi_M$ and $\xi_N$ are the fundamental vector fields of the $S^1$-action. So we can write $\xi$ for both $\xi_M$ and $\xi_N$. Moreover, using $f$ to identify $P^M_t=\mu_M^{-1}(t)$ and $P^N_{t}=\mu_N^{-1}(t)$, 
 we have $f_*({\partial}_{t'})={\partial}_{t'}+v$, where $v\in T{P}_{t'}$. So, since $\omega^{U,N}(\xi,v)=\mu_N(v)=0$,
 $$\omega^{U,N}(f_*(\xi),f_*({\partial}_{t'}))=\omega^{U,N}(\xi,{\partial}_{t'}+v)=\omega^{U,N}(\xi,{\partial}_{t'}).$$
 This implies that $\omega^{U}(s)(\xi,{\partial}_{t'})$ does not depend on $s$. Since it is not $0$ for $s=0$, it is never $0$.

   Due to $TM_{|\mu_M^{-1}(t')}=\ker {\alpha}_{|P_{t'}}\oplus \R \xi \oplus \partial_{t'}$ for any $t'$, it remains to check that $\omega^{U}(s)$ does not degenerate on $\ker \alpha_{|P_{t'}}\cong T M_{t'}$ to deduce that $\omega^{U}(s)$ is symplectic. This follows from the assumption that ${\omega}_{t'}(s)$ is a symplectic form, since $$\omega^{U}(s)_{|TP_{t'}}=\pi^{*}\omega^{U}_{t'}(s).$$
\end{proof}
Under the identifications $M\cong \mu_M^{-1}(0)\times [0,1]$ and $N\cong \mu_N^{-1}(0)\times [0,1]$ obtained from the normalized gradient flow of the momentum map, any almost symplectic $\mu-S^1$-diffeomorphism $f\colon M\to N$ naturally gives rise to a smooth family $f^t \colon P^M_0 \to P^N_0$ of equivariant diffeomorphisms. This smooth family descends to a smooth family of almost symplectic diffeomorphisms $f_t\colon (M_0,\omega^M_t)\to (N_0,\omega^N_t)$ between the reduced spaces such that $f_t^*(e(P^N_0))=e(P^M_0)$. \Cref{lem:isotopybase} implies that the converse is true.
    \begin{lemma}\label{lem:lift}
        Any smooth family of (almost symplectic) diffeomorphisms (or homeomorphisms)
        $f_t \colon (M_0,\omega^M_t)\to (N_0,\omega^N_t) \,\, t \in [0,1]$
        such that $f_t^*(e(P^N_0))=e(P^M_0)$ lifts to a(n almost symplectic) $\mu-S^1$-diffeomorphism (or homeomorphism) $f\colon M\to N.$ More precisely, given a lift $\tilde{f}^0$ of $f_0$ to $P^M_0$, $f$ may be assumed to restrict to $\tilde{f}^0$ on $P^M_0$.
        
      If $f_t(s)$ depends smoothly on a parameter $s$ (from some smooth manifold), then $f_t(s)$ lifts to a smooth family of (almost symplectic) $\mu-S^1$-diffeomorphisms $f(s)\colon M\to N$. If $\partial_s f_t(s)\equiv 0$ for $t\in [0,\eps]$, then it may be assumed that $\partial_s f(s)\equiv 0$ on $\mu_M^{-1}([0,\eps])$.
    \end{lemma}
    \begin{proof}
        We view $f_0(s)$ as a map $M_0\times [0,1]\to N_0\times [0,1]$ and lift the family $f_0(s)$ to a family $f^0(s)$ on $P_0$. Such a lift exists because $f_0(s)$ was assumed to intertwine the Euler classes; we prove this for completeness in \Cref{lem:Eulerclasses}. If needed, we can take the lift to be a specified one.\\
        Assume at first that the $f_t$ are homeomorphisms. Given a lift a lift $f^0$ of $f_0$, we directly obtain a continuous family $f^t\colon P_0^M\to P_0^N$ of equivariant maps using the homotopy lifting property of principal bundles. These are indeed bijective (and hence homeomorphisms), since $f^t(p)=f^t(p')$ implies that $p$ and $p'$ are in the same $S^1$-orbit, and due to equivariance of $f^t$ and the freeness of the action this implies $p=p'$.\\
        
        Now assume the $f_t$ are diffeomorphisms. For fixed $s$, we now view $f_t(s)$ as a $t$-flow on the manifold $P^N_0/S^1$ and consider the time-dependent vector field $Y_t(s)$ it generates. We can lift $Y_t(s)$, using any connection of the principal $S^1$-bundle $P^N_0\to P^N_0/S^1$, horizontally to the total space. For a fixed $s$, we consider the $t$-flow $f^t(s)$ of the lifted vector field, starting at $f^0(s)$.\\
        Define $f(s)\colon M\to N$ by
        \[
        f(s)\colon P^M_0\times [0,1]\to P^N_0\times [0,1], \quad f(s)(p,t'):=(f^{t'}(s)(p),t').
        \]
        By definition, $f(s)\colon M\to N$ is a $\mu-S^1$-diffeomorphism. Also, each $f(s)$ satisfies $[f(s)^*\omega^N]=[\omega^M]$, since $[f(s)_t^*\omega_t^N]=[\omega_t^M]$ for all $t\in [0,1]$ by assumption. Therefore by \Cref{lem:isotopybase}, if each $f_t(s)$ is almost symplectic, then so is $f(s)$.
       
        Finally, if $f_t(s)$ does not depend on $s$ for $t\in [0,\eps]$, then $f^0(s)$ may be assumed not to depend on $s$. Since $Y_t(s)$ does not depend on $s$ and $f^t(s)$ is just the $t$-flow of $Y_t(s)$ starting at $f^0(s)$, $f^t(s)$ and hence $f(s)$ on $\mu_M^{-1}([0,\eps])$ do not depend on $s$. 
    \end{proof}

     The following corollary of \Cref{lem:lift} will allow us to 'extend' an (almost symplectic) $\mu-S^1$-diffeomorphism $f_1 \colon \mu_M^{-1}([0,1/2]) \to \mu_N^{-1}([0,1/2])$, for example, to $M$ by giving an extension $f_2$ only of $f_1^{1/2}\colon \mu_M^{-1}(1/2)\to \mu_N^{-1}(1/2)$ to $\mu_M^{-1}([1/2,1]) \to \mu_N^{-1}([1/2,1])$. This is not clear a priori because the homeomorphism obtained by piecing $f_1$ and $f_2$ together might not be smooth at $\mu_M^{-1}(1/2)$.\\
    We will formulate this corollary in a more general setting than in the rest of this section.
  
\begin{corollary}\label{cor:smoothing}
    Let $M$ and $N$ be connected semi-free Hamiltonian $S^1$-manifolds with proper momentum maps onto $[0,1]$, possibly with fixed points.
    For some $0<t<1$ and arbitrarily small $\eps>0$ such that $[t-\eps,t+\eps]$ is a regular interval for $M$, let $f_1$ and $f_2$ be (almost symplectic) $\mu-S^1$-diffeomorphisms 
    $$f_1 \colon \mu_M^{-1}([0,t+\eps/2]) \to \mu_N^{-1}([0,t+\eps/2]),$$
    $$f_2 \colon \mu_M^{-1}([t+\eps/2,1]) \to \mu_N^{-1}([t+\eps/2,1]).$$
    Assume that $(f_1)_{t+\eps/2}$ and $(f_2)_{t+\eps/2}$ are isotopic on $M_{t+\eps/2}$ through (almost symplectic) diffeomorphisms. Assume also that $M_t$ and hence any reduced space is simply-connected.
    
    Then
    there exists a $\mu$-$S^1$-diffeomorphism $f \colon M \to N$ that agrees with $f_1$ on $\mu_M^{-1}([0,t])$ and with $f_2$ on $\mu_M^{-1}([t+\eps,1])$.
    Moreover, in case $f_1$ and $f_2$ are almost symplectic, $sf^*\omega_N+(1-s)\omega_M$ is non-degenerate for all $s\in [0,1]$. 
  In that case, $f$ is almost symplectic, that is, also $[f^*\omega_N]=[\omega_M]$.

\end{corollary}

\begin{proof}
Denote by  ${\{\Psi'_{r}\}_{r \in [0,1]}}$  the given isotopy from $(f_1)_{t+\eps/2}$ to $(f_2)_{t+\eps/2}$.
Let $\Psi_r\colon N_{t+\eps/2}\to N_{t+\eps/2}$ be the flow defined by $\{\Psi'_{r}\}$, that is, $$\Psi_r:=\Psi'_r\circ (\Psi'_0)^{-1}.$$
First, we extend $\Psi$  to  a path of isotopies $\{F_{s}\}_{s \in[0,t+\eps/2]}$ between $(f_1)_{s}$ and $(f'_1)_{s}$ for  a $\mu-S^1$-diffeomorphism 
\begin{equation}\label{eq:f'1}
f'_1 \colon \mu_{M}^{-1}([0,t+\eps/2]) \to \mu_{N}^{-1}([0,t+\eps/2]) \text{ such that }{f_{1}}|_{\mu_{M}^{-1}([0,t])}={f'_{1}}|_{\mu_{M}^{-1}([0,t])},
\end{equation}
and $f'_1$ is almost symplectic if $f_1$ and $f_2$ are.

For $0<\delta<\eps/2$, define a monotone smooth function 
$$\rho=\rho(\delta)\colon [0,t+\eps/2]\to [0,1]$$
that equals $1$ near $t+\frac{\eps}{2}$ and $0$ on $[0,t+\eps/2-\delta]$. Under the identification $M_s\cong M_{t+\eps/2}$ resp.\ $N_s\cong N_{t+\eps/2}$, $s\in [0,t+\eps/2]$ 
obtained from the normalized gradient flow of the momentum map,
we set
$$(f_1)'_{s}:=\Psi_{\rho(s)} \circ (f_1)_{s}\colon M_s\to N_s.$$
If $f_1$, $f_2$, and the $\Psi'_{r}$s are almost symplectic, the forms 
\[
    \omega(a)=a(f_1)_s^* \Psi_{\rho(s)}^* \omega^N_s+(1-a)\omega^M_s, \quad a\in [0,1],
    \]
    are non-degenerate for $s \in [0,t+\eps/2]$. Moreover,
    this is true for
    \[
    a(f_1)_{t+\eps/2}^* \Psi_{s'}^*\omega^N_{t+\eps/2}+(1-a)\omega^M_{t+\eps/2}, \quad a\in [0,1],
    \]
    and arbitrary $s'$ in $[0,1]$. Hence, since non-degeneracy is an open condition (see \Cref{rem:almostsymplecticopen}), there is $\delta>0$ small enough so that
    \[
    a(f_1)_{s}^* \Psi_{s'}^*\omega^N_{s}+(1-a)\omega^M_{s}, \quad a\in [0,1],
    \]
    is non-degenerate for all $s\in [t+\eps/2-\delta,t+\eps/2]$ and $s'$ in $[0,1]$.
    \\
    Further, there is an isotopy $F_s(s')$, $s'\in [0,1]$, between $(f_1)_s$ and $(f_1)'_s$ given by
    \[
    F_s(s'):=\Psi_{s'\rho(s)}\circ (f_1)_s.
    \]
     So, by \Cref{lem:lift}, we can lift $(f_1)'_{s}$ to an (almost symplectic) $\mu-S^1$-diffeomorphism $f'_1$ as in \eqref{eq:f'1}.
The  maps $f'_1$ and $f_2$ differ on $\mu_M^{-1}(t+\eps/2)$ only by a map $M_{t+\eps/2}\to S^1$, but since $M_{t+\eps/2}$ is simply-connected, any such map is nullhomotopic via some homotopy $h_s$. We define $f_1''$ by
    $$f_1''(x):=h_{\rho(\mu_M(x))}\cdot f_1'(x).$$ 
 
Next,  we apply the same fact on non-degeneracy and \Cref{lem:isotopybase} to show that we can locally modify both $f_1''$ and $f_2$ to (almost symplectic) $\mu-S^1$-diffeomorphisms whose restriction to $V_{\delta}:=P_{t+\eps/2} \times (t+\eps/2-\delta,t+\eps/2+\delta)$, for some $\delta>0$, factors as 
$(p,t') \mapsto ({{f_2}|_{P_{t+\eps/2}}}(p),t')$.
See \Cref{lem:smoothing} below.
This allows us to paste the maps to get a $\mu-S^1$-diffeomorphism $f \colon M \to N$.

Moreover, if $f_1$ and $f_2$ are almost symplectic, then we get  a $\mu-S^1$-diffeomorphism $f$ such that $sf^*\omega_N+(1-s)\omega_M$ is non-degenerate for all $s\in [0,1]$. 
If the $S^1$-action on $M$ and $N$ is free, the Duistermaat-Heckman formula implies that $[f^{*}\omega_N]=[\omega_M]$. Otherwise, the equality $[f^{*}\omega_N]=[\omega_M]$ is by \Cref{lem:cohomologyclass}, saying that for two symplectic forms to be cohomologous on $M$, it is enough that they are cohomologous on   $\mu_M^{-1}([0,t+\eps/2])$ and on $\mu_M^{-1}([t+\eps/2,1])$.
\end{proof}

In the proof, we used the following lemma.
\begin{lemma}\label{lem:smoothing}
Let $(M,\omega^M,\mu_M)$ and $(N,\omega^N,\mu_N)$ be connected Hamiltonian $S^1$-manifolds; assume that the circle action is {free} and the 
momentum maps are proper and onto $[0,1]$.
    Let $f \colon M \to N$ be a(n almost symplectic) $\mu-S^1$-diffeomorphism. For any $t\in [0,1]$ there are $0<\delta<\eps$
    and a(n almost symplectic) $\mu-S^1$-diffeomorphism $g\colon M\to N$ such that
    \begin{itemize}
        \item $g$ and $f$ agree outside $V_{\eps}:=P^M_t\times (t-\eps,t+\eps)$;
        \item when restricted to $P^M_t$, $f$ and $g$ are equal to the same map $h\colon P^M_t\to P^N_t$;
        \item $g$ is of the form $g(p,t')=(h(p),t')$ on $V_{\delta}:=P^M_t\times (t-\delta,t+\delta)$.
    \end{itemize}
\end{lemma}
  Again, it is understood that $(0-\eps,0+\eps)=[0,\eps)$ and $P^M_0\times (0-\eps,0+\eps)=P^M_0\times [0,\eps)$, for example.

  We will only prove the
  case that $f$ is also almost symplectic. 
  The arguments required for the case that the maps are $\mu-S^1$-diffeomorphisms 
are included in the proof.
\begin{proof}

  For some $0<\delta<\eps$, let $\rho\colon [0,1]\to [0,1]$ be a monotone smooth function with the properties:
    \begin{itemize}
        \item $\rho(t')=t$ for all $t'\in [t-\delta,t+\delta]$,
        \item $\rho$ is the identity outside $[t-\eps,t+\eps]$.
    \end{itemize}
    Note that $\rho$ can be chosen arbitrarily close under the maximum norm to the identity map when $\eps$ is chosen to be small enough.\\
   Viewing $f$ as a smooth family $f^s$, $s\in [0,1]$, of $S^1$-equivariant diffeomorphisms $P^M_s\to P^N_s$, we define $g$ to be
    \[
    g\colon M\to N, \quad g(p,s)=(f^{\rho(s)}(p),s).
    \]
   For fixed $s$, $f_s^*\omega^N_s$ and $\omega^M_s$ are isotopic under the standard homotopy, so there is $r>0$ such that $f_s^*\omega^N_{s'}$ and $\omega^M_{s'}$ are isotopic under the standard homotopy, provided that $|s-s'|<r$. Since $s\in [0,1]$, we may choose $r$ to be universal for all $s$. So if we choose $\eps$ such that $\rho$ is closer than $r$ to the identity with respect to the maximum norm, then each $$g_s\colon (M_s,\omega^M_s)\to (N_s,\omega^N_s)$$ on the reduced spaces at level $s$ has the property that $g_s^*(\omega^N_s)$ is isotopic to $\omega^M_s$ under the standard homotopy. We are done in view of \Cref{lem:isotopybase}.
\end{proof}

The next proposition highlights the main application of rigidity in the proof of \Cref{thm:mainresult}. It will allow us to extend an isomorphism below a critical level to a level arbitrarily close to the critical level.

\begin{proposition}\label{prop:extension}
 Assume that $M$ and $N$ are connected free Hamiltonian $S^1$-manifolds with proper momentum maps onto $[0,1]$.
 The following statements hold.
 \begin{itemize}
     \item Given any $\mu-S^1$-diffeomorphism $f\colon \mu_M^{-1}([0,\eps])\to \mu_N^{-1}([0,\eps])$ for any $\eps\in (0,1)$, we find a $\mu-S^1$-diffeomorphism $M\to N$ that agrees with $f$ near the $0$-level set.
     \item Assume that $(M_0,\omega^M_{t\in [0,1]})$ is rigid (as in \Cref{def:rigid}). Given any isomorphism $f\colon \mu_M^{-1}([0,\eps])\to \mu_N^{-1}([0,\eps])$ for any $\eps\in (0,1)$, we find an isomorphism $M\to N$ that agrees with $f$ near the $0$-level set.
 \end{itemize}
\end{proposition}

In the proof of the proposition we will use Gonzales' definition of equivalence for families of symplectic forms on a compact manifold.

\begin{definition}\label{def:equivalent}\cite[Definition 1.4]{Go11}
        Let $B$ be a compact manifold. Let $\omega_t$ and $\omega'_t$, $t\in I=[t_0,t_1]$, be families of symplectic forms on $B$ such that $[\omega_t]=[\omega'_t]$ in $H^2(M;\R)$ for all $t\in I$. We say that $\omega_t$ and $\omega'_t$ are \textbf{equivalent} if there exists a smooth family $\omega_{s,t}$ of symplectic forms such that
        \begin{equation} \label{eq:equiv}
        \partial_s [\omega_{s,t}]=0,\quad \text{and} \quad \omega_{0,t}=\omega_{t}, \,\, \omega_{1,t}=\omega'_{t}
        \text{
        for all }(s,t)\in [0,1]\times I.
        \end{equation}
    \end{definition}

\begin{lemma}\label{lem:equivalent} (cf. \cite[Lemma 3.4]{Go11}.)
     Let $(B,\omega_t)$ be rigid, and let $\omega'_t$ be any family of symplectic forms on $B$ such that $[\omega_t]=[\omega'_t]$ in $H^2(M;\R)$, for all $t\in I=[t_0,t_1]$, and $\omega_{t_0}=\omega'_{t_0}$.
     \begin{enumerate}
         \item Then $\omega_t'$ is equivalent to $\omega_t$. 
         \item Furthermore, if there exists $R>t_0$ such that $\omega_t=\omega'_t$ for all $t_0\leq t\leq R$, then there exists a smooth family $\omega_{s,t}$ such that $\partial_s\omega_{s,t}=0$ for all $t_0\leq t\leq R$. 
         \item If case (2) holds, there also exists a smooth family $f_{s,t}$ of diffeomorphisms $B\to B$ such that $f_{s,t}=\id$ for $t_0\leq t\leq R-\eps$ and $f_{s,t}^*\omega_{s,t}=\omega_{t_0,t}$.
     \end{enumerate}
\end{lemma}
 Item (1) is simply the statement in \cite[Lemma 3.4]{Go11}. The proof of Item (2) is similar to the proof of \cite[Lemma 3.4]{Go11}. Item (3) is by Moser's method.
      We prove Items (2) and (3) in \Cref{app:extrafour}.\\

To prove the proposition we will also use Moser's trick. 
 Since the action is free and the manifold is not closed, we need to argue that Moser's method still works. For the application of Moser's trick in case fixed points exist, see \Cref{rem:moserwithfixedpoints}.
 
\begin{lemma}\label{lem:moserwithoutfp}
Let $f$ be an almost symplectic $\mu-S^1$-diffeomorphism from $(M,\omega^M)$ to $(N,\omega^N)$.
     The isotopy of forms $\omega(s)=sf^*\omega^N+(1-s)\omega^M$  between $f^{*}\omega^N$ and $\omega^M$ integrates to an isotopy of $\mu-S^1$-diffeomorphisms 
  $$\Psi_s\colon M\to M \text{ such that }\Psi_s^*\omega(s)=\omega(0)=\omega^M.$$
  In particular, $f(s):=f\circ \Psi_s$ is an isotopy of $\mu-S^1$-diffeomorphisms $M \to N$ such that $f(1)^*\omega^N=\omega^M$.\\
    If, moreover, $f$ is a symplectomorphism when restricted to $V:=\mu_M^{-1}([0,\eps])$ with $\eps>0$, then the isotopy $\omega(s)$ is constant on $V$, and the corresponding isotopies $\Psi_s$ respectively $f(s)$ on $M$ may be chosen to have support outside $V$ as well.
\end{lemma}
\begin{proof}
    We will treat both cases ($f$ being a symplectomorphism on $V$ or not) at the same time. We set
    $$\Omega:=\partial_s \omega(s)=f^*(\omega^N)-\omega^M.$$
    Note that $\Omega$ descends to a form $\omega'$ on $M/S^1$, since it is invariant under the $S^1$-action and $$f^*(\omega^N)(\xi)-\omega^M(\xi)=df^*(\mu_N)-d\mu_M=0.$$
    Since $f$ is almost symplectic, each $f_t \colon M_t \to N_t$ is almost symplectic, and so $0=[\omega']\in H^2(M/S^1;\R)$. If, moreover, $f$ is a symplectomorphism on $V$, then $\omega'$ vanishes on $V/S^1$, so it defines an element in the relative deRham cohomology group $H^2(M/S^1,V/S^1;\R)$. Since the inclusions $V\cong M_0\times [0,\eps]\hookrightarrow M\cong M_0\times [0,1]$ resp.\ $V/S^1\hookrightarrow  M/S^1$ are homotopy equivalences, it follows that $0=[\omega']\in H^2(M/S^1,V/S^1;\R)$. So there is a one-form $\beta'$ on $M/S^1$ with 
    $$d\beta'=\omega'$$
    that vanishes on $V/S^1$ if $\omega'$ does. The pullback $\beta$ of $\beta'$ under $\pi\colon M\to M/S^1$ is then a one-form whose differential is $\Omega$. 
    Define a vector field $X_s$ by 
    $$\beta=\omega(s)(\cdot,X_s).$$ We get that $X_s$ is a Moser vector field for $\omega(s)$, i.e., $\partial_{s}\omega(s)+\mathcal{L}_{X_s}\omega(s)=0$.
   The vector field $X_s$ leaves the momentum map $\mu_M=(1-s)\mu_M+sf^*\mu_N$ invariant, since
    $$d((1-s)\mu_M+sf^*\mu_N)(X_s)=\omega(s)(\xi,X_s)=\beta(\xi)=0.$$
   Since $\mu_M$ is proper, this implies that the flow $\{\Psi_s\}$ of $X_s$ on $M$ is well-defined. Since $X_s$ is a Moser vector field for $\omega(s)$, we have $${\Psi}_s^{*}\omega(s)=\omega^M.$$ Also, the flow preserves the momentum map and thus is a $\mu-S^1$-diffeomorphism. 
   If, moreover, $\beta$ vanishes on $V$, the flow is supported away from $V$. This completes the proof.
\end{proof}

\begin{proof}[Proof of \Cref{prop:extension}]
Consider an isomorphism (or $\mu-S^1$-diffeomorphism) $f\colon \mu_M^{-1}([0,\eps])\to \mu_N^{-1}([0,\eps])$, with $\eps\in (0,1)$.
    Using the gradient flow of $\mu$ to identify $M\cong P_0\times [0,1]$, we can extend $f^{\eps}\colon \mu_M^{-1}(\eps)\to \mu_N^{-1}(\eps)$ to a $\mu-S^1$-diffeomorphism $\mu_M^{-1}([\eps,1])\to \mu_N^{-1}([\eps,1])$. 
    By \Cref{cor:smoothing}, we can piece the latter map and $f$ to obtain a $\mu-S^1-$diffeomoprhism $$f''\colon M\to N$$ that agrees with $f$ near $\mu_M^{-1}(0)$, already showing the first item.\\
    
     For the second, now consider the two families of reduced forms: $\omega^M_t$ and $$\omega'_t:=(f'')^*\omega^N_t,$$ with $t\in I$. By the setting of $f''$ we have $\omega_0=\omega'_0$.  
     Moreover, the Duistermaat-Heckman formula $[\omega_{t}]=[\omega_{t'}]+(t'-t)e(P)$ implies that $[\omega_t]=[\omega'_t]$ for $t\in I$. 
     Hence, since by assumption $(M_0,\omega^M_{t \in [0,1]})$ is rigid, \Cref{lem:equivalent} applies. We deduce that the two families
     are equivalent in the sense of \Cref{def:equivalent},  i.e., there is a smooth family $\omega_{s,t}$ on $M_0$ such that \eqref{eq:equiv} holds.
Moreover, by the second part of the lemma, since $\omega^M_t$ and $\omega'_t$ agree for $t$ near $0$, it may be assumed that the smooth familiy $\omega_{s,t}$ does not depend on $s$ for $t$ near $0$, and furthermore, there is a smooth family of diffeomorphisms $f_{s,t} \colon M_0  \to M_0$ such that $f_{s,t}$ is the identity near $0$ and $f_{s,t}^{*}{\omega_{s,t}}=\omega_{0,t}$. In particular 
      $$f_{1,t}^{*}{\omega'_t}=f_{1,t}^{*}{\omega_{1,t}}=\omega_{0,t}=\omega^M_t.$$
      By \Cref{lem:lift}, the $2$-parameter family $f_{s,t}$ lifts smoothly in $s$ to a  smooth family of $\mu-S^1$-diffeomorphisms $g(s)\colon M\to M$ that is the identity near level $0$. The lift $g(1)$ of the $1$-parameter family $f_{1,t}$ is almost symplectic with respect to $\omega^M$ and $\omega'$. So 
     $$f':=f''\circ g(1)$$ is almost symplectic with respect to $\omega^M$ and $\omega^N$, and equals $f$ near level $0$.
     We now apply \Cref{lem:moserwithoutfp} to get an isotopy $f'(s)$ of $\mu-S^1$-diffeomorphisms $M \to N$ such that $f'(1)^{*}\omega^N=\omega^M$. Moreover, the isotopy $f'(s)$ has support away from level $0$ so $f'(1)$ equals $f$ near $\mu_M^{-1}(0)$, and is an extension as required.

\end{proof}

\section{Isomorphisms of neighborhoods of critical levels}\label{sec:neighborhoods}

Another ingredient in our proof of \Cref{thm:extending-g} is understanding the implications of having the same $*$-small fixed point data on extending a symplectomorphism of reduced spaces below the critical level to a neighbourhood of the critical level, 
as we do in \Cref{lem:psi_textremum}, \Cref{lem:connected<4}, and \Cref{lem:nonextremal-2}. We assume the following setting.

\begin{Setting}
\label{set:merged}

Let $M^1=(M^1,\omega^1,\mu_1)$ and $M^2=(M^2,\omega^2,\mu_2)$ be connected symplectic semi-free Hamiltonian manifolds of dimension six whose momentum maps $\mu_1$ and $\mu_2$ are proper and have bounded images.
     
Assume that $\lambda$ is a critical value for both $\mu_1$ and $\mu_2$, either non-extremal for both or maximal (minimal) for both.
Assume that $M^1$ and $M^2$ have the same $*$-small fixed point data at $\lambda$.
  
  If $\lambda$ is {\bf non-extremal}, assume that the connected components of the fixed point set at $\lambda$ are either points or \textbf{exceptional spheres}, that is, symplectically embedded spheres of self-intersection $-1$ in $M^{i}_{\lambda}$.\\  
If $\lambda$ is {\bf extremal} assume that the fixed point set is simply connected.
\end{Setting}

In some of the claims, we assume that the reduced space of dimension four at the critical level is a symplectic rational surface. 
This assumption  allows us to apply the 
Gromov-Seiberg-Witten-Taubes theory.

\begin{theorem} \cite{Mc96}, \cite[Conclusion in p.\ 17]{Sa13}. \label{thm:cohomologousformsarediffeomorphic}
    Let $(N^1,\omega^1)$ and $(N^2,\omega^2)$ be symplectic rational surfaces and $f\colon N^1\to N^2$ a diffeomorphism that satisfies $f^*[\omega^2]=[\omega^1]$. Then there is a symplectomoprhism $(N^1,\omega^1)\to (N^2,\omega^2)$ whose induced map on homology agrees with that of $f$.
\end{theorem}

 \subsection*{Case I: the critical value is extremal.}
If $\lambda$ is an extremal critical value, then $M^{i}_{\lambda}$ coincides with the fixed point set $F_i$ at level $\lambda$. 
 First, we describe a neighborhood of $F_i$ as an equivariant symplectic vector bundle.  
  For that, we use the construction of a symplectic form on a $D^{2n}$-bundle with structure group $\U(n)$ over a compact symplectic manifold.
\begin{noTitle}\label{nt:normalbundle}
Let $(B,\omega_B)$ be a compact, simply-connected symplectic manifold. Let $$\pi\colon D\to B$$ be a $D^{2n}$-bundle with structure group $\U(n)$ over $B$. We consider $D^{2n}$ as a subset of $\C^n$ and endow $D$ with any $S^1$-action that fixes $B$ and acts linearly fiberwise linearly; such an action always exists because the structure group is $\U(n)$, for example $\rho(t)( z_1,\hdots,z_n)=(t^{-1}z_1,\hdots,t^{-1}z_n)$.\\

 Let $\{U_1,\hdots,U_k\}$ be an open cover of $B$ such that $\phi_i\colon U_i\times D^{2n}\to \pi^{-1}(U_i)$ is a trivialization of $D$. This gives fiber inclusions $\iota_p\colon D^{2n}\hookrightarrow \pi^{-1}(p)$ for all $p\in B$ that are well-defined up to the action of $\U(n)$. In particular, we can endow $D$ with a fiber metric $\langle \cdot,\cdot \rangle$ that is just the standard Euclidean scalar product on each fiber, such that
$\langle v,v \rangle \leq 1$ for all $v\in D$.\\

 Since the disk fiber is contractible, the cohomology class of the fiber form is $0$. Hence, by a theorem of Thurston (see \cite{Th76} and \cite[Theorem 6.1.4]{MS98}), there exists a closed $2$-form $\eta$ on $D$ that restricts to the standard symplectic form on each fiber, and represents the class $0$ in $H^2(D;\R)\cong H^2(B;\R)$, i.e., its restriction to the $0$-section $B$ is exact\footnote{In general, we can not guarantee that its restriction is $0$.}. By averaging $\eta$ w.r.t. the $S^1$-action on $D$, we further have that $\eta$ is $S^1$-invariant. Moreover, for $c>0$ small enough (to be fixed from now on), the invariant closed form 
 $$\omega=\pi^*(\omega_B)+c\eta$$
 restricts to a symplectic form on $B$. Also, it is clear that $\omega$ is non-degenerate in vertical directions, i.e., on $\ker d\pi(x)$ for all $x$. It follows that there is $r>0$ sufficiently small such that $\omega$ is symplectic in a neighborhood $U$ of $B$ consisting of those $v$ with $\langle v,v\rangle \leq r$. Since $[\eta]=0\in H^2(B;\R)$, the form $\omega$ is isotopic to $\omega_B$ under the standard homotopy.\\
 
 Since $B$ is simply-connected, there is always a momentum map $\mu$ for this symplectic form that restricts to the standard momentum map on each fiber (because $\omega$ restricts to the standard symplectic form on the fiber), which is a disk of radius $r$ with respect to the standard Euclidean metric.\\
 Therefore, if the action is of the form $\rho(t)( z_1,\hdots,z_n)=(t^{-1}z_1,\hdots,t^{-1}z_n)$, there is a surjective, proper momentum map $$\mu\colon U\to [\lambda-\delta,\lambda]$$
 such that $\mu(B)=\lambda$ (for some $\lambda\in \R$) and $\delta={r^2}/2$. After renaming $U$ as $D$, we denote the reduced space at $t\in [\lambda-\delta,\lambda]$ by
$$D_t=\mu^{-1}(t)/S^1.$$
If $n=1$, each reduced space is diffeomorphic to $B$, and scalar multiplication in the $D^{2}$-fiber gives a fixed identification $D/S^1\cong B\times [\lambda-\delta,\lambda]$ that intertwines $\mu$ and the projection $B\times [\lambda-\delta,\lambda]\to [\lambda-\delta,\lambda]$.
\end{noTitle}

In case $\dim F_i=4$, we use the above construction with $n=1$, Weinstein's symplectic neighbourhood theorem (\Cref{thm:Weinstein}), \Cref{thm:cohomologousformsarediffeomorphic}, and Moser's method to get an equivariant symplectomorphism of neighborhoods of the maximal (minimal) level that agrees on homology with a given symplectomorphism of a level below (above) the maximal (minimal) level.
\begin{theorem}\label{thm:Weinstein}{\cite{We71}, \cite{Ca08}}.
    Let $M$ be a manifold equipped with two symplectic forms $\omega_1$ and $\omega_2$ on which a compact Lie group $G$ acts symplectically. Further, let $N$ be a compact, connected manifold also acted on by $G$ smoothly, together with a smooth, equivariant embedding $f\colon N\to M$ such that $f^*\omega_1=f^*\omega_2$.

    Then there exist two open $G$-invariant neighborhoods $U_1,U_2$ of $N=f(N)$ as well as an equivariant diffeomorphism $\psi\colon U_1\to U_2$ such that $\psi^*\omega_2=\omega_1$ and $\psi_{|N}=\text{id}_{|N}$.
    \end{theorem}
    \begin{remark}\label{rem:Weinstein}
        Often times \Cref{thm:Weinstein} is used in a slightly different setting. Suppose that two symplectic manifolds $(M_i,\omega_i)$ as well as another symplectic manifold $(N,\omega_N)$, all acted on symplectically by $G$, together with equivariant symplectic embeddings $f_i\colon N\to M_i$ are given. Suppose further that the pullback of the normal bundles of the images are equivariantly isomorphic as vector bundles. Therefore, by the usual equivariant tubular neighborhood theorem, there are open neighborhoods $V_1$ and $V_2$ of $N\subset M_1$ resp.\ $N\subset M_2$ as well as an equivariant diffeomorphism $\psi'\colon V_1\to V_2$ such that $\psi\circ f_1=f_2$. By pulling back $\omega_2$ along $\psi'$, we obtain two symplectic forms $\omega_1$ and $(\psi')^*\omega_2$ on $V_1$ that agree on $N$; therefore, by \Cref{thm:Weinstein}, there exist $G$-invariant neighborhoods $U_1$ and $U_2$ of $N\subset V_1$ as well as an equivariant symplectomorphism $\psi\colon (U_1,\omega_1)\to (U_2,(\psi')^*\omega_2)$ that is the identity on $N$. By concatenating $\psi$ and $\psi'$, we find an equivariant symplectomorphism
        $$\psi\colon (U_1,\omega_1)\to (\psi'(U_1),\omega_2) \text{ such that }\psi\circ f_1=f_2.$$
    \end{remark}
\begin{lemma}\label{lem:psi_textremum}
Let $M^1$, $M^2$ be as in \Cref{set:merged}  with momentum images $[\lambda-\eps,\lambda]$ for $\eps>0$, and assume that the critical value $\lambda$ is both maximal and the only critical value of $\mu_1$ and $\mu_2$. Assume that $F_1$ and $F_2$ are of dimension four. Let
$$f_{\lambda-\eps} \colon M^1_{\lambda-\eps} \to M^2_{\lambda-\eps}$$
be a homoemorphism  that intertwines the Euler classes.\\
Then there is $0< \delta\leq \eps$ and an equivariant homeomorphism 
$$g \colon \mu_1^{-1}([\lambda-\delta,\lambda]) \to \mu_2^{-1}([\lambda-\delta,\lambda])$$ such 
that 
$g_{\lambda-\delta}$ and $f_{\lambda-\eps}$, the latter seen as a map $M^1_{\lambda-\delta}\to M^2_{\lambda-\delta}$ using the normalized gradient flow of the momentum map, induce the same map on homology.\\
Moreoever, if the $(F_i,\omega^i)$s are symplectic rational surfaces and $f_{\lambda-\eps}$ is a symplectomorphism then there is such $g$ that is an equivariant symplectomorphism.
\\
The statement holds when $\lambda$ is minimal and $\lambda-\eps$, $\lambda-\delta$ and $[\lambda-\delta,\lambda]$ are replaced with $\lambda+\eps$, $\lambda+\delta$, and $[\lambda,\lambda+\delta]$.
\end{lemma}

\begin{proof}
  We construct $(D^i,\omega,\mu)$  as in \Cref{nt:normalbundle} with $n=1$ and $(B,\omega_B)=(F_i,\omega^{i})$. That way, we have two symplectic embeddings $(F_i,\omega_i)\to (M^i,\omega^i)$ and $(F_i,\omega_i)\to (D^i,\omega)$; the latter is \textbf{not} the standard embedding $(F_i,\omega^{i})=(B,\omega_B)\hookrightarrow D^i$ (because $\omega$ does not necessarily restrict to $\omega_B$ on $B$), but isotopic to it as explained in  \Cref{nt:normalbundle}. Therefore, due to the equivariant symplectic tubular neighborhood theorem \Cref{thm:Weinstein} resp.\ \Cref{rem:Weinstein}, there is $\delta>0$ small enough such that
\begin{align}\label{eq:weinstein}
\mu_i^{-1}([\lambda-\delta,\lambda]) & \text{ is equivariantly symplectomorphic to }\mu^{-1}([\lambda-\delta,\lambda])\subset D^i \nonumber
\end{align}
 and this equivariant symplectomorphism is isotopic to the identity on $F_i$.\\
Thus, $M_{\lambda-\eps}^{i}$ is diffeomorphic to $D_{\lambda-\eps}^{i}$.
   Moreover, using the identifications $D^{i}/S^1\cong F_i\times [\lambda-r,\lambda]$, we get a diffeomorphism from $D^{i}_{\lambda-\eps}$ to $F_i=M_{\lambda}^{i}$.
   Hence, and by the setting of $f_{\lambda-\eps}$, 
   we have a  homeomorphism (diffeomorphism, if $f_{\lambda-\eps}$ is a diffeomorphism)
     $$\psi' \colon F_1\to D^{1}_{\lambda-\eps}\to D^{2}_{\lambda-\eps} \to F_2.$$ 
     The map $\psi'$ intertwines the Euler classes of the normal bundles of $F_1$ and $F_2$ in $M^1$ and $M^2$, since these are just the Euler classes of the principal $S^1$-bundles over $D^{1}_{\lambda-\eps}$ and over $D^{2}_{\lambda-\eps}$ under the identifications 
     $F_i\to D^{i}_{\lambda-\eps}$. 
     So, by \Cref{lem:Eulerclasses}, $\psi'$ lifts to an equivariant  isomorphism (smooth, if $f_{\lambda-\eps}$ is a diffeomorphism) $g$ between the normal bundles and therefore to an equivariant homeomorphism (diffeomorphism) between the neighborhoods $\mu_{i}^{-1}([\lambda-\delta,\lambda])$ of $F_i$.\\
     
     Moreover, if $f_{\lambda-\eps}$ is a symplectomorphism, it follows that $\psi'$ also preserves the cohomology classes of the symplectic forms on $F_1$ and $F_2$. Indeed, by the DH formula \eqref{eq:dh}, these are determined by the Euler class $e(P)$ and the cohomology classes of the symplectic forms on $M^1_{\lambda-\eps}$ resp.\ $M^2_{\lambda-\eps}$, which are intertwined by $f_{\lambda-\eps}$.  If the $(F_i,\omega^i)$s are symplectic rational surfaces, then, by  \Cref{thm:cohomologousformsarediffeomorphic}, there is a symplectomorphism $\psi_{\lambda}\colon F_1\to F_2$ that acts in the same fashion as $\psi'$ on homology. In particular, $\psi_{\lambda}$ intertwines the Euler classes of the normal bundles of $F_i$ in $M^i$ and therefore lifts to an equivariant bundle isomorphism $\psi\colon D^1\to D^2$ by \Cref{lem:Eulerclasses}. Since we can identify $D^i$ with $\mu_i^{-1}([t,\lambda])$ as Hamiltonian $S^1$-manifolds for $t$ sufficiently close to $\lambda$, it is only left to show that $\psi\colon D^1\to D^2$ (after possibly shrinking $D^1$ and $D^2$) can be isotoped into an equivariant symplectomorphism.\\

     The homotopy $\omega(s):=s\omega^1+(1-s)\psi^*\omega^2$ does not degenerate on $TD^1_{|F_1}$. This is since $\omega^1$ and $\psi^*\omega^2$ agree on $F_1$, and for a vertical vector $v\in TD^1_{|F_1}$ we have that $\omega^1(v,i\cdot v)>0$, where $i\in S^1$ is the imaginary unit. Since $\psi_*(v)$ is also vertical in $TD^2_{|F_2}$ and $\psi_*$ commutes with the circle action, we obtain similarly that $\omega^2(\psi_*(v),i\cdot \psi_*(v))>0$, yielding that $\omega(s)$ is indeed non-degenerate.
     Therefore, $s\omega^1+(1-s)\psi^*\omega^2$ does not degenerate near $F_1$, so it defines an isotopy near $F_1$. This implies that there is $0<\delta'\leq\delta$ such that $\mu_1^{-1}([\lambda-\delta',\lambda])$ and $\mu_2^{-1}([\lambda-\delta',\lambda])$ are equivariantly symplectomorphic, using Moser's method as in \Cref{rem:moserwithfixedpoints}. This completes the proof.

\end{proof}

\begin{remark}\label{rem:moserwithfixedpoints}
    If the $S^1$-action on $M$ has a fixed point and the momentum map is proper, it is standard that Moser's method works as usual. We sketch the proof. Consider a  smooth one-parameter family $\omega_t$ of symplectic forms, $t\in [0,1]$, of the form $\omega_t=t\omega_1+(1-t)\omega_0$ for $S^1$-invariant, cohomologuous symplectic forms $\omega_0$ and $\omega_1$. Assume that $\omega_0$ and $\omega_1$ admit momentum maps $\mu_0$ and $\mu_1$ that agree on the components of the fixed point set. This gives rise to an $S^1$-invariant time-dependent vector field $X_t$ such that 
\begin{equation}\label{eq:moser}
    \frac{d}{dt}\omega_t+\mathcal{L}_{X_t}\omega_t=0,
    \end{equation}
    as usual. Moreover, if $\frac{d}{dt}\omega_t\equiv 0$ on $U:=\mu^{-1}(-\infty,r)$  for a certain level $r\in \mu(M)$ such that $H^1(U;\R)=0$, then it may be assumed that $X_t\equiv 0$ on $U$ as well; this is due to the fact that $\Omega:=\frac{d}{dt}\omega_t=\omega_1-\omega_0$ is not only $0$ in $H^2(M;\R)$, but even represents the $0$-class in $H^2(M,U;\R)$. Indeed, the long exact relative cohomology sequence
    \[
    \hdots \to H^1(M;\R)\to H^1(U;\R)\to H^2(M,U;\R)\to H^2(M;\R) \to \hdots
    \]
    tells us that $\Omega$ is in the image of $H^1(U;\R)\to H^2(M,U;\R)$, but $H^1(U;\R)=0$ by assumption. Hence, we may write $\Omega=d\beta$ for an invariant one-form $\beta$ vanishing on $U$, implying that $X_t\equiv 0$ on $U$.\\
    Now, around each point in $M$, the flow $\Psi_t$ of $X_t$, defined by $\frac{d}{d_t}{\Psi_t}=X_t \circ \Psi_t$, is still defined for small $t$. The flow $\Psi_t$ necessarily preserves fixed point components. By \eqref{eq:moser}, $\Psi_t^{*}\omega_t=\omega_0$.
  Denote $\mu_t:=t \mu_1+(1-t)\mu_0$. Note that $\mu_t$ is a momentum map for $\omega_t$ and that the $\mu_t$s agree on the components of the fixed point set.
   We claim that $$\Psi_t^*\mu_t=\mu_0.$$ Indeed, $\Psi_t^*\mu_t$ is \textbf{a} momentum map for $\omega_0(=\Psi_t^{*}\omega_t)$, meaning that $\Psi_t^*\mu_t$ and $\mu_0$ agree up to a shift by a constant. Since $\Psi_t$ leaves fixed point components invariant, the constant has to be zero.\\
    It follows that the vector field $X_t$ leaves the momentum map levels invariant.
   Thus, since we assumed that the momentum map is proper, the flow integrates on $M$; we get an equivariant diffeomorphim $\Psi_1$ such that $\Psi_1^{*}\omega_1=\omega_0$.
Note that if there are no fixed points, $X_t$ might not leave the momentum map levels invariant.
\end{remark}
In case the extremal submanifolds are spheres, 
we let $n=2$ in \cref{nt:normalbundle}. We need further preliminary observations about Hirzebruch surfaces.
\begin{noTitle}\label{nt:Hirzebruch}
Let $E\to S^2$ be a complex vector bundle of rank $2$. It is well-known that there is an integer $k$ such that $E\cong S^3\times_{S^1} (\C_{k}\oplus \C_{0})$ as complex bundles, where $S^1$ acts on $S^3\subset \C^2$ via the standard diagonal action and on each $\C$-summand according to the index. We endow $E$ with the anti-diagonal, fiberwise $S^1$-action. Using \cref{nt:normalbundle} for $n=2$ and $(B,\omega_B)=(S^2,\omega_{S^2})$, where $\omega_{S^2}$ is some  symplectic form on $S^2$, we can endow a neighborhood $U$ of the $0$-section of $E$ with an $S^1$-invariant symplectic form $\omega^E$ that restricts to the standard symplectic form on each disk fiber. For the momentum map $\mu_E\colon U\to \R$ such that the $0$-section of $E$ is sent to $\lambda$, we let $t<\lambda$  such that $\mu^{-1}([t,\lambda])\subset U$. Therefore, $\mu^{-1}([t,\lambda])$ is a closed tubular neighborhood of $E_0$, and $\mu^{-1}(t)$ is the total space of an $S^3$-bundle over $S^2$, and the circle acts on the $S^3$-fiber, only. It follows that the reduced space $E_{t}=\mu^{-1}(t)/S^1$ is an $S^2$-bundle over $S^2$ with structure group $\SO(3)$. Since $\pi_1(\SO(3))=\Z_2$, $E_t$ is diffeomorphic to either $H_0:=S^2\times S^2$, the total space of the trivial bundle, or  $H_1:=\C\PP^2\# \overline{\C\PP^2}$ (the blowup of $\C \PP^2$), which is the total space of the non-trivial $S^2$-bundle over $S^2$. In particular, $E_t$ is simply-connected.\\

    Now, since $E\cong L_k\oplus \underline{\C}$, $\underline{\C}$ being the trivial $\C$-bundle and $L_k=S^3\times_{S^1} \C_k$, there is a section
    \begin{equation}\label{eq:section}
        s'\colon S^2\to \mu^{-1}(t)
    \end{equation}
    coming from the section of the trivial $\C$-bundle. 
    We will use the following facts on the section in the proof of the next lemma.
    \begin{itemize}
        \item The Euler class $e$ of the $S^1$-bundle $S^1\to \mu^{-1}(t) \overset{\pi}{\to} E_{t}$  evaluated on the image of $s\colon S^2\overset{s'}{\to} \mu^{-1}(t)\overset{\pi}{\to} E_{t}$ is $0$.
        \item Therefore, $\int_{s(S^2)}\omega^E_{t'}$ does not depend on $t'\in [t,\lambda]$ due to the DH formula, \cref{eq:dh}.
        \item The normal bundle of $s(S^2)$ in $E_{t}$ can be identified with the complex line bundle $S^3\times_{S^1} (\{0\}\oplus \C_k)$, which, with the orientations on base and fiber induced by the symplectic form $\omega^E_t$, has Chern class $-k$. This is because the tautological bundle $S^3\times_{S^1} (\{0\}\oplus \C_1)$ has Chern class $-1$.
        \item The image $s(S^2)$ and the $S^2$-fiber $F$ generate $H_2(E_{t};\Z)$. This can be seen by looking at the Serre spectral sequence of the bundle, which collapses at the second page because neither base nor fiber have odd integer cohomology; therefore we have a short exact sequence $0\to H_2(F;\Z)\to H_2(E_t;\Z)\to H_2(S^2;\Z)\to 0$ for which the map induced by $s\colon S^2\to E_t$ on $H_2$ is a split.
        \item As always, we give any  symplectic submanifold of $E_t$ (including $E_t$ itself) the orientation induced by the symplectic form $\omega^E_t$, and take intersection numbers of homology classes in $H_2(E_t)$ with respect to these orientations. Then, the intersection number of $s(S^2)$ and any $S^2$-fiber is $1$, and the intersection number of any $S^2$-fiber with itself is $0$, because the fiber has trivial normal bundle. The intersection number of $s(S^2)$ with itself is the Chern class of its normal bundle and hence $-k$.
    \end{itemize}
\end{noTitle}

\begin{figure}[h]
	\begin{center}
		\begin{tikzpicture}
			\filldraw[black] (0,0) circle (2pt);
			\filldraw[black] (6,0) circle (2pt);
			\filldraw[black] (2,2) circle (2pt);
			\filldraw[black] (0,2) circle (2pt);
			\draw[green, thick] (0,0) -- (6,0);
			\draw[blue, thick] (6,0) -- (2,2);
			\draw[red, thick] (2,2) -- (0,2);
			\draw[gray, thick] (0,0) -- (0,2);
		\end{tikzpicture}
	\end{center}
 \caption{The momentum image of a $T^2$-action on a Hirzebruch surface $E_t$. The red line corresponds to $s(S^2)$, and this sphere has self-intersection $-k$ if the blue line is going in direction $(k,-1)$. The sphere corresponding to the green line has self-intersection $k$, where the blue line represents the fiber $F$.}
 \end{figure}
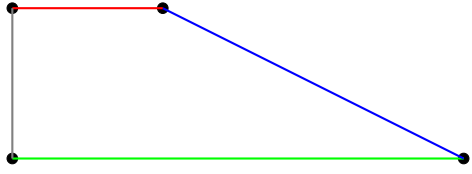

\begin{lemma} \label{lem:connected<4}
For $i=1,2$, let $(M^i,\omega^i,\mu_i)$  
be a connected semi-free, Hamiltonian $S^1$-manifold of dimension six with proper momentum map such that $\lambda$ is the only critical value and $\mu_i(M^i)=[\lambda-\kappa,\lambda]$, for some $\kappa>0$ (the same $\lambda$ and $\kappa$ for $i=1,2$). Assume that the corresponding maximal fixed point set $F_i$ is a sphere. Then the following hold.
\begin{enumerate}
    \item 
    For any $t\in [\lambda-\kappa,\lambda)$, the group $\mathcal{G}_t$ of orientation-preserving homeomorphisms of $M^1_t$ that preserve the Euler class of the bundle $S^1\to \mu_1^{-1}(t)\to M^1_t$ is connected. Consequently, the group $\mathcal{G}^t$ of orientation-preserving equivariant homeomorphisms $\mu_1^{-1}(t)\to \mu_1^{-1}(t)$ is connected.
    \item  If there exists an orientation-preserving equivariant homeomorphism $g\colon \mu_1^{-1}(\lambda-\kappa)\to \mu_2^{-1}(\lambda-\kappa)$, then, after rescaling the symplectic form $\omega^1$ by the factor $\omega^2(F_2)/\omega^1(F_1)$, neighborhoods of $F_1$ and $F_2$ are equivariantly symplectomorphic.
    \item There exists $\theta>0$ such that for any $t\in [\lambda-\theta,\lambda)$, the symplectomorphism group of $M^1_t$ is connected. 
\end{enumerate}
A similar statement holds when $\lambda$ is minimal.
\end{lemma}

\begin{proof}
    The normal bundle $E^i$ of $F_i$ is equipped with a fiberwise $S^1$-action, and therefore is a complex bundle. By the equivariant symplectic neighborhood theorem, there is $\delta>0$ such that $\mu_i^{-1}([\lambda-\delta,\lambda])$ is equivariantly symplectomorphic to $(\mu_{E^i})^{-1}([\lambda-\delta,\lambda])$. We are therefore allowed to use the notations (additionally indexed with an $i$) and facts stated in \Cref{nt:Hirzebruch}. In particular,  $M^{i}_t$ is $E^i_t$ and hence simply-connected.
    \begin{enumerate}
    \item 
    Clearly, $\mathcal{G}_t$ and $\mathcal{G}^t$ do not depend on $t \in [\lambda-\kappa,\lambda)$, since the normalized gradient flow of $\mu_1$ gives an equivariant diffeomorphism $M_t\to M_{t'}$ for $t\neq t'\in [\lambda-\kappa,\lambda)$ preserving orientation and the Euler class, so we might as well assume that $t\in [\lambda-\delta,\lambda)$.\\
    First, we argue that the connectedness of $\mathcal{G}^t$ follows from the connectedness of $\mathcal{G}_t$. Given an equivariant self-homeomorphism $f \in \mathcal{G}^t$ of $\mu_1^{-1}(t)$, let $f_{t} \in \mathcal{G}_t$ be the induced orientation-preserving self-homeomorphism on $M^1_{t}$, preserving the Euler class of the $S^1$-bundle. Then the connectedness of $\mathcal{G}_t$ implies that there exists an isotopy $\psi_s\colon M^1_t\to M^1_t$, $s\in [0,1]$, from $\psi_0=f_{t}$ to the identity $\psi_1$ through homeomorphisms in $\mathcal{G}_t$. By \Cref{lem:lift}, this isotopy lifts to an isotopy $\psi^s\colon \mu_1^{-1}(t)\to \mu_1^{-1}(t)$ such that $\psi^0=f$ and $\psi^1$ differs from the identity only by a continuous map $M^1_t\to S^1$. But any such map is nullhomotopic because $\pi_1(M^1_{t})=0$, so there is an isotopy from $\psi^1$ to the identity through equivariant homeomorphisms in $\mathcal{G}^t$.\\
    
    Now let us show that $\mathcal{G}:=\mathcal{G}_t$ is connected.  
     For $g \in \mathcal{G}$, we need to argue that it acts trivially on homology and is thus isotopic to the identity through homeomorphisms, by \cite[Theorem 1.1]{FQ86}.
    Set $H:=M^1_t$, so that $H_2(H)\cong H_2(s(S^2))\oplus H_2(F)\cong \Z\oplus \Z$ (see \Cref{nt:Hirzebruch}). 
   As an homeomorphism $H\to H$ preserving the Euler class $e$, the map $g$ needs to preserve the kernel of $e$, considered as a homomorphism $H_2(H)\to \Z$. Thus, $g$ needs to send $(1,0)$ to $\pm (1,0)$. Also, it has to send the class $(0,1)\in \Z\oplus \Z$ to $(m,1)$, with $m \in \Z$, since otherwise $e$ is not preserved. So $g(1,0)=(1,0)$ because $g$ preserves orientation.\\
   It remains to show that $m=0$. For that, observe that the self-intersection of the class $(m,1)$ needs to be $0$ (since this is so for $(0,1)$), and that its intersection with the class $(1,0)$ has to be $1$ (since this is so for $(0,1)$). Since the self-intersection of the class $(1,0)$ is $-k$, we calculate
    \begin{align*}
    (m,1)\cdot (m,1)&=(m,0)^{2}+(m,0)\cdot (0,1)+(1,0)\cdot (0,m)+(0,1)^{2}=-m^2k+2m\overset{!}{=} 0,\\
    (m,1)\cdot (1,0)&=(m,0)\cdot (1,0)+(0,1)\cdot (1,0)=-km+1\overset{!}{=} 1.
    \end{align*}
    By the second equality, either $m=0$ or $k=0$. In the latter case, by the first equality, $2m=0$, so we are done.
        \item Assume that there is an orientation-preserving equivariant homeomorphism $g\colon \mu_1^{-1}(\lambda-\kappa)\to \mu_2^{-1}(\lambda-\kappa)$. Using the normalized gradient flow, we obtain an orientation-preserving equivariant homeomorphism (also called $g$) $\mu_1^{-1}(\lambda-\delta)\to \mu_2^{-1}(\lambda-\delta)$;  we refer to this $g$ from here  on.\\
        As explained in the beginning of the proof, we identify the equivariant normal bundle of $F_i$ in $M^i$ with $E^i=S^3\times_{S^1} (\C_0\oplus \C_{k_i})$  with the $S^1$-action given by $t\cdot [p,(z_1,z_2)]:=[p,(t^{-1}z_1,t^{-1}z_2)]$. That way, we have equivariant diffeomorphisms $\mu_i^{-1}(t)\to \mu_i^{-1}(\lambda-\delta)$. We therefore view $g$ as an equivariant homeomorphism $g\colon \mu_1^{-1}(t)\to \mu_2^{-1}(t)$, and write $g_{\lambda-\delta}$ for the induced homeomorphism $H_1:=\mu_1^{-1}(t)/S^1\to \mu_2^{-1}(t)/S^1=:H_2$ on the orbit spaces.\\
        By the functoriality of the Euler class, $g_{\lambda-\delta}$ intertwines the Euler classes $e_i$ of the $S^1$-bundles over the $H_i$, and in particular sends the kernels $H_2(s_i(S^2))$ of the homomorphisms $e_i\colon H_2(H_i)\to \Z$ into each other. 
        The normal bundle of $s_i(S^2)\subset H_i$ can be identified with the bundle $S^3\times_{S^1} (\{0\}\oplus \C_{k_i})$ (see \Cref{nt:Hirzebruch}). We claim that $k_1=k_2$, which would immediately imply that there is an isomorphism of the equivariant (smooth) normal bundles of $F_1$ and $F_2$.\\
        Indeed, the self-intersection of $s_i(S^2)$ with itself is $-k_i$. Since $g_{\lambda-\delta}$ preserves orientations and $g_{\lambda-\delta}$ sends a generator of $[s_1(S^2)]$ to a generator of $[s_2(S^2)]$, we have
        \[
        [s_1(S^2)]\cdot [s_1(S^2)]=[\pm s_2(S^2)]\cdot [\pm s_2(S^2)]=[s_2(S^2)]\cdot [s_2(S^2)]
        \]
        and hence $k_1=k_2$.\\

        Finally, due to the equivariant symplectic neighborhood theorem and Moser's characterization of compact symplectic $2$-manifolds, the existence of an isomorphism of the equivariant normal bundles of $F_1$ and $F_2$ implies that neighborhoods of $F_1$ and $F_2$ are equivariantly symplectomorphic after rescaling $\omega_1$ by $\omega_2(F_2)/\omega_1(F_1)$.
        \item The symplectomorphism groups of $H_0$ and $H_1$ endowed with any symplectic form are well studied in \cite{AM00}. Indeed, by \cite[Theoem 1.4]{AM00}, the symplectomorphism group of $H_1$ when endowed with any symplectic form is connected, so we only have to deal with the case that $H=S^2\times S^2$.
    Then, \cite[Theorem 1.1]{AM00} states that the symplectomorphism group of $(H=S^2\times S^2,\omega)$, when endowed with any symplectic form, is connected except if the form is symmetric in the sense that
    $$\omega(S^2\times \{pt.\})=\omega(\{pt.\}\times S^2).$$
    If, however, this is the case for some $(M^1_{t'},\omega^1_{t'})$, then it has to be the case for all $(M^1_t,\omega^1_t)$, $t\in [t',\lambda)$, due to the linear dependence of $\omega^1_{t'}$ with respect to $t'$ (see \Cref{eq:dh}).\\
    The symplectic volume of the fiber of the bundle $M^1_t\to S^2$ approaches $0$ as $t$ approaches $\lambda$, therefore both $\omega^1_t(S^2\times \{pt.\})$ and $\omega_t^1(\{pt.\}\times S^2)$ approach $0$, implying that the symplectic volume $\omega_t^1(S)$ of any sphere $S\subset H$ approaches $0$. But this is not the case, because $\omega^1_t(s_1(S^2))=\omega^1(s'(S^2))$ for $s'$ as in \Cref{eq:section}, and $s'(S^2)$ is homotopic to $F_1$ in $M^1$ (since this is a section of a vector bundle over $F_1$), so $\omega^1_t(s_1(S^2))=\omega^1(F_1)\neq 0$.
    \end{enumerate}
   
\end{proof}

\subsection*{Case II: the critical value $\lambda$ is not extremal}
We begin with giving a meaning to a map being the identity on or near the fixed point sets $F_i$ at level $\lambda$ in $M^i$ or their preimages $F'_{i}$ under the Morse flow.

\begin{Setting}\label{set:identitynearF}
In the situation of \Cref{set:merged}, we choose, for each isolated fixed point $p\in M^i$ at level $\lambda$, a local normal form $$\chi_i^p\colon {\ball}^6\to M^i,$$ where the ball ${\ball}^6$ centered at the origin is considered to be equipped with the standard metric, the standard symplectic form and the standard circle action with weights $(+1,\pm 1,-1)$, depending on the index of the fixed point. We may choose a metric on $M^i$ that agrees with the standard metric of ${\ball}^6$ on all $\chi_i^p({\ball}^6)$. Similarly, for two fixed spheres in $\mu_1^{-1}(\lambda)$ resp.\ $\mu_2^{-1}(\lambda)$ with isomorphic equivariant normal bundles, we may identify neighborhoods of them with the same equivariant local model, which is a $\C^2$-bundle with fiberwise $S^1$-action over $S^2$ (note that the identification of the neighborhoods does not have to be symplectic).\\
Denote the union of all these local models in $M^1$ resp.\ $M^2$ by $\mathcal{U}_1$ resp.\ $\mathcal{U}_2$, and let $\kappa>0$ be small enough such that
\begin{equation} \label{eq:kappa}
\mu_i(\mathcal{U}_i)\cap (-\infty,\lambda)=(\lambda-\kappa,\lambda).
\end{equation}
As before, we denote by $F_i'$ the preimage at some $t$ below $\lambda$ (such that there is no critical value in $[t,\lambda)$) of the set $F_i$ of fixed points in $M^{i}_{\lambda}$ under the Morse flow.
\end{Setting}
\begin{definition}\label{def:identitynearF_1}
Let $t\in (\lambda-\kappa,\lambda]$.
Assume that $f$ is a $\mu-S^1$-diffeomorphism (or equivariant homeomorphism) from $M^1$ to $M^2$ defined near level $t$ that maps $F_1$ bijectively into $F_2$ (or $F_1'$ into $F_2'$) and intertwines the isomorphism type of the equivariant normal bundles of $F_1$ and $F_2$ (or $f_{\Morse}(F_1')$ and $f_{\Morse}(F_2')$). For a fixed, shared local model of $F_1$ and $F_2$, we say that $f$ is {\bf the identity near $F_1$/near $F_1'$/on $F_1'$} if the induced map on a neighborhood of $F_1$/ of $F_1'$ or on $F_1'$ in the local model (as in \Cref{set:identitynearF}) is the identity.\\
In the sequel, when we say that $f$ is the identity near/on a set, we assume that the preliminary assumption on $f$ is satisfied.
\end{definition}

\begin{remark}\label{rem:htoh}
We will frequently make use of the following consideration. Assume that the $\mu-S^1$-diffeomorphism $f$ is the identity near $F_1'$ \textbf{in the orbit space}. Then $f$ is isotopic, through $\mu-S^1$-diffeomorphisms, to a $\mu-S^1$-diffeomorphism that is the identity near $F_1'$ (in the total spaces), and this isotopy may be chosen to have support only in a neighborhood of $F_1'$. Indeed, if $U\subset \mu_1^{-1}((\lambda-\eps,\lambda-\eps/2))$ is a neighborhood of $F_1'$ on which $f$ is the identity in the orbit spaces, then $f$ differs from the identity on $U$ only by a smooth map $h\colon U/S^1\to S^1$. Such a map is homotopic to the constant map $U/S^1\to \{1\}\subset S^1$ via a homotopy $h_t$ with $h_0(U/S^1)=\{1\}$, since $U/S^1$ may be assumed to be a tubular neighborhood of spheres and is therefore simply-connected. For some smooth function $\rho \colon \mu_1^{-1}((\lambda-\eps,\lambda-\eps/2))\to [0,1]$ that equals $1$ near $F_1'$ and has support inside $U/S^1$, we now redefine $f$ to be $h_{\rho(p)t}(\pi(p))\cdot f(p)$, where $\pi$ is the orbit map. Then $f$ is the identity near $F_1'$ (in the total space) as desired.
\end{remark}

Next, we will use the Morse flow to obtain an isomorphism of neighborhoods of the critical level that agrees on homology with a given $h$ that is a
$\mu-S^1$-diffeomorphism of levels below the critical level, assuming that $h$ is the identity near $F'_1$ at some level. 
We will also apply the following variation of \cite[Lemma 3.10]{Go11}, which is based on \cite[Lemma 3.4]{Mc09} and also \cite[Theorem 13.1]{GS89}.

\begin{lemma}\label{lem:go10} (cf. \cite[Lemma 3.10]{Go11} and \cite[Section 3.4, Item 7]{Go11}.)
Let $M^1$ and $M^2$ be connected semi-free Hamiltonian $S^1$-manifolds of dimension six with proper momentum maps onto $[\lambda-\kappa,\lambda+\kappa]$, where $\lambda$ is a common non-extremal critical value and $\kappa>0$ is such that there is no other critical value in $[\lambda-\kappa,\lambda+\kappa]$. Assume that, at level $\lambda$, $M^1$ and $M^2$ have the same amount of isolated fixed points of index $1$, the same amount of isolated fixed points of index $2$, and the same amount of fixed surfaces, all of which are spheres.\\
Suppose that there is a symplectomorphism $$\psi\colon (M^1_{\lambda},\omega^1_{\lambda})\to (M^2_{\lambda},\omega^2_{\lambda})$$ that maps $F_1$ into $F_2$ bijectively,  intertwines the isomorphism type of their equivariant normal bundles in $M^1$ resp.\ $M^2$, is the identity near $F_1$, and intertwines the Euler classes $e^1_-$ and $e^2_-$ at $\lambda$ (where $e_{-}$ is as in  \Cref{not:elambda}).\\
Then there is $\kappa>\delta>0$ and an isomorphism $$g\colon \mu_1^{-1}([\lambda-\delta,\lambda+\delta])\to \mu_2^{-1}([\lambda-\delta,\lambda+\delta])$$ such that $g$ is the identity near $F_1\subset M^1$ and the induced map of $g_{\lambda}$ on homology agrees with that of $\psi$.
    \end{lemma}

    The proof is almost the same as in \cite[Lemma 3.10]{Go11}\footnote{\cite[Lemma 3.10]{Go11} was formulated in case all fixed points at $\lambda$ having index $1$, but it was remarked in \cite[Section 3.4, Item 7]{Go11} that this also holds in the case of mixed indices.}. Indeed, the assumptions there were only used in the first paragraph of the proof in order to find the symplectomorphism $\psi$ (there, it was called $\phi$) whose existence is already assumed in our lemma.
\begin{proof}[Sketch of proof]
Since $\psi$ is the identity near $F_1$, there is a unique diffeomorphism $\psi_{\lambda-\eps}\colon M^1_{\lambda-\eps} \to M^2_{\lambda-\eps}$ that is the identity near $F_1'$ and satisfies $f_{\Morse}\circ \psi_{\lambda-\eps}=\psi\circ f_{\Morse}$. The assumption that $\psi$ intertwines $e^1_-$ and $e^2_-$ implies that $\psi_{\lambda-\eps}$ intertwines the Euler classes of $S^1\to \mu_i^{-1}(\lambda-\eps)\to M^i_{\lambda-\eps}$, hence $\psi_{\lambda-\eps}$ lifts to an equivariant diffeomorphism $\psi^{\lambda-\eps} \colon \mu_1^{-1}(\lambda-\eps)\to \mu_2^{-1}(\lambda-\eps)$, which we may assume to be the identity near $F_1'$ due to \Cref{rem:htoh}.

Using the Morse flow, we can extend $\psi^{\lambda-\eps}$ to a $\mu-S^1$-diffeomorphism
$$g'\colon \mu_1^{-1}([\lambda-\eps,\lambda+\eps])\to \mu_2^{-1}([\lambda-\eps,\lambda+\eps])$$
which is the identity near $F_1$ and descends to $\psi$ on $M^1_{\lambda}$.

After possibly shrinking the neighborhoods, $g'$ can be isotoped, through $\mu-S^1$-diffeomorphisms, to an isomorphism $g\colon \mu_1^{-1}([\lambda-\delta,\lambda+\delta])\to \mu_2^{-1}([\lambda-\delta,\lambda+\delta])$ for some $\delta>0$, using Moser's method. This works near $F_1$ since $g'$ is the identity there, and outside a neighborhood of $F_1$ as in \Cref{lem:moserwithoutfp}, because $g_{\lambda}$ is symplectic and hence $g_t$ is almost symplectic for $t$ close enough to $\lambda$.
In particular, $\psi$ and $g_{\lambda}$ induce the same map on homology.
\end{proof}

To apply the above lemma, we need the following two lemmata on symplectomorphisms of four-manifolds. We say that two symplectomorphisms $\psi,\psi'\colon M\to M$ are isotopic through symplectomorphisms if there is a smooth map $f\colon M\times [0,1] \to M\times [0,1]$ preserving the second factor such that $f_0=\psi$, $f_1=\psi'$ and each $f_t$ is a symplectomorphism.
\begin{remark}\label{rem:isotpicvspathconnected}
    It is well-known that isotopy classes of diffeomorphisms of a smooth, compact manifold are identical to $\pi_0(\text{Diff}(M))$, more precisely, for any path $\gamma\colon [0,1]\to \pi_0(\text{Diff}(M))$ there is also a path $\gamma'$ with the same starting and the same end point as $\gamma$ such that the induced map $$M\times [0,1]\to M\times [0,1],\quad (p,t)\mapsto (\gamma'(t)(p),t)$$
    is smooth. Moreover, $\gamma'$ may be chosen in such a way that $\gamma(t)$ and $\gamma'(t)$ are arbitrarily $C^{\infty}$-close, for all $t$.\\
    The same holds for the isotopy classes of symplectomorphisms, i.e., if two symplectomorphisms $\psi,\psi'$ can be connected by a path $\gamma$ through symplectomorphisms, then there is another path $\gamma'$ connecting those that actually represents an isotopy through symplectomorphisms from $\psi$ to $\psi'$. This follows by first choosing a path $\gamma''$ through diffeomorphisms from $\psi$ to $\psi'$ such that $\gamma''(t)$ and $\gamma(t)$ are close enough for all $t$ such that $\gamma''(t)$ is almost symplectomorphic, and then to apply Moser's method in the sense of \Cref{rem:Moser} on the two families $\omega_t\equiv \omega, \omega'_t=\gamma''(t)^*(\omega)$. The resulting path $\gamma'$ connecting $\psi$ and $\psi'$ then represents an isotopy as desired.
\end{remark}

\begin{Setting} \label{set:psi}
Let  $N^1$ and $N^2$ be compact, connected symplectic manifolds of dimension $4$.
Assume that in  $N^i$, for $i=1,2$, there are pairwise disjoint, exceptional  spheres $S_1^i,\hdots,S_k^i \subset N^i$ as well as a symplectomorphism $$\psi\colon N^1\to N^2 \text{ s.t. $\psi$ sends the homology class of $S^1_j$ into that of $S^2_j$ } \, \forall \, 1\leq j\leq k.$$ Also, let $A_1,A_2\subset N^1,N^2$ be finite sets of the same cardinality.
By identifying neighbourhoods of the exceptional spheres resp.\ the points in $A_1,A_2$ with the same local model, there is again meaning to saying that $\psi\colon N^1\to N^2$ is the identity on/near $A_1$ or on/near $S_1^j$. 
\end{Setting}

The first lemma says that there is a path, through symplectomorphisms, from $\psi$ to a symplectomorphism that is the identity near $A_1$. This follows from the well-known fact that the symplectomorphism group of a connected symplectic manifold acts transitively on sets of a given finite cardinality (see \cite[Theorem A]{Bo69}), and  from \cite[Proposition 7.1.22]{MS98}, which gives the path through symplectomorphisms from a symplectomorphism that intertwines $A_1$ and $A_2$ to a symplectomorphism that is the identity near $A_1$.
\begin{lemma}\label{lem:normalization}
Assume  \Cref{set:psi}.
 Let $U_1\subset N^1$ be any connected open subset and define $U_2:=\psi(U_1)$. Assume that $A_1\subset U_1$ and $A_2\subset U_2$. Then there is another symplectomorphism $\psi'\colon U^1\to U^2$ that intertwines $A_1$ and $A_2$ and an isotopy from $\psi$ to $\psi'$ through symplectomorphisms. Further, there is another isotopy $\psi'_s$ through symplectomorphisms intertwining $A_1$ and $A_2$, $s\in [0,1]$, from $\psi'_0=\psi'$ to $\psi'_1=:\psi''$ such that
  $\psi''$ is the identity near $A_1$.\\
  Both isotopies may be assumed to have compact support.
\end{lemma}
In the next lemma, we include the exceptional spheres $S^i_1,\hdots, S^i_k$.
\begin{lemma}\label{lem:normalizationspheres}
  Assume  \Cref{set:psi}.
  There is an isotopy, through symplectomorphisms, from $\psi\colon N^1\to N^2$ to another symplectomorphism $\psi'$ such that $\psi'$ is the identity near $S^1_j$ for all $1\leq j\leq k$. Further, $\psi'$ is isotopic, through symplectomorphisms, to a symplectomorphism $\psi''$ that is the identity near the $S^1_j$'s and $A_1$, and this isotopy may be chosen with support outside a neighborhood of the $S^1_j$'s.
\end{lemma}
\begin{remark}\label{rem:isotopictoinvariant}
Before we prove this, let us first establish that $\psi$ is isotopic through symplectomorphisms to a symplectomorphism that maps each $S^1_j$ into $S^2_j$, not necessarily being the identity near or even on them. This follows from \cite[page 6]{AKP24}. There they consider $\widetilde{M}_c$, a $k$-fold blowup of a compact symplectic manifold $M$ by the sizes $c=(c_1,\hdots,c_k)$; they denote by $\Sigma$ the disjoint union of the $k$ exceptional divisors and by $\mathcal{C}_c(\Sigma)$ the configuration of these spheres, that is, symplectic embeddings $\Sigma\to \widetilde{M}_c$ up to parametrization. Then they conclude that the identity component Symp$_0(\widetilde{M}_c)$ of Symp$(\widetilde{M}_c)$ acts transitively on $\mathcal{C}_c(\Sigma)$.\\

Now, in the situation of \Cref{lem:normalizationspheres}, we may view $N^2$ as $\widetilde{M}_c$, where $c=(\omega^2(S^2_1), \hdots, \omega^2(S^2_k))$, $\Sigma$ is the disjoint union of the $S^2_1,\hdots, S^2_k$ and $M$ is the blow-down along the $S^2_1,\hdots, S^2_k$. Then $\psi(S^1_1), \hdots, \psi(S^1_k)$ is another configuration in $\mathcal{C}_c(\Sigma)$, and so there is an isotopy, through symplectomorphisms, from the identity to a symplectomorphism that maps $\psi(S^1_1), \hdots, \psi(S^1_k)$ to $S^2_1,\hdots, S^2_k$.\\
Finally, concatenating $\psi$ with that isotopy yields an isotopy from $\psi$ to a symplectomorphism that intertwines the $S^i_j$.
\end{remark}
\begin{proof}[Proof of \Cref{lem:normalizationspheres}]
    As explained in \Cref{rem:isotopictoinvariant}, there is an isotopy from $\psi$ to a symplectomorphism $\tilde{\psi}$ that leaves the $S^i_j$ invariant. By \cite[Lemma 2.3]{LP04}, the group of symplectomorphisms that leave the $S^i_j$ invariant is homotopy equivalent to the group of symplectomorphisms that are $\U(2)$-linear near the $S^i_j$ \footnote{While the Lemma is stated only for one exceptional divisor $E$, it holds for multiple, since the homotopy can be chosen to have support in a small neighborhood of $E$.}, so $\tilde{\psi}$ may be assumed to act like $\U(2)$ on each connected component of a tubular neighborhood $U$ of the $S^i_j$. Since $\pi_0(\U(2))$ is trivial, we find an isotopy from $\tilde{\psi}_{|U}$ to the identity; since $\pi_1(U)=0$, this isotopy is Hamiltonian and therefore can be extended to an isotopy of the whole manifold, with support arbitrarily close to $U$. This gives the desired isotopy from $\psi$ to $\psi'$.\\

    Now, simply apply \Cref{lem:normalization} on $N^i\smallsetminus \cup_{j=1}^k S^i_j$ to obtain the isotopy from $\psi'$ to $\psi''$. This finishes the proof.
\end{proof}
We will use \Cref{lem:go10} to prove the next lemma, for which the assumption on the fixed points set at the level $\lambda$ becomes important. The main point is that if the fixed point sets $F_1$ and $F_2$ at $\lambda$ contain any fixed surfaces, a symplectomorphism $M^1_{\lambda}\to M^2_{\lambda}$ is not necessarily isotopic, through symplectomorphisms, to a symplectomorphism that sends $F_1$ into $F_2$. However, if all fixed surfaces are exceptional spheres, this is the case by \Cref{lem:normalizationspheres}.
\begin{lemma} \label{lem:nonextremal-2}
Let $M^1$, $M^2$ and $\lambda$ be as in \Cref{set:merged}, and assume that $\lambda$ is a non-extremal critical value for both.
   Let $\kappa>0$ be such that \eqref{eq:kappa} holds, and $0<\eps<\kappa$.
Consider a $\mu-S^1$-diffeomorphism
\begin{equation} 
h\colon \mu_1^{-1}([\lambda-\eps,\lambda-\eps/2])\to \mu_2^{-1}([\lambda-\eps,\lambda-\eps/2])
\end{equation}
with the property that $h_{\lambda-\eps/2}$
is the identity near $F_1'$.

    For $\delta'>0$ small enough such that there is no critical value in $(\lambda,\lambda+\delta']$ for both $\mu_1,\mu_2$, there is a $\mu-S^1$-diffeomorphism
    \begin{equation} \label{eq:g'}
    g'\colon \mu_1^{-1}([\lambda-\eps,\lambda+\delta'])\to \mu_2^{-1}([\lambda-\eps,\lambda+\delta'])
    \end{equation}
    that agrees with $h^{\lambda-\eps}$ at level $\lambda-\eps$ and is the identity near $F_1$.\\
    For any extension of $h$ to a $\mu-S^1$-diffeomorphism
    $$h\colon \mu_1^{-1}([\lambda-\eps,t])\to \mu_2^{-1}([\lambda-\eps,t]),$$
    the maps $g'_{t}$ and $h_{t}$ induce the same map on homology, for any $t \in (\lambda-\eps/2,\lambda)$.

    Further, if $[h_{\lambda-\eps/2}^*\omega_{\lambda-\eps/2}^2]=[\omega_{\lambda-\eps/2}^1]$ and if the $(M^i_{\lambda},\omega^i_{\lambda})$ are symplectic rational surfaces, then there is $\eps>\delta>0$ and an equivariant symplectomorphism 
    \begin{equation*} 
    g\colon \mu_1^{-1}([\lambda-\delta,\lambda+\delta])\to \mu_2^{-1}([\lambda-\delta,\lambda+\delta])
    \end{equation*}
    that is the identity near $F_1$ such that $g'_{\lambda-\delta}$ and $g_{\lambda-\delta}$ induce the same map on homology and so do $g'_t$ and $g_t$ for any $t\in (\lambda-\delta,\lambda)$.
\end{lemma}

\begin{proof}
By \Cref{cor:restdiffeo}, the Morse flow $M_{\lambda-\eps/2}^{i} \to M_{\lambda}^{i}$ induced by \eqref{eq:morse}  restricts to a diffeomorphism $M_{\lambda-\eps/2}^{i} \smallsetminus {F}'_i \to M_{\lambda}^{i} \smallsetminus F_i$.
On the complement of $F'_{i}$ (in the total space), we use the restriction of the Morse flow to extend $h^{\lambda-\eps/2}$ to the level $\lambda$, and then continue the flow up to level $\lambda+\delta$.
Near $F_{i}'$ and near $F_{i}$, we define the extension as the identity.
By \Cref{cor:smoothing} with $f_1$ being $h$ and $f_2$ being said extension, the resulting map $$g'\colon \mu_1^{-1}([\lambda-\eps,\lambda+\delta'])\to \mu_2^{-1}([\lambda-\eps,\lambda+\delta'])$$ is smooth at level $\lambda-\eps/2$, so $g'$ is a $\mu-S^1$-diffeomorphism. \\
    The map $g'$ is well-defined in spite of the singularities of the Morse flow at level $\lambda$, since $h$ is the identity near $F_1'$ in the {total space} at level $\lambda-\eps/2$, making $g'$ also the identity near the singularities. 
  For $t \in (\lambda-\eps,\lambda)$, the maps $g'_t$ and $g'_{\lambda-\eps}$ induce the same map on homology, and, after extending $h$, so do $h_t$ and $h_{\lambda-\eps}$, under the identification $\mu_i^{-1}([\lambda-\eps,t])\cong \mu_i^{-1}(\lambda-\eps)\times [\lambda-\eps,t]$. Since $g'_{\lambda-\eps}=h_{\lambda-\eps}$ by construction, $g'_{t}$ and $h_{t}$ induce the same map on homology.
     This proves the first part of the lemma.\\
     
   Now, assume the further assumption on $h_{\lambda-\eps/2}$, $(M_{\lambda}^i,\omega_{\lambda}^{i})$, and $F_i$.
    The map $g'$ might not be almost symplectic. Still, the diffeomorphism $$g'_{\lambda}\colon M^1_{\lambda}\to M^2_{\lambda}$$ maps the cohomology class of $\omega_\lambda^1$ into that of $\omega_\lambda^2$. 
    Indeed, 
    since for each level $t$ between $\lambda-\eps/2$ and $\lambda$ we have that
    $g'_t$ and $h_{\lambda-\eps/2}$ induce the same action on homology, the map $g'_t$ preserves the Euler classes and sends $[\omega_{\lambda-\eps/2}^1]$ to $[\omega_{\lambda-\eps/2}^2]$. So, by the DH formula \eqref{eq:dh}, $g'_t$ sends $[\omega^1_t]$ to $[\omega^2_t]$, and in particular it intertwines those classes when being restricted to a map
    $$M^1_t\smallsetminus F_1'\to M^2_t\smallsetminus F_2'.$$
   Identifying $M^i_{\lambda}\smallsetminus F_i$ with $M^i_t\smallsetminus F_i'$ using the restriction
   of the Morse flow, the map 
   $g'_{\lambda}\colon M^1_{\lambda}\smallsetminus F_1\to M^2_{\lambda}\smallsetminus F_2$ also sends $[\omega^1_{\lambda}]$ to $[\omega^2_{\lambda}]$. Since $g'_{\lambda}$ is the identity near the $F_i$, certainly $(g'_{\lambda})^*[\omega^2_{\lambda}]=[\omega^1_{\lambda}]$ after restriction to the $F_i$. Since each $F_i$ is a finite union of points and exceptional spheres, we have 
   \begin{equation*}\label{eq;MWFi}
       H^2(M^i_{\lambda})=H^2(M^i_{\lambda}\smallsetminus F_i)\oplus H^2(F_i)
   \end{equation*}
   by the Mayer-Vietoris sequence, and hence $g'_{\lambda}$ sends $[\omega^1_{\lambda}]$ on  $M^1_{\lambda}$ to $[\omega^2_{\lambda}]$ on $M^2_{\lambda}$. Similarly, $g'_{\lambda}$ intertwines $e^1_-$ and $e^2_-$, see \Cref{not:elambda}.
    
    Therefore, and since we assumed that the $(M^i_{\lambda},\omega^i_{\lambda})$s are symplectic rational surfaces, we can use \Cref{thm:cohomologousformsarediffeomorphic} to find a symplectomorphism $$\psi \colon M^1_{\lambda}\to M^2_{\lambda}$$ between the reduced spaces whose action on homology agrees with that of $g'_{\lambda}$, in particular sending the classes of the exceptional spheres in $F_1$ bijectively to the classes of exceptional spheres in $F_2$ and intertwining $e^2_-$ and $e^1_-$. By \Cref{lem:normalizationspheres}, since $F_1$ and $F_2$ are disjoint unions of exceptional spheres and finitely many points with the same cardinality by assumption, we may assume that $\psi$ intertwines $F_1$ and $F_2$ and is the identity near them.
    
    Thus, by \cite[Lemma 3.10]{Go11} and \cite[Section 3.4, Item 7]{Go11} as stated in \Cref{lem:go10}, applied to $\psi$, we get an 
    equivariant symplectomorphism 
    $$g\colon \mu_1^{-1}([\lambda-\delta,\lambda+\delta])\to \mu_2^{-1}([\lambda-\delta,\lambda+\delta])$$ as required.

Since $g_{\lambda}$ and $g'_{\lambda}$ act the same way on homology and $g$ and $g'$ are the identity near $F_1$, it follows that $g_{\lambda-\delta}$ and $g'_{\lambda-\delta}$ act the same way on homology. Indeed, by the proof of \Cref{lem:go10},  $g_{\lambda-\delta}$ is isotopic through diffeomorphisms to the map that is obtained from $g_{\lambda}$ using the Morse flow $f_{\Morse}(\delta)$ 
(going in both directions) on $M^1_{\lambda-\delta}\smallsetminus F'_1$. Since the same is true for $g'$ by our construction, their actions on homology agree on $M^1_{\lambda-\delta}\smallsetminus F_1'\cong M^1_{\lambda}\smallsetminus F_1$, and certainly their actions on $H_2(F_1')$ agree. The statement now follows from 
$H_2(M^1_{\lambda-\delta})=H_2(M^1_{\lambda-\delta}\smallsetminus F_1')\oplus H_2(F_1')$, which is again due to the Mayer-Vietoris sequence and the fact that $F_1'$ is a union of isolated points and exceptional spheres in $M^1_{\lambda-\delta}$.\\
  This implies that for  any $t\in (\lambda-\delta,\lambda)$, the maps $g_t$ and $g'_t$ act the same way on homology.
\end{proof}

\begin{corollary}\label{cor:psi_t1}
    Let $M^1$, $M^2$ and $\lambda$ be as in \Cref{set:merged}, and assume that $\lambda$ is a non-extremal critical value for both. Let $\kappa>0$ be such that \eqref{eq:kappa} holds, and $0<\eps<\kappa$.\\
    Consider a $\mu-S^1$-diffeomorphism 
    $$f \colon \mu_{1}^{-1}(-\infty,\lambda-\eps]) \to \mu_{2}^{-1}(-\infty,\lambda-\eps])$$
    such that $f_{\lambda-\eps}$ is isotopic through diffeomorphisms $M^{1}_{\lambda-\eps} \to M^{2}_{\lambda-\eps}$ to a diffeomorphism that is the identity near $F_1'$.\\
    
Then there is $\eps >\delta>0$  and a $\mu-S^1$-diffeomorphism
    \begin{equation} \label{eq:g}
    g\colon \mu_1^{-1}([\lambda-\delta,\lambda+\delta])\to \mu_2^{-1}([\lambda-\delta,\lambda+\delta])
    \end{equation}
    that is the identity near $F_1$,  such that for any $t\in (\lambda-\delta,\lambda)$ and any extension of $f$ to a $\mu-S^1$-diffeomorphism
    $$f\colon \mu_1^{-1}((-\infty,t]) \to \mu_2^{-1}((-\infty,t]),$$
    $f_{t}$ and $g_{t}$ induce the same map on homology. Moreover, 
        if $[f_{\lambda-\eps}^{*}{\omega^{2}_{\lambda-\eps}}]=[\omega^{1}_{\lambda-\eps}]$, 
        then we can require that $g$ in \eqref{eq:g} is symplectic.
\end{corollary}

\begin{proof}
    We first extend $f^{\lambda-\eps}$ to a $\mu-S^1$-diffeomorphism 
    \begin{equation} \label{eq:htotal}
h\colon \mu_1^{-1}([\lambda-\eps,\lambda-\eps/2])\to \mu_2^{-1}([\lambda-\eps,\lambda-\eps/2])
\end{equation}
with the property that $h_{\lambda-\eps/2}=h'$ by the following steps.
\begin{itemize}
    \item Since $f$ is an isomorphism, $f_{\lambda-\eps}$ pulls back $e(P^{2}_{\lambda-\eps})$ to $e(P^{1}_{\lambda-\eps})$; thus, so does each of the maps in the given isotopy.
    \item We identify ${\mu_i^{-1}([\lambda-\eps,\lambda-\eps/2])}/{S^1}$ with $M^i_{\lambda-\eps}\times [\lambda-\eps,\lambda-\eps/2]$,  
    using the flow of the gradient vector field of the momentum map. The gradient flow also identifies $e(P^{i}_{\lambda-\eps})$ with $e(P^{i}_{t})$ for all $t \in [\lambda-\eps,\lambda-\eps/2]$. Then the isotopy obtained in the first item gives a smooth family  of diffeomorphisms
    $$h_t\colon M^1_t \to M^2_t, \,\, t \in [\lambda-\eps,\lambda-\eps/2]$$ on the orbit spaces, such that $h_{\lambda-\eps/2}$ is the identity near $F_1'$ and $h_{\lambda-\eps}=f_{\lambda-\eps}$.
    Moreover,
    since $f$ is a $\mu-S^1$-diffeomorphism, $f_{\lambda-\eps}$ pulls back $e(P^{2}_{\lambda-\eps})$ to $e(P^{1}_{\lambda-\eps})$; thus, 
    $h_{t}^{*}(e(P_{\lambda-\eps}^2))=e(P_{\lambda-\eps}^1)$. Therefore, using the above identification, 
    \begin{equation}\label{eq:hep}
        h_{\lambda-\eps/2}^{*}(e(P^{2}_{\lambda-\eps/2}))=e(P^{1}_{\lambda-\eps/2}).
    \end{equation}
    \item  Applying \Cref{lem:lift}, we lift the family ${(h_t)}_{t \in [\lambda-\eps,\lambda-\eps/2]}$ to a $\mu-S^1$-diffeomorphism
    $h$ on the total space, as in \eqref{eq:htotal}.
    \end{itemize}
    If $[f_{\lambda-\eps}^{*}{\omega^{2}_{\lambda-\eps}}]=[\omega^{1}_{\lambda-\eps}]$, then, by \eqref{eq:hep} and the DH formula \eqref{eq:dh}, we have $[h_{\lambda-r/2}^*\omega_{\lambda-r/2}^N]=[\omega_{\lambda-r/2}^M]$. 
   Now, by \Cref{lem:nonextremal-2} and \Cref{rem:htoh}, we have a $\mu-S^1$-diffeomorphism $g$ as required, and, moreover, if the further assumption on $f_{\lambda-\eps}$, $(M_\lambda^i,\omega_\lambda^i)$ holds, then we can further require that $g$ is symplectic.

\end{proof}
 
\section{Extending an isomorphism over a critical level}\label{sec:extend}

In this section we prove \Cref{thm:extending-g}.
 Let $(M,\omega,\mu)$ be a connected semi-free Hamiltonian $S^1$-manifold of dimension six; assume that the momentum map $\mu$ is proper and its image is bounded. 
    Consider a critical value $\lambda$ of $\mu$.
We say that a regular value (or level) $t$ of the momentum map is {\bf right below} the critical value $\lambda$ if there is no critical value in $[t,\lambda)$. 
    Assume that $\lambda$ is not extremal.
    Denote by $\mathcal{D}$  the set of homology classes of degree $2$ over $\Z$ in a reduced space at a regular value $t$ right below $\lambda$ that correspond to the spheres in $F'$ that are sent by the Morse flow to isolated fixed points of index $2$ at level $\lambda$. It follows from the definition of $f_{\Morse}$ in \S \Cref{rem:criticalvalue} that $\mathcal{D}$ is well defined: does not depend on the regular value right below $\lambda$.

\begin{proposition}\label{prop:characterizationD}
For $i=1,2$, let $(M^i,\omega^i,\mu_i)$ and $\lambda$ be as in  \Cref{set:intro}; assume that $\lambda$ is non-extremal, the only critical value of $\mu_i$ and that $M^1$ and $M^2$ have the same $*$-small fixed point data at $\lambda$. Consider an isomorphism 
 $$f \colon \mu_1^{-1}((-\infty,\lambda-r]) \to \mu_2^{-1}((-\infty,\lambda-r])$$
 with $\lambda-r$ right below $\lambda$.
Then $f$ 
    maps $\mathcal{D}_1$ into $\mathcal{D}_2$. 
\end{proposition}

Note that the proposition holds in a more general setting than that of \Cref{thm:extending-g}: we do not need to assume that the fixed point set at level $\lambda$ is a finite set of points;it might include fixed surfaces of arbitrary genus.

For $\eps>0$ small enough, $\mathcal{D}$ consists of
exceptional classes in $(M_{\lambda-\eps},\omega_{\lambda-\eps})$ on which the
evaluation of $\omega_{\lambda-\eps}$ equals $\eps$.
To prove the proposition, we will show that $\mathcal{D}$ equals the set of the exceptional classes with this property.
Recall that a homology class of degree $2$ over $\Z$ in a symplectic manifold is {\bf exceptional} if it is represented by a symplectically embedded sphere and its self-intersection number is $-1$. 
 We will use the following notations.

\begin{Notation} \label{not:E'}
    Let $\eps>0$.
    Denote by $$\mathcal{E}^{\eps}$$ the set of exceptional classes in $H_{2}(M_{\lambda-\eps};\Z)$ w.r.t.\ the symplectic form $\omega_{\lambda-\eps}$. 
Denote by $$\mathcal{E}_{\min}^{\omega_{\lambda-\eps}}$$ the set of exceptional classes of minimal size in $(M_{\lambda-\eps},\omega_{\lambda-\eps})$.
\end{Notation}

The gradient flow of the momentum map and the Duistermaat-Heckman Theorem allow us to understand the set $\mathcal{E}^{\eps}$ and a subset of it for $\eps>0$ small enough.
\begin{lemma}\label{lem:E'}
Let $(M,\omega,\mu)$ be a connected semi-free Hamiltonian $S^1$-manifold of dimension six; assume that the momentum map $\mu$ is proper and its image is bounded.
Let $\lambda$ be a non-extremal critical value of the momentum map. 
Assume that for any $t>0$ such that $\lambda-t$ is a regular value right below $\lambda$, the symplectic reduced space $(M_{\lambda-t},\omega_{\lambda-t})$ is a symplectic rational surface.

\begin{enumerate}
\item 
The set $\mathcal{E}^{t}$ of exceptional classes, as a set of classes in $H_{2}$ of the diffeomorphism type of $M_{\lambda-t}$, does not depend on $t>0$ as long as $\lambda-t$ is right below $\lambda$; we denote it $\mathcal{E}$.
   \item The subset ${\mathcal{E}^{\eps}}'$
   of $\mathcal{E}^{\eps}$ of exceptional classes such that
    \begin{itemize}
        \item[(*)] the evaluation of $\omega_{\lambda-t}$ on the class equals $t$ for all $0<t\leq\eps$, 
 \end{itemize}
   as a set of classes in $H_{2}$ of the diffeomorphism type of $M_{\lambda-\eps}$, 
   does not depend on $\eps>0$ as long as $\lambda-\eps$ is right below $\lambda$; we denote it $\mathcal{E}'$.
    \item The Euler class $e(P)$ of the principal $S^1$-bundle over a reduced space $M_{s}$ for $s$ right below $\lambda$ evaluates on every class in $\mathcal{E}'$ with value $-1$.
    \item For an exceptional class $E$ in $\mathcal{E}$, if $\omega_{\lambda-t_{E}}(E)<t_{E}$ for some $t_E>0$ s.t.\ $\lambda-t_E$ is right below $\lambda$ then there exists $0<\eps<t_E$ such that $\omega_{\lambda-t}(E)\geq t$ for every $0<t\leq \eps$.
    \item 
    For $\eps>0$ small enough, the set $\mathcal{E}'$ is a subset of the set $\mathcal{E}_{\min}^{\omega_{\lambda-\eps}}$ and each two disjoint classes in $\mathcal{E}'$ have intersection number zero.

\end{enumerate}

\end{lemma}

We will also use the description of the set of exceptional classes of minimal symplectic size in blowups of $\C \PP^2$ by \cite{KK17}. In particular, we will apply  the following characterization of exceptional classes, which follows from McDuff’s “C1 lemma” \cite[Lemma 3.1]{Mc90}, Gromov’s compactness theorem \cite[1.5.B]{Gr85}, and the adjunction formula \cite[theorem 2.6.4]{MS12}. 

\begin{lemma} \label{lem:expchar} \cite[Lemma 2.12]{KK17}.
For a homology class $E$ of self-intersection $-1$ in
$H_2(\C \PP^2 \# k \overline{\C \PP^2};\Z)$, if 
there exists a blowup form $\omega$ such that the class $E$ is represented by an embedded $\omega$-symplectic sphere then for every blowup form $\omega$, the class $E$ is represented by an embedded $\omega$-symplectic sphere.
\end{lemma}

\begin{proof}[Proof of \Cref{lem:E'}]

The gradient flow of the momentum map (w.r.t.\ some invariant metric) gives an equivariant diffeomorphism between $M_{\lambda-t}$ and $M_{\lambda-t'}$ for $0<t<t'$ 
such that there is no critical value in $[\lambda-t',\lambda)$. 
If the diffeomorphism type is $S^2 \times S^2$ or $\C \PP^2$, the set $\mathcal{E}^{t}$ is empty, and items (1)--(5) hold in the empty sense. So we assume that the diffeomorphism type is $\C \PP^2 \# k \overline{\C \PP^2}$ and $k \geq 1$.
\begin{enumerate}
\item 
The gradient flow of the momentum map sends the blowup form $\omega_{\lambda-t}$ to some blowup form $\omega'_{\lambda-t'}$ on $M_{\lambda-t'}$. Moreover, it sends an embedded $\omega_{\lambda-t}$-symplectic sphere of self-intersection $-1$ to an embedded $\omega'_{\lambda-t'}$-symplectic sphere of self-intersection $-1$.
Item (1) now follows from \Cref{lem:expchar}.

\item Item (2) follows from item (1) and the definition of ${\mathcal{E}^{\eps}}'$.

\item The equivariant diffeomorphism between $M_{\lambda-t}$ and $M_{\lambda-t'}$, given by the gradient flow of the momentum map, allows us to identify $e(P_{\lambda-t})$ and $e(P_{\lambda-t'})$ and denote it by $e(P)$.
The Duistermaat-Heckman formula in \cref{eq:dh},
and property (*) of the classes in $\mathcal{E}'$ imply item (3).

\item  Let $E$ be a class in $\mathcal{E}$. Assume that $\omega_{\lambda-t_{E}}(E)<t_{E}$ for some $t_E>0$ such that 
there is no critical value in $[\lambda-t_E,\lambda)$. 
 By \eqref{eq:dh}, for $0<t<t_E$,
   $$\omega_{\lambda-t}(E)=\omega_{\lambda-t_E}(E)+(t_E-t)e(P)(E).$$ The number $e(P)(E)$ is an integer. If it is not negative, it is clear that, eventually, $\omega_{\lambda-t}(E)\geq t$. If it is negative, then there exists $t>0$ such that $\omega_{\lambda-t}(E)=0$, which is impossible since $E$ is represented by an embedded symplectic sphere.

   \item  
   If $\mathcal{E}'$ is empty, item (5) holds in the empty sense. So assume that $\mathcal{E}'$ is not empty.
   By \cite[Theorem 1.4]{KK17}, if $k \geq 3$, we can assume, up to a change of basis of $H_{2}$, that the vector $(a^{\eps},b^{\eps},\delta^{\eps}_{1},\ldots,\delta^{\eps}_{k})$ encoding\footnote{A vector $(\lambda,\delta_1,\ldots,\delta_k)$ in $\R^{1+k}$ \emph{encodes} a blowup form $\omega$ if $\frac{1}{2\pi} \langle [\omega],L \rangle =\lambda$ and $\frac{1}{2\pi} \langle [\omega],E_i \rangle =\delta_i$ for $i=1,\ldots,k$, where $\langle \cdot,\cdot \rangle$ is the pairing between cohomology and homology on $\C \PP^2 \# k \overline{\C \PP^2}$.} the blowup form $\omega_{\lambda-\eps}$ is reduced. This means that $\delta^{\eps}_{1}+\delta^{\eps}_{2}+\delta^{\eps}_{3} \leq \alpha^{\eps}$ and $\delta^{\eps}_1 \geq \delta^{\eps}_{2} \geq \ldots \geq \delta^{\eps}_{k} > 0$.  Therefore, by \cite[Lemma 3.10]{KK17}, if $k \geq 3$, the class $E_k$ is in the set $\mathcal{E}^{\omega_{\lambda-\eps}}_{\min}$ of exceptional classes of minimal size.  If $k\in \{1,2\}$, the set $\mathcal{E}$ is finite, see \cite{De80}. Therefore, by item (4), applied to $E_k$ if $k \geq 3$ and to the finitely many classes in $\mathcal{E}$ if $k \in \{1,2\}$, the set $\mathcal{E}'$ is a subset of $\mathcal{E}^{\omega_{\lambda-\eps}}_{\min}$ for $\eps>0$ small enough. It remains to show that for $\eps>0$ small enough, two disjoint classes in $\mathcal{E}^{\omega_{\lambda-\eps}}_{\min}$ have intersection number zero. 
   
 Since we assumed that $\mathcal{E}'$ is not empty, 
 the minimal size of an exceptional class converges to $0$ as $\eps$ goes to $0$. 
 If $k=1$, 
 $$\mathcal{E}^{\omega_{\lambda-\eps}}_{\min}=\mathcal{E}=\{E_1\};$$ 
 if $k=2$,  the set of exceptional classes is $\{E_1,E_2,L-E_1-E_2\}$ \cite{De80}. So if $k=2$, 
the minimal size equals either $\delta^{\eps}_2$ or $\alpha^{\eps}-\delta^{\eps}_1-\delta^{\eps}_2$.
By \Cref{cor:restdiffeo1},
the symplectic volume $\frac{1}{2}({\alpha^{\eps}}^2-{\delta^{\eps}_1}^2-{\delta^{\eps}_2}^2)=(\alpha^{\eps}-{\delta^{\eps}_1})(\alpha^{\eps}-\delta^{\eps}_2)-\frac{1}{2}(\alpha^{\eps}-\delta^{\eps}_1-\delta^{\eps}_2)^2$  of the reduced space  $(M_{\lambda-\eps},\omega_{\lambda-\eps})$ approaches the symplectic volume of the reduced space $(M_{\lambda},\omega_{\lambda})$. 
 Therefore, we have
 $$\alpha^{\eps} \text{ and }\alpha^{\eps}-\delta^{\eps}_1\text{ do not approach }0 \text{ as }\eps \to 0.$$
 Thus for $\eps>0$ small enough, we cannot have $\delta^{\eps}_2=\frac{\alpha^{\eps}-\delta^{\eps}_1}{2}=(\alpha^{\eps}-\delta^{\eps}_1-\delta^{\eps}_2)$. Therefore, by \cite[Remark 3.15]{KK17}, 
  $$\mathcal{E}^{\omega_{\lambda-\eps}}_{\min}=\{E_2\} \text{ or }\mathcal{E}^{\omega_{\lambda-\eps}}_{\min}=\{E_1,E_2\} \text{ or }\mathcal{E}^{\omega_{\lambda-\eps}}_{\min}=\{L-E_1-E_2\}.$$
If $k \geq 3$, 
the minimal size equals $\delta^{\eps}_k$ \cite[Lemma 3.10]{KK17}. So 
 $$\delta^{\eps}_k \to 0 \text{ as } \eps \to 0.$$ Since the symplectic volume $\frac{1}{2}({\alpha^{\eps}}^2-{\delta^{\eps}_1}^2-\cdots-{\delta^{\eps}_k}^2)$ of the reduced space  $(M_{\lambda-\eps},\omega_{\lambda-\eps})$ approaches the symplectic volume of the reduced space $(M_{\lambda},\omega_{\lambda})$, 
 we have
 $$\alpha^{\eps} \text{ does not approach }0 \text{ as }\eps \to 0.$$ Similarly, since the symplectic volume $\frac{1}{2}({\alpha^{\eps}}^2-{\delta_1^{\eps}}^2-\cdots-{\delta^{\eps}_k}^2)$ equals $(\alpha^{\eps}-\delta^{\eps}_1)(\alpha^{\eps}-\delta^{\eps}_2)-\frac{1}{2}[(\alpha^{\eps}-\delta^{\eps}_1-\delta^{\eps}_2)^2-{\delta^{\eps}_3}^2-\cdots-{\delta^{\eps}_k}^2]$, 
 $$\alpha^{\eps}-\delta^{\eps}_1 \text{ does not approach }0 \text{ as }\eps \to 0.$$ 
 So, if $k \geq 3$, we can assume that $$\delta^{\eps}_k<\min \Big\{\frac{\alpha^{\eps}}{3},\frac{\alpha^{\eps}-\delta^{\eps}_1}{2}\Big\}.$$
Therefore, by \cite[Theorem 3.12]{KK17},
$$\mathcal{E}^{\omega_{\lambda-\eps}}_{\min}=\{E_{j+1},\hdots,E_k\} \text{ or }\mathcal{E}^{\omega_{\lambda-\eps}}_{\min}=\{L-E_1-E_2,E_3,\hdots,E_k\},$$
where $j$ is the smallest non-negative integer for which $\delta_{j+1}^{\eps} = \ldots  \delta_{k}^{\eps}$.
In all the above options for $\mathcal{E}^{\omega_{\lambda-\eps}}_{\min}$, the intersection number of disjoint classes is zero.

   \end{enumerate}
\end{proof}

\begin{remark}\label{rem:divisors}
    Let $C_1,\hdots,C_\ell$, for $\ell \geq 1$, be pairwise disjoint exceptional spheres of minimal area representing different classes in $\mathcal{E}'$ in the rational symplectic surface $(M^i_{\lambda-\eps},\omega^i_{\lambda-\eps})$. Then $(M^i_{\lambda-\eps},\omega^i_{\lambda-\eps})$ is a $k$-blowup of $(\CP^2,\lambda \omega_{\operatorname{FS}})$ with exceptional divisors in the classes $(E_1,\hdots,E_k)$, and 
    if $\eps>0$ is small enough, then,  
    by item (5) of \Cref{lem:E'} and its proof,  the set $\{[C_1],\hdots,[C_{\ell}]\}$ is a subset of either the set of classes of the exceptional divisors of minimal area or of $\{L-E_1-E_2,E_3,\hdots,E_k\}$. In the latter case $(M^i_{\lambda-\eps},\omega^i_{\lambda-\eps})$ is symplectomorphic to a $(k-1)$-blowup of $S^2 \times S^2$ with exceptional divisors in the classes $(\tilde{E}_{1},\ldots,\tilde{E}_{k-1})$, 
    by a symplectomorphism that sends $(L-E_1-E_2,E_3,\hdots,E_k)$ to  $(\tilde{E}_{1},\ldots,\tilde{E}_{k-1})$. 
   Indeed, it is enough to see this for $k=2$, and then it follows from Delzant's theorem \cite{De88}, since the two manifolds admit Hamiltonian $T^2$-actions with the same momentum map image. See \Cref{fig:cropped-two-cuts}.
By  \cite[page 6]{AKP24}, 
the identity component of the  symplectomorphisms group acts transitively on the space of configurations of disjoint exceptional spheres in  the classes of exceptional divisors.
Thus, we may assume that $C_1,\hdots,C_\ell$ are exceptional divisors of $(M^i_{\lambda-\eps},\omega^i_{\lambda-\eps})$.

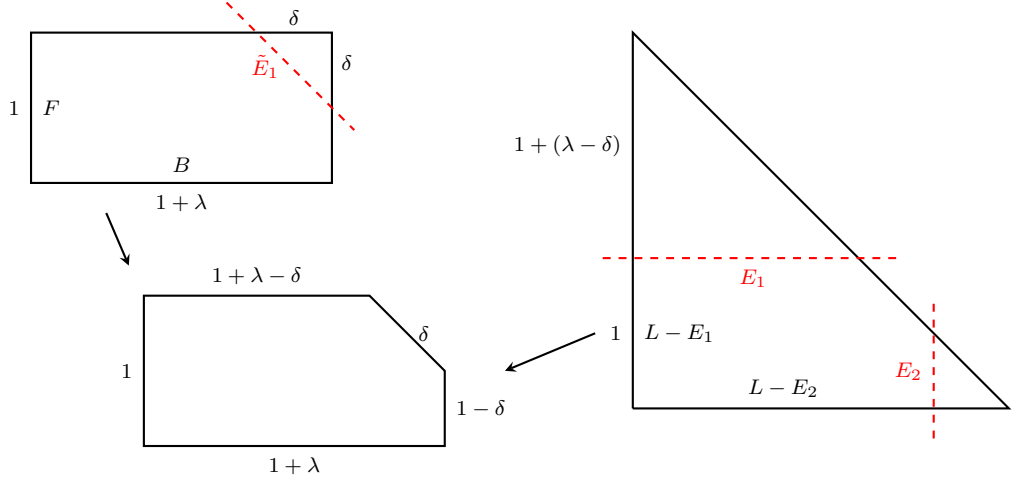
\begin{figure}
\begin{tikzpicture}[thick, >=stealth]
	\draw (0,3) -- node[below]{\tiny $1+\lambda$} node[above]{\tiny $B$} (4,3) -- (4,5) -- 
	(0,5) -- node[left]{\tiny $1$} node[right]{\tiny $F$} cycle;
	\draw (3.5,5.2) node{\tiny $\delta$};
	\draw (4.2,4.6) node{\tiny $\delta$};
	\draw[red, dashed] (2.6,5.4) -- node[left]{\tiny $\tilde{E}_1$}+(-45:2.4cm); 
	\draw (8,0) -- node[above]{\tiny $L-{E}_2$}(12,0) -- (13,0) -- (8,5) -- node[left]{\tiny $1+(\lambda-\delta)$}node[right]{}(8,2) -- node[left]{\tiny $1$} node[right]{\tiny $L-E_1$}(8,0); 
	\draw[red, dashed] (7.6,2) -- node[below]{\tiny $E_1$} (11.6,2) 
	(12,-0.4) -- node[left]{\tiny ${E}_2$}(12,1.4);
	\draw (1.5,-0.5) -- node[below]{\tiny $1+\lambda$} 
	(5.5,-0.5) -- node[right]{\tiny $1-\delta$}
	(5.5,0.5) -- node[right]{\tiny $\delta$} 
	(4.5,1.5) -- node[above]{\tiny $1+\lambda-\delta$}
	(1.5,1.5) -- node[left]{\tiny $1$} cycle; 
	\draw[->] (1,2.6) to (1.3,1.9); 
	\draw[->] (7.5,1) to (6.3,0.5);
\end{tikzpicture}
\caption{Toric moment polytopes of blowups in two ways.}\label{fig:cropped-two-cuts}
\end{figure}

\end{remark}
We will also need the following corollary of positivity of intersections of $J$-holomorphic curves in an almost complex manifold of dimension four.

\begin{lemma}\label{lem:dis-rep}
Let $B=\C \PP^2 \# k \overline{\C \PP^2}$ and $\omega$ a blowup form on $B$.
Consider a finite set $\{C_1,\ldots,C_{\ell}\}$ of pairwise disjoint embedded symplectic spheres, all of self-intersection $-1$. Let $A$ be an exceptional class of minimal size.

If the intersection number of $A$ with each of the classes of the given symplectic spheres is $0$, the class $A$ is represented by an embedded symplectic sphere that is disjoint from any of the given symplectic spheres. 
\end{lemma}

We sketch the proof for completeness. Recall that an \emph{almost complex structure} on a manifold $B$ is an automorphism $J \colon TB \to TB$ such that $J^{2}=-\id$. It is \emph{tamed by a symplectic form} $\omega$ if $\omega(v,Jv)>0$ for all $v \neq 0$. A \emph{$J$-holomorphic sphere} is a map $u \colon S^2 \to B$ such that
$$du \circ j_{S^2}=J \circ du,$$ where $j_{S^2}$ is the almost complex structure induced from a complex atlas on $S^2$. 
If $u$ is an embedding, we call its image $u(S^2)$ an \emph{embedded $J$-holomorphic sphere}.

\begin{proof}[Sketch of proof.]
    By the positivity of intersections of $J$-holomorphic curves in a four-dimensional almost complex manifold, for
 distinct embeddings of $J$-holomorphic spheres  $u_0, u_1 \colon S^2 \to B$,
     $$\sharp\{(z_0,z_1) \in S^2 \times S^2 \,|\, u_0(z_0)=u_1(z_1)\}\leq [u_0(S^2)] \cdot [u_1(S^2)].$$ 
     See \cite[Theorem 2.6.3]{MS12}.
     Note that if $J$ is $\omega$-tame then an embedded $J$-holomorphic sphere is symplectic. 
Thus, it is enough to show that there is an $\omega$-tame almost complex structure $J$ such that
 each of the given embedded symplectic spheres is $J$-holomorphic and 
 $A$ is represented by an embedded $J$-holomorphic sphere.
 This follows from the following facts:
 \begin{itemize}
 \item Since $A$ is an exceptional class with minimal $\omega$-symplectic size, for every $\omega$-tame almost complex structure $J$ there exists an embedded $J$-holomorphic sphere in the class $A$. See \cite[Corollary 2.4]{KKP15}\footnote{The proof is using Gromov-Witten invariants. This result was also obtained by \cite{Pi08}, for more general four-manifolds, using Seiberg-Witten-Taubes theory.}.

   \item There is an $\omega$-tame almost complex structure $J^*$ such that each $C_i$ is an embedded $J^*$-holomorphic sphere. We construct $J^{*}$ by first defining ${J^*}|_{T{C_i}}$ for each $C_i$ such that the symplectic embedding $S^2 \to B$ whose image is $C_i$ is holomorphic. For every $i$, extend ${J^*}|_{T{C_i}}$ to an $\omega$-tame fiberwise complex structure on the symplectic vector bundle ${TB}|_{C_i}$. 
   Then extend the obtained structure 
 to an $\omega$-tame almost complex structure on $B$. See \cite[Section 2.6]{MS98}.

\end{itemize}
 
\end{proof}

\begin{proof}[Proof of  \Cref{prop:characterizationD}]

Consider the isomorphism
   $$f \colon \mu_1^{-1}((-\infty,\lambda-r]) \to \mu_2^{-1}((-\infty,\lambda-r]).$$
 We claim that $f$ maps  $\mathcal{D}_1$ into  $\mathcal{D}_2$.
 Since $f$ is an isomorphism, it maps $\mathcal{E}'_1$ into $\mathcal{E}'_2$, where $\mathcal{E}'_i$ is as denoted in \Cref{lem:E'}(2).
     Thus, it is enough to show that $\mathcal{D}_i=\mathcal{E}'_i$. We can apply \Cref{lem:E'}, since, by assumption, for any regular value $\lambda-\eps$ right below $\lambda$,
the reduced space $(M^{i}_{\lambda-\eps},\omega^{i}_{\lambda-\eps})$ is 
as required. By the rigidity assumption on $(M^i,\omega^i,\mu_i)$, we can apply \Cref{prop:extension}  to extend $f$ to an isomorphism 
$$f \colon \mu_1^{-1}((-\infty,\lambda-\eps]) \to \mu_2^{-1}((-\infty,\lambda-\eps]),$$
with $\eps>0$ as small as required for \Cref{lem:E'}(5).

By definition, the classes in $\mathcal{D}_{i}$ are represented by the spheres $$C^{i}_1,\ldots,C^{i}_{\ell}$$  in $M^{i}_{\lambda-\eps}$ that are sent to the isolated fixed points of index $2$ at $M^{i}_{\lambda}$ by the Morse flow $f_{\Morse}$.
These spheres are pairwise disjoint $\omega^{i}_{\lambda-\eps}$-symplectically embedded spheres of self-intersection $-1$ and size $\eps$. Moreover, the evaluation of $\omega_{\lambda-t}$ on the class of $C_{j}^{i}$ in $M_{\lambda-t}$ equals $t$ for all $0<t\leq\eps$. See  \S \Cref{rem:criticalvalue}.
Thus, $\mathcal{D}_i$ is contained in $\mathcal{E}'_i$.

 Assume that $A \in \mathcal{E}^{'}_{i}\smallsetminus \mathcal{D}_i$.
By \Cref{lem:E'}(5), for any two different classes $E,E' \in \mathcal{E}^{'}_{i}$ we have $[E]\cdot [E']=0$. 
Therefore, by \Cref{lem:dis-rep}, 
the homology class $A \in \mathcal{E}^{'}_{i}\smallsetminus \mathcal{D}_i$ is represented by an embedded symplectic sphere $S$ of self-intersection $-1$
 in $(M^i_{\lambda-\eps},\omega^{i}_{\lambda-\eps})$ that is disjoint from $C^{i}_1,\ldots,C^{i}_\ell$.
  By a small perturbation, we can assume that $S$ is also disjoint from the  points in ${F'}^{i}$ that are sent by the Morse flow to isolated fixed points of index $1$ at $\lambda$.

 By \Cref{cl:top-bup}, the manifold $M^{i}_{\lambda-\eps}$ is homeomorphic to the blowup of $M^{i}_{\lambda}$ at the isolated fixed points of co-index $1$, with the exceptional divisors in the classes $[C^{i}_1],\ldots,[C^{i}_{\ell}]$.
 So we can present 
 \begin{equation}\label{eq:present}
     H_{2}(M^{i}_{\lambda-\eps})=H_{2}(M^{i}_{\lambda})\oplus[C_1]\oplus \cdots \oplus [C_\ell].
 \end{equation}
 Since $C^{i}_1,\ldots,C^{i}_{\ell}$ are of minimal area, $(M^i_{\lambda-\eps},\omega_{\lambda-\eps}^i)$ is symplectomorphic to a symplectic rational surface for which the $C^{i}_1,\ldots,C^{i}_{\ell}$ are exceptional divisors, see \Cref{rem:divisors}. Symplectically blowing down $(M^{i}_{\lambda-\eps},\omega^{i}_{\lambda-\eps})$ along $C^{i}_1,\ldots,C^{i}_{\ell}$ therefore gives a symplectic rational surface homeomorphic to $M^{i}_{\lambda}$, and thus diffeomorphic to it because $M^{i}_{\lambda}$ is also assumed to be a symplectic rational surface,  with a form ${\omega'}^{i,\eps}_{\lambda}$.
The continuous blowup map $f_{\Morse}\colon M^{i}_{\lambda-\eps}\to M^{i}_{\lambda}$ is a diffeomorphism into its image on the complement of $F'^i$.
Since $S$ is disjoint from $F'^i$, the image of $S$ under this map
is a smoothly embedded sphere of self-intersection $-1$; it is symplectic with respect to ${\omega'}^{i,\eps}_{\lambda}$. Thus, by \Cref{lem:expchar},
the class of the image of $S$ is represented by an embedded symplectic sphere in $(M^{i}_{\lambda},\omega^{i}_{\lambda})$. 
The presentation \eqref{eq:present} of $H_{2}(M_{\lambda-\eps}^{i})$ allows us to identify $A$ with the class in $H_{2}(M^{i}_{\lambda})$ of the image of $S$ under the blowup map.

 However, since the class $A$ is in $\mathcal{E}^{'}_{i}$, it has coupling $t$ with $[\omega^{i}_{\lambda-t}]$ as a class in $H_{2}(M^{i}_{\lambda-t})$ for all $0<t\leq\eps$.  
 By construction, $S \subset M^{i}_{\lambda-\eps} \smallsetminus {F'}^{i}_{\iso}$. 
 If $F^{i}_{\iso}=F^{i}$ then the form $(f_{\Morse}|_{M^{i}_{\lambda-\eps} \smallsetminus {F'}^{i}_{\iso}})_{*}{\omega^{i}_{\lambda-\eps}}$ converges to ${\omega^{i}_{\lambda}}|_{M^{i}_{\lambda} \smallsetminus F^{i}}$ as $\eps$ converges to $0$, by \Cref{cor:restdiffeo1}. If $F^{i}_{\iso} \subsetneq F^{i}$ then still the cohomology class $(f_{\Morse}|_{M^{i}_{\lambda-\eps}\smallsetminus {F'}^{i}_{\iso}})_{*}[\omega^{i}_{\lambda-\eps}]$ converges to ${[\omega^{i}_{\lambda}]}$ on ${M^{i}_{\lambda} \smallsetminus F^{i}_{\iso}}$ as $\eps$ goes to $0$, by \Cref{lem:cohomologyconverges}.
 Hence the coupling of $A$ in $H_{2}(M^{i}_{\lambda})$ 
 with $[\omega^{i}_{\lambda}]$ is $0$.
 We get a contradiction, showing that $\mathcal{E}^{'}_{i} \smallsetminus \mathcal{D}_i$ is empty. This completes the proof.
\end{proof}
We are almost ready to prove \Cref{thm:extending-g}. The assumption on $f_{\lambda-r}$ intertwining $\mathcal{D}^1_{\sph}$ and $\mathcal{D}^2_{\sph}$, and also $\mathcal{D}^1$ and $\mathcal{D}^2$ by \Cref{prop:characterizationD}, is not yet enough to apply \Cref{lem:nonextremal-2} and \Cref{cor:psi_t1}; we still need to argue that $f$ can be manipulated in such a way that $f_{\lambda-r}$ is the identity near $F_1'$. For this, we certainly need to show the following lemma.
\begin{lemma}\label{lem:isomorphicnormalbundles}
    In the situation of \Cref{thm:extending-g}, let $C_1,C_2$ be fixed spheres in $M^1_{\lambda},M^2_{\lambda}$ such that $[f_{\lambda-r}(C_1')]=[C_2']$. Then the equivariant normal bundles of $C_1$ and $C_2$ in $M^1$ and $M^2$ are isomorphic. 
\end{lemma}
This, in turn, will be a consequence of the next lemma.
\begin{lemma}\label{lem:relationlinebundles}
    Let $V'=\underline{\C_{-1}}\oplus \underline{\C_1}$ be a sum of two complex line bundles over a symplectic surface $(\Sigma,\omega_{\Sigma})$ on which $S^1$ acts fiberwise as $t\dot (z_1,z_2)=(t^{-1}z_1,tz_2)$. Equip $V'$ with a fiber metric just as in \Cref{nt:normalbundle}, and let $V$ be a neighborhood of the $0$-section such that there exists an $S^1$-invariant symplectic form on $V$ whose momentum map agrees with $\mu\colon V'\to \R, \; \mu(z_1,z_2)=|z_2|^2-|z_1|^2$ on $V$.\\
    Then, if $c_{\pm}$ denotes the first Chern class of $\underline{\C_{\pm}}$ and $c$ denotes the first Chern class of the normal bundle of $\Sigma$ in $\mu^{-1}(0)/S^1$, we have $c_-+c_+=c$.
\end{lemma}
\begin{proof}
    Let $p$ be a point of $\Sigma$ and $U_1$ a closed disk around it. Denote by $U_2$ the closure of $\Sigma\smallsetminus U_1$. Then $V'=\underline{\C_{-1}}\oplus \underline{\C_1}$ admits trivializations over $U_1$ and $U_2$, and the transition function $\gamma\colon \partial U_1=\partial U_2=S^1\to S^1\times S^1\subset \U(2)$ is of the form
    $$\gamma(t)(z_1,z_2)=(t^{c_-},t^{c_+})(z_1,z_2)=(t^{c_-}z_1,t^{c_+}z_2).$$
    Similarly, for the bundle $\underline{\C_{-1}}\otimes \underline{\C_1}$, the transition function is (by definition)
    $$\gamma^{\otimes}(t)(z_1\otimes z_2)=(t^{c_-}z_1 \otimes t^{c_+}z_2).$$
    It is enough to construct an embedding $\iota\colon \underline{\C_{-1}}\otimes \underline{\C_1}\to \mu^{-1}(0)/S^1$ as
    \[
    z_1\otimes z_2 \mapsto [z_1z_2:|z_1z_2|],
    \]
    where $[z_1:z_2]$ denotes a representative of $(z_1,z_2)$ in $V/S^1$. This is clearly well-defined over the interior of $U_1$ and $U_2$, so we need to check that this map is compatible with the transition functions whenever $z_1\otimes z_2$ is in a fiber over $t\in \partial U_1=\partial U_2$. This means that $\gamma^{\otimes}(t)(z_1\otimes z_2)$ must be mapped to $\gamma(t)(\iota([z_1\otimes z_2]))$. We calculate
    \[
    \iota(t^{c_-}z_1\otimes t^{c_+}z_2)=[t^{c_-+c_+}z_1z_2:|z_1,z_2|]=[t^{c_-}z_1z_2:t^{c_+}|z_1,z_2|]=\gamma(t)(\iota([z_1\otimes z_2])).
    \]
    This finishes the proof.
\end{proof}
\begin{proof}[Proof of \Cref{lem:isomorphicnormalbundles}]
    We know that the Euler class of the  negative normal bundle equals the Euler class $e^i$ of the principal $S^1$-bundle $S^1\to \pi^{-1}(C_i')\to C_i'$ (where $\pi$ is the orbit map $\mu_i^{-1}(\lambda-r)\to M^i_{\lambda-r}$) under the diffeomorphism $C_i'\to C_i$. Since $[f_{\lambda-r}(C_1')]=[C_2']$, the Euler classes $e^i$ agree, so $C_1$ and $C_2$ already have negative normal bundles with the same first Chern class, hence they have isomorphic negative normal bundles.\\
    Now, by assumption, the first Chern classes of $C_1$ and $C_2$ in $M^1_{\lambda}$ resp.\ $M^2_{\lambda}$, with respect to their symplectic orientations, are $-1$, and in particular agree. Using \Cref{lem:relationlinebundles}, we conclude that they also have isomorphic positive line bundles.
\end{proof}

\begin{corollary}\label{cor:psi_t2}
    Let $\kappa>0$ be such that \eqref{eq:kappa} holds.
For $i=1,2$, let $(M^i,\omega^i,\mu_i)$ and $\lambda$ be as in \Cref{set:intro}, and assume that they have the same $*$-small fixed point data at $\lambda$. Assume that $\lambda$ is non-extremal, the only critical value of $\mu_i$ and that the fixed surfaces  of $M^i$ are all spheres of self-intersection $-1$ in $M^i_{\lambda}$.\\

 Consider an isomorphism 
 $$f \colon \mu_1^{-1}((-\infty,\lambda-\eps]) \to \mu_2^{-1}((-\infty,\lambda-\eps])$$
 with $\lambda-\eps$ right below $\lambda$ and with $\eps \in (0,\kappa)$.\\

    Then, if $f_{\lambda-\eps}$ intertwines $\mathcal{D}^1_{\sph}$ and $\mathcal{D}^2_{\sph}$,
    there is an isotopy of symplectomorphisms $M^{1}_{\lambda-\eps} \to M^{2}_{\lambda-\eps}$ connecting $f_{\lambda-\eps}$ and a symplectomorphism that is the identity near $F'_1$ (see \Cref{def:identitynearF_1}).
    \end{corollary}

    \begin{proof}
By \Cref{prop:characterizationD}, $f_{\lambda-\eps}$ maps $\mathcal{D}_1$ bijectively into $\mathcal{D}_2$. 
    Therefore,  using \Cref{lem:normalizationspheres} and the assumption on $f_{\lambda-\eps}$,
    we find an isotopy connecting $f_{\lambda-\eps} \colon M^{1}_{\lambda-\eps} \to M^{2}_{\lambda-\eps}$ and a symplectomorphism that is the identity near $F_1'$.
    \end{proof}

We can now prove \Cref{thm:extending-g}.
    \begin{proof}[Proof of \Cref{thm:extending-g}]
    Let $\lambda$ be a critical value of $\mu_1$ and $\mu_2$, either non-extremal for both or maximal for both. Let  $$f \colon \mu_1^{-1}((-\infty,\lambda-r]) \to \mu_2^{-1}((-\infty,\lambda-r])$$ be an isomorphism,  where $\lambda-r$ is right below $\lambda$. 
Use  the rigidity assumption and \Cref{prop:extension} to extend $f$, i.e., to replace it 
by an isomorphism
\begin{equation}\label{eq:f}
    f\colon \mu_1^{-1}((-\infty,\lambda-\eps]) \to \mu_2^{-1}((-\infty,\lambda-\eps]) \text{ with } 0<\eps<r 
\end{equation}
that agrees with the given isomorphism on $\mu_1^{-1}((-\infty,\lambda-(r+\eps')])$ for $\eps'>0$ arbitrarily small; in case $\lambda$ is non-extremal, we know that \Cref{lem:isomorphicnormalbundles} holds and thus also ask that $\eps$ is in $(0,\kappa)$, for $\kappa>0$ such that \eqref{eq:kappa} holds. If $\lambda$ is maximal and $\dim F_1=2$, then we choose $\eps$ such that $0<\eps<\theta$,  where $\theta$ is as in \Cref{lem:connected<4}. 

We claim that for some $\eps>\delta>0$, 
we have an isomorphism
$$g\colon \mu_1^{-1}([\lambda-\delta,\lambda+\delta])\to \mu_2^{-1}([\lambda-\delta,\lambda+\delta]).$$ Moreover, for any $t\in (\lambda-\delta,\lambda)$ and any extension of $f$ to  an isomorphism
    $$f\colon \mu_1^{-1}((-\infty,t]) \to \mu_2^{-1}((-\infty,t]),$$
\begin{itemize}
\item  $f_{t}$ and $g_{t}$ induce the same map on homology, and, 
\item $f_t$ and $g_t$ are furthermore isotopic through symplectomorphisms $(M^1_{t},\omega^{1}_{t}) \to (M^{2}_{t},\omega^{2}_{t})$ if $\lambda$ is maximal and $\dim F_1<4$.
\end{itemize}
This follows from
\begin{itemize}
   \item \Cref{cor:psi_t2} and \Cref{cor:psi_t1}, if $\lambda$ is non-extremal.
    \item  \Cref{lem:psi_textremum}, if $\lambda$ is maximal and $\dim F_1=4$.
    \item \Cref{lem:connected<4}, if $\lambda$ is maximal and $\dim F_1=2$.
\item the fact that if $\lambda$ is maximal and $F_1$ is a point then the weights of the $S^1$-action at the points $F_1,F_2$ are $-1,-1,-1$, so neighborhoods of $F_1$ and $F_2$ are equivariantly symplectomorphic to a neighborhood of $0$ in $\C_{-1}\oplus \C_{-1} \oplus \C_{-1}$ endowed with the standard symplectic form, see \S \ref{nt:localnormalform} and \S \ref{nt:local6}; 
in particular, the reduced space of $M^i$  at $t \in (\lambda-\delta,\lambda)$  is symplectomorphic to $\C \PP^2$ endowed with a multiple of the Fubini-Study form $\omega_{\operatorname{FS}}$;
the symplectomorphism group of $(\C \PP^2,\alpha \omega_{\operatorname{FS}})$ retracts onto the
isometry group $\PU(3)$ of $\C \PP^2$, by \cite[Remark in 0.3.C]{Gr85}, and hence is connected.
\end{itemize}

Now, extend $f$ to an isomorphism $\mu_1^{-1}(-\infty,\lambda-\eta] \to \mu_2^{-1}(-\infty,\lambda-\eta]$ for $\delta>\eta>0$ using \Cref{prop:extension}. 
By assumption, $M_{\lambda-\eta}^{i}$ is a symplectic rational surface.
If $M_{\lambda-\eta}^{1}=S^2\times S^2$, the fact that the symplectomorphism 
 $(g_{\lambda-\eta})^{-1} \circ f_{\lambda-\eta} \colon (M^1_{\lambda-\eta},\omega^1_{\lambda-\eta}) \to (M^1_{\lambda-\eta},\omega^1_{\lambda-\eta})$
acts as the identity on homology implies that it is isotopic to the identity through symplectomorphisms, by \cite{Gr85}.
If $M_{\lambda-\eta}^{1}$ is a $k$-blowup of $\C \PP^{2}$, then, by \cite[Theorem A.1]{LLW22}, a symplectomorphism that acts as the identity on homology 
 is isotopic to the identity through diffeomorphisms. Therefore, the rigidity assumption implies that $f_{\lambda-\eta}$ and $g_{\lambda-\eta}$ are isotopic through symplectomorphisms $(M^1_{\lambda-\eta},\omega^1_{\lambda-\eta})\to (M^2_{\lambda-\eta},\omega^2_{\lambda-\eta})$ also if $\lambda$ is non-extremal and if $\lambda$ is maximal and $\dim F_1=4$.

Thus we can use \Cref{cor:smoothing}  to paste $f$ and $g$  on level $\lambda-\eta$ to an almost symplectic $\mu-S^1$-diffeomorphism
$$ g'\colon \mu_1^{-1}((-\infty,\lambda+\delta])\to \mu_2^{-1}((-\infty,\lambda+\delta]) $$
whose restriction to $\mu_1^{-1}((-\infty,\lambda-\eta])$ resp.\ $\mu_1^{-1}([\lambda-\eta,\lambda+\delta])$  differs from $f$ resp.\ $g$ only near level $\lambda-\eta$. 
In particular, the forms in the standard homotopy
\begin{equation}\label{eq:g'as}
    \omega(s)=s(g')^*\omega^2+(1-s)\omega^1, \quad s \in[0,1]
\end{equation}
are all symplectic and represent the same cohomology class on $\mu_1^{-1}((-\infty,\lambda+\delta])$.
Note that the restriction of $\omega(s)$ to  $\mu_1^{-1}((-\infty,\lambda-\eps''])$ coincides with $\omega^1$ for $\eps''>\eps$ arbitrarily close to $\eps$, and in particular on  $\mu_1^{-1}((-\infty,\lambda-(r+\eps')])$.

Finally, apply Moser's method, see \Cref{rem:moserwithfixedpoints}, to get a family $\Psi_t$ of equivariant diffeomorphisms $\mu_1^{-1}((-\infty,\lambda+\delta]) \to \mu_1^{-1}((-\infty,\lambda+\delta])$ such that $\Psi_t^{*}\omega(t)=\omega^1$ and  ${\Psi_{t}}|_{\mu_1^{-1}((-\infty,\lambda-(r+\eps')])}=\id$. 
The map $h=g' \circ \Psi_1$ is the required isomorphism.

 \end{proof}

\appendix
\section{Local Data} \label{Local Data}
The emphasis of our paper is on whether fixed point data determine the isomorphism type. However, we note that in our counter example the non-isomorphic semi-free Hamiltonian $S^1$-manifolds also have the same local data. See  \Cref{lem:mm'local}. In \Cref{rem:proofthm2.6}, we address the problem in \cite{Go11}'s proof that the local data determine the isomorphism type. Here we recall \cite{Go11}'s definition of local data. We begin with the notions of cobordism, regular slice, and gluing map.
\begin{definition}\label{def:cobordism}
\cite[Definition 2.2]{Go11}.
Let $\eps_0>0$ and $\lambda \in \mathbb{R}$. A \textbf{cobordism at $\lambda$} $(Y,H,\eps)$, $0<\eps<\eps_0$, consists
of a Hamiltonian $S^1$-manifold $Y$ with 
momentum map $H\colon Y\to (\lambda-\eps,\lambda+\eps)$,
such that $\lambda$ is the only critical value.
 If $\lambda$ is not a minimum nor a maximum, we require that the momentum map $H$ is onto $ (\lambda-\eps,\lambda+\eps)$. If $\lambda$ is a minimum (maximum), we require that $H(Y)=[\lambda,\lambda+\eps)$ ($H(Y)=(\lambda-\eps,\lambda])$.\\
Two cobordisms $(Y,H,\eps)$ and $(Y',H',\eps')$ at a non-extremal $\lambda$ are \textbf{equivalent} if there is $0<\eps''<\min\{\eps,\eps'\}$ such that $(H^{-1}(\lambda-\eps'',\lambda+\eps''),H)$ and $({H'}^{-1}(\lambda-\eps'',\lambda+\eps''),H')$ are isomorphic as Hamiltonian $S^1$-manifolds. Equivalence is defined similarly if $\lambda$ is a maximum or a minimum.\\
A \textbf{critical germ} $G(\lambda,\eps_0)$ is an equivalence class of cobordisms at 
$\lambda$. 
\end{definition}
Note that if $0<\delta_0<\eps_0$, there is a natural restriction map $G(\lambda,\eps_0) \to G(\lambda,\delta_0)$.

\begin{definition}\label{def:slice}
\cite[Definition 2.3]{Go11}.
    Let $I$ be an open interval. A \textbf{regular slice} $(Z,K,I)$ consists of
    a \emph{free} Hamiltonian $S^1$-manifold $Z$ with surjective momentum map $K\colon Z\to I$.\\
    We say that regular slices $(Z,K,I)$ and $(Z',K',I)$ are \textbf{equivalent} if $(Z,K)$ and $(Z',K')$ are isomorphic as Hamiltonian $S^1$-manifolds.
    We denote by $F(I)$ an equivalence class of such slices.
    \end{definition}

     \begin{definition}\label{def:gluing}
    \cite[Definition 2.4]{Go11}.
        Let $I=(\lambda,\lambda')$ and $0<\eps_0<\lambda'-\lambda$. 
        Let $(Z,K,I)$ be in $F(I)$ and $(Y,H,\eps)$ be in $G(\lambda,\eps_0)$.
        A \textbf{gluing map} $(\phi,\delta)\colon (Y,H,\delta) \to (Z,K,I)$ consists of a positive $\delta<\eps$
        and an isomorphism $\phi \colon H^{-1}(\lambda,\lambda+\delta) \to K^{-1}(\lambda,\lambda+\delta)$.
        \\
        
Two gluing maps $$(\phi,\delta) \colon (Y,H,\delta) \to (Z,K,I)$$
and
        $$(\phi',\delta')\colon (Y',H',\delta') \to (Z',K',I),$$
        where $(Z',K',I)$ is in $F(I)$ and $(Y',H',\eps')$ is in $G(\lambda,\eps_0)$, are \textbf{equivalent} if there is $0<\delta''<\delta, \delta'$
        as well as isomorphisms 
        \[
        f\colon Y\to Y', \quad g\colon Z\to Z'
        \]
        such that $g\circ \phi=\phi'\circ f$
        on $H^{-1}(\lambda,\lambda+\delta'')$.
        A \textbf{gluing class} $\Phi\colon G(\lambda,\eps_0)\to F(I)$ is an equivalence class of such gluing maps. \\
We similarly define gluing maps and class for $F(\lambda',\lambda)$ and $G(\lambda,\eps_0)$; in that case $\phi \colon H^{-1}(\lambda-\delta,\lambda)$. 
    \end{definition}

The gluing map allows us to glue a regular slice and a cobordism to get a Hamiltonian $S^1$-manifold.
For a regular slice $(Z,K,(\lambda,\lambda'))$,
    a cobordism $(Y,H,\eps)$ and a gluing map $\phi \colon \colon H^{-1}(\lambda,\lambda+\delta) \to K^{-1}(\lambda,\lambda+\delta)$ with $0<\delta<\eps$, consider the manifold $$Y\cup_{(\phi,\delta)} Z,$$ that is 
    $$Y \sqcup Z/\sim \text{ where }x \sim y \text{ iff } \phi(x)=y.$$ 
    Since $\phi$ is an isomorphism, the symplectic forms, the $S^1$-actions and the momentum maps on $Z$ and $Y$ induce well defined symplectic form, $S^1$-action, and momentum map on $Y\cup_{(\phi,\delta)} Z$. Moreover, the gluing of a critical germ and a regular slice along a gluing class is well defined up to isomorphism \cite[Lemma 2.5]{Go11}. 
   \begin{definition}\label{def:localdata}
    Let $(M,\omega,\mu)$ be 
    a semi-free Hamiltonian $S^1$-manifold
with a proper momentum map $\mu \colon M \to \R$ whose image is bounded. Its set of \textbf{local data} consists of the following:
        \begin{itemize}
            \item Its critical levels $\lambda_0<\hdots < \lambda_k$.
            \item The critical germs $G(\lambda_i,\eps_i)$, where $\eps_i>0$ is small enough such that $\lambda_i$ is the only critical value in $(\lambda_i-\eps_i,\lambda_i+\eps_i)$, defined by $\mu^{-1}(\lambda_i-\delta,\lambda_i+\delta)$ with $0<\delta<\eps_i$, for all $i$.
            \item The equivalence classes of regular slices $F(\lambda_i,\lambda_{i+1})$ defined by $\mu^{-1}(\lambda_{i},\lambda_{i+1})$, for all possible $i$.
            \item The gluing classes $\Phi_i^-$, $\Phi_i^+$ from $G(\lambda_i,\eps_i)$  to $F(\lambda_{i-1},\lambda_{i})$, $F(\lambda_{i},\lambda_{i+1})$ for all possible $i$.
        \end{itemize}

     \end{definition}

\section{Proofs of results in \Cref{sec:almostsymplectic}}\label{app:extrafour}

We prove results that are used in \Cref{sec:almostsymplectic}. The first is classic.

\begin{lemma}\label{lem:Eulerclasses}
    Let $\pi\colon P\to B$ and $\pi'\colon P'\to B'$ be principal $S^1\cong \SO(2)$-bundles, and denote by $e(P)$ resp.\ $e(P')$ their Euler classes in $H^2(B;\Z)$ resp.\ $H^2(B',\Z)$. Let $f\colon B\to B'$ be a diffeomorphism (homeomorphism) that intertwines $e(P)$ and $e(P')$. Then $f$ lifts to a smooth (continuous) bundle isomorphisms $\tilde{f}\colon P\to P'$ equivariant with respect to the $\SO(2)$-action on each fiber, that is, we have $\pi'\circ \tilde{f}=f\circ \pi$.\\
    The same holds for $D^2$-bundles with structure group $\SO(2)$.
\end{lemma}
\begin{proof}
    We only need to show this for the $\SO(2)$-bundles, since the statement about the $D^2$-bundles immediately follows from that by applying it to the underlying $\SO(2)$-bundles.\\
    Let $f^*(P')\to B$ be the pullback-bundle of $P'\to B'$ with respect to $f$. Then $f$ lifts to a smooth (continuous) bundle isomorphism $f'\colon f^*(P')\to P'$ by the universal property of the pullback bundle. Further, we have $e(f^*(P'))=f^*e(P')=e(P)$, so $f^*(P')\to B$ is isomorphic to the bundle $P\to B$, preserving base-points. By concatenating this bundle isomorphism with $f'$, we obtain the desired lift $\tilde{f}$ of $f$. 
\end{proof}

The next lemma is used in the proof of \Cref{cor:smoothing}.
\begin{lemma}\label{lem:cohomologyclass}
    Let $(M,\omega,\mu)$ be a connected Hamiltonian $S^1$-manifold (of any dimension) with proper momentum map whose image is either $\mu(M)=[\lambda_{min},\lambda+\eps)$ with $\eps>0$ or $\mu(M)=[\lambda_{min},\lambda]$, where $\lambda$ is the highest critical value of $M$ in both cases (that is, $\lambda$ is maximal in the second case but not in the first case, in which $M$ is not compact). Let $\delta>0$ be small enough so that there is no critical value in $[\lambda-\delta,\lambda)$. Set 
    $$U:=\mu^{-1}((-\infty,\lambda-\delta])$$
     and
    $$V:=\begin{cases}
        \mu^{-1}([\lambda-\delta,\lambda+\eps)) & \text{ if }\mu(M)=[\lambda_{min},\lambda+\eps)
        \\
        \mu^{-1}([\lambda-\delta,\lambda]) & \text{ if }\mu(M)=[\lambda_{min},\lambda]
    \end{cases}.$$
    Then any two symplectic forms $\omega,\,\omega'$ on $M$ with momentum maps $\mu,\mu'$ that are cohomologous as diferential forms in $U$ and in $V$ are also cohomologous in $M$. 
\end{lemma}
\begin{proof}
If $\omega$ and $\omega'$ are cohomologous in $U$ and $V$, their restrictions on the fixed point set $F$ have to agree. Hence, the restriction of their equivariant extensions $\omega-\mu$ resp.\ $\omega'-\mu'$ in the Cartan model of $H^{*}_{S^1}(U\cup V;\R)$ also agree. Since $H^*_{S^1}(U\cup V;\R)\to H_{S^1}^*(F;\R)$ is injective by \cite{Ki84}, it follows that $\omega-\mu=\omega'-\mu'$ in $H^*_{S^1}(U\cup V;\R)$, and hence $\omega$ and $\omega'$ are cohomologous.

\end{proof}

Now, we prove Items (2) and (3) of \Cref{lem:equivalent}. For that, we need a parametrized version of Moser's method.
{ Recall that the following is used in Moser's method.}
\begin{remark}\label{rem:hodgetheory}
    Pick a metric on a smooth, oriented manifold $M$. Let $(\alpha,\beta)$ be the scalar product $\langle\alpha, \beta \rangle$ of $k$-forms $\alpha,\beta \in \Omega^k$ induced by this metric. Then, w.r.t.\ that scalar product, the differential $d\colon \Omega^1\to \Omega^{2}$ has an adjoint $d^*\colon \Omega^{2}\to \Omega^1$, and $d_{|\text{im}(d^*)}\colon \text{im}(d^*)\to d\Omega^1$ is an isomorphism. Indeed, it is surjective because of the Hodge decomposition
    $$\Omega^2(M)=d\Omega^1\oplus d^* \Omega^3\oplus \mathcal{H}^2,$$
    where $\mathcal{H}^2$ denotes the two-forms on $M$ that are both $d$- and $d^*$-closed (see, e.g., \cite[Theorem 6.8]{Wa71}). It is injective because if $dd^*\alpha=0$, then $0=\langle dd^*\alpha, \alpha \rangle=\langle d^*\alpha, d^*\alpha \rangle$, so $d^*\alpha$ already vanishes.\\
    Therefore, for any family of exact two forms $\omega_s$ paramterized by a smooth manifold, we find a corresponding smooth family $\beta_s$ such that $d\beta_s=\omega_s$ by setting $\beta_s:=d_{|\text{im}(d^*)}^{-1}(\omega_s)$. Further, $\beta_s=0$ if and only if $\omega_s=0$.
\end{remark}
\begin{remark}\label{rem:Moser}
    Let $B$ be a compact manifold and $\omega_t$, $\omega_t'$ be smooth families, parametrized by some smooth manifold, of two-forms on it. Let $\omega_{s,t}$, $s\in [0,1]$, be another smooth family of two-forms such that, for each $t$, $\omega_{0,t}=\omega_t$, $\omega_{1,t}=\omega'_t$ and $\omega_{s,t}$ defines an isotopy between $\omega_{0,t}$ and $\omega_{1,t}$.\\
    Using \Cref{rem:hodgetheory}, we find a smooth family $\alpha_{s,t}$ of one-forms whose differential equals $\partial_s \omega_{s,t}$, meaning that we can apply Moser's method for each isotopy of forms $s\mapsto \omega_{s,t}$ independently and obtain a smooth (because $\alpha_{s,t}$ is smooth) family of diffeomorphisms $f_{s,t}\colon B\to B$ such that, for each $(s,t)$, $f_{s,t}^*(\omega_{s,t})=\omega_{0,t}$ (in particular $f_{1,t}^*(\omega'_t)=\omega_t$), and $f_{0,t}=\text{id}_B$. Further, for any $t$ such that $\partial_s \omega_{s,t}\equiv 0$, it may be assumed that $f_{s,t}=\text{id}_B$ for all $s$.
\end{remark}

\begin{proof}[Sketch of proof of Items (2) and (3) of \Cref{lem:equivalent}]
 Assume w.l.o.g. that $[t_0,t_1]=[0,1]$.
Item (3) is \Cref{rem:Moser}.

    To prove Item (2),
    define
    \[
    \mathcal{D}:=\left \{T\colon \left(\exists \; {(\omega_{s,t})}_{s\in [0,1], t\in [0,T]}: \quad \partial_s [\omega_{s,t}]=0,\;
    \omega_{0,t}=\omega_t, \omega_{1,t}=\omega'_t \text{ and } \left(\partial_s\omega_{s,t}=0\; \forall t\leq R \right) \right)\right \}
    \]
    (where it is understood that $T\in [0,1]$). We want to show that $\mathcal{D}=[0,1]$; we  will show that $\mathcal{D}$ is open and closed. Clearly, $R\in \mathcal{D}$, and $\mathcal{D}$ is open around $R$ since
    $$\omega_{s,t}=s\omega'_t+(1-s)\omega_t, \; (s,t)\in [0,1]\times [0,R+\delta]$$
    is an isotopy for $\delta>0$ small enough.
    
    Let us show that this is open in general. For any $R<T\in \mathcal{D}$ and the corresponding family $(\omega_{s,t})_{s\in [0,1], t\in [0,T]}$, we consider the isotopy $f_{s,t}$, $(s,t)\in [0,1]\times [0,T]$, obtained by Moser's method (see \Cref{rem:Moser}). This has the property that $f_{s,t}^*\omega_{s,t}=\omega_t$, $(s,t)\in [0,1]\times [0,T]$, and $f_{s,t}=\text{id}$ for $(s,t)\in [0,1]\times [0,R]$.\\
    We let $0<\delta<T-R$ such that $[0,T+\delta]\subset [0,1]$ and define, for $t\in [0,T+\delta]$, $f'_{t}:=f_{1,\rho(t)}$ where
    $\rho\colon [0,1]\to [0,1]$ is a smooth, monotone function such that
    \begin{enumerate}
        \item $\rho$ is the identity on $[0,T-\delta]$.
        \item $\rho(t)=T$ for $t\in [T-\delta/2,T+\delta]$.
        \item the distance of $\rho$ to the identity function with respect to the maximum norm is at most $\delta$.
    \end{enumerate}
        Then, since non-degeneracy is an open condition, all the forms
        $$ s (f'_{t})^*\omega'_{t} +(1-s) \omega_{t}=s f_{1,\rho(t)}^*\omega'{t} +(1-s) \omega_{t}, \; (t,s)\in[0,T+\delta]\times  [0,1],$$
        are indeed symplectic if $\delta$ is close enough to $0$ (compare with \Cref{rem:almostsymplecticopen}). Applying Moser's method again (\Cref{rem:Moser}), we obtain another smooth family $\tilde{f}_{s,t}$, $(s,t)\in [0,1]\times [0,T+\delta]$, with
        \begin{enumerate}
            \item $\tilde{f}_{s,t}={f}_{s,t}$ for $(s,t)\in [0,1]\times [0,T-\delta]$, in particular $\tilde{f}_{s,t}=\text{id}$ for $(s,t)\in [0,1]\times [0,R]$.
            \item $\tilde{f}_{1,t}^*\omega'_{t}=\omega_t$ for $t\in [0,T+\delta]$.
        \end{enumerate}
        This implies that $(\omega_t)_{t\in [T+\delta]}$ and $(\omega'_t)_{t\in [T+\delta]}$ are also equivalent via $\tilde{f}_{s,t}$, so $\mathcal{D}$ is open.

        That $\mathcal{D}$ is closed is by the same arguments as in the proof of \cite[Lemma 3.4]{Go11}. We sketch them here. If $0<T\leq 1$ is such that all $T-\eps$, $T>\eps>0$, are in $\mathcal{D}$, then Gonzales finds a smooth family $\hat{\beta}_{s,t}$, where $(s,t)\in [0,1]\times [0,T-\eps,T]$, of symplectic forms interpolating between $\omega_t$ and $\omega'_t$ using the rigidity. Following that, he takes the smooth family $\omega_{s,t}$, $(s,t)\in [0,1]\times [T-\eps]$, provided by the fact that $T-\eps \in \mathcal{D}$, and extends it with the help of $\hat{\beta}_{s,t}$ 'after smoothing' (we provide the details in  \Cref{lem:smoothingfamilies}). Any such operation would not change the way $\omega_{s,t}$ looks like for $t\leq R$. This shows that $\mathcal{D}$ is closed.
    \end{proof}

\begin{lemma}\label{lem:smoothingfamilies}
    Let $\omega_{s,t}$, $(s,t)\in [0,1]^2$, be a family of symplectic forms on a compact manifold $B$ with the property that
    \begin{enumerate}
        \item $\omega_{s,t}$ depends continuously on $(s,t)$ for all $(s,t)\in [0,1]^2$. 
        \item smoothly on $s$ for all $(s,t)\in [0,1]^2$ and smoothly on $t$ for all $(s,t)\in [0,1]\times ([0,1]\smallsetminus (1/2))$.
        \item $\partial_s [\omega_{s,t}]=0\in H^2(B;\R)$ for all $(s,t)\in [0,1]^2$.
        \item there is a cohomology class $[c]\in H^2(B;\R)$ with $\partial_t [\omega_{s,t}]=[c]$ \footnote{It would be enough to assume that $[\omega_{s,t}]$ is continuously differentiable, but we do not need this.}.
    \end{enumerate}
    Then we find a smooth family $\omega'_{s,t}$ of symplectic forms on $B$ such that
    \begin{enumerate}
        \item $\partial_s [\omega'_{s,t}]=0\in H^2(B;\R)$ for all $(s,t)\in [0,1]^2$.
        \item $\omega'_{0,t}=\omega_t$ and $\omega'_{1,t}=\omega'_t$ for all $t\in [0,1]$.
    \end{enumerate}
    If, furthermore, there is $0\leq R< 1/2$ such that $\omega_{s,t}$ does not depend on $s$ whenever $0\leq t\leq R$, then the same may be assumed for $\omega'_{s,t}$ as well. 
\end{lemma}
Before we prove the lemma, we need some preparation.
\begin{lemma}\label{lem:prepsmoothing}
    Let $B$ be a manifold and $\omega_t$, $\omega'_t$ and $\omega''_t$ a smooth family of symplectic forms on $B$, parametrized by some smooth manifold. Assume further that there are smooth families $\omega^1_{s,t}$ and $\omega^2_{s,t}$, $s\in [0,1]$, such that $\omega^1_{s,t}$ is an isotopy between $\omega_t$ and $\omega'_t$, and $\omega^2_{s,t}$ is an isotopy between $\omega'_t$ and $\omega''_t$.\\
    Then $\omega_t$ is isotopic to $\omega''_t$. To be precise, for any $1/2>\eps>0$ there is an isotopy $\omega_{s,t}$ between them that equals
    \[
    \omega''_{s,t}:=\begin{cases}
        \omega_{s,t}=\omega^1_{2s,t} & \text{if } (s,t)\in [0,1/2]\times [0,1]\\
        \omega_{s,t}=\omega^2_{2(s-1/2),t} & \text{if } (s,t)\in [1/2,1]\times [0,1]
    \end{cases}
    \]
    whenever $s\notin [1/2-\eps,1/2+\eps]$, and if there is $0\leq R\leq 1$ such that $\omega_t=\omega'_t=\omega''_t$ for all $0\leq t\leq R$, then $\omega_{s,t}$ does not depend on $s$ for all $0\leq t\leq R$.
\end{lemma}
\begin{proof}
    For $1/2>\eps>\delta>0$ let $\rho\colon [0,1]\to [0,1]$ be a monotone smooth function with the properties:
    \begin{itemize}
        \item $\rho(t')=1/2$ for all $t'\in [1/2-\delta',1/2+\delta']$,
        \item $\rho$ is the identity outside $[1/2-\eps',1/2+\eps']$.
    \end{itemize}
    Define $\omega_{s,t}:=\omega''_{s,\rho(t)}$. This has all the properties we want it to have.
\end{proof}
\begin{proof}[Proof of \Cref{lem:smoothingfamilies}]
    Let $c$ be a representative of $[c]$. Then there is $\delta>0$ such that $\omega_{s,t}- \kappa c$ is a symplectic form for all $(s,t)$ and $0\leq |\kappa| \leq \delta$. For $1/2>\eps'>\delta'>0$, let $\rho\colon [0,1]\to [0,1]$ be a monotone smooth function with the properties:
    \begin{itemize}
        \item $\rho(t')=1/2$ for all $t'\in [1/2-\delta',1/2+\delta']$,
        \item $\rho$ is the identity outside $[1/2-\eps',1/2+\eps']$.
    \end{itemize}
    Note that $\rho$ can be chosen arbitrarily close under the maximum norm to the identity map when $\eps'$ is chosen to be small enough. In particular, $\rho$ can be chosen to have distance less than $\delta$ to the identity. For that choice, we set
    $$\omega''_{s,t}:=\omega_{s,\rho(t)}.$$
    We note that $\omega''_{s,t}$ now depends smoothly on $s,t$, and that $\rho$ can also be chosen in such a way that $\omega''_{s,t}$ does not depend on $s$ for $t<R$ if that is true for $\omega_{s,t}$.\\
    We now define
    $$\omega'_{s,t}:=\omega''_{s,t}-(\rho(t)-t)c,$$
    which is symplectic by choice of $\rho$, since $|\rho(t)-t|<\delta$. Also, we note that $[\omega_{s,t}]=[\omega'_{s,t}]$ for all $(s,t)$ due to the assumption that $\partial_t[\omega_{s,t}]=c$, and again $\omega'_{s,t}$ does not depend on $s$ for $t<R$ if that is true for $\omega''_{s,t}$.\\
    Therefore, $\omega'_{s,t}$ now defines an isotopy between $\omega'_{0,t}$ and $\omega'_{1,t}$, but it does not hold necessarily $\omega'_{0,t}=\omega_{0,t}$ and $\omega'_{1,t}=\omega_{1,t}$. However, we claim that $\omega'_t:=\omega'_{0,t}$ and $\omega_t:=\omega_{0,t}$, for example, are isotopic under the standard homotopy if $\delta$ from above is chosen to be small enough; this would allow us to immediately finish the proof with \Cref{lem:prepsmoothing}.\\
    To see this, we write explicitly for $a\in [0,1]$
    $$a\omega'_t+(1-a)\omega_t= a (\omega_{\rho(t)}-(\rho(t)-t)c)+(1-a)\omega_t=[a\omega_{\rho(t)}+(1-a)\omega_t]+(1-a)(\rho(t)-t)c.$$
    It is now clear that, for each individual $t$, $\delta$ can be chosen small enough such that the above form is non-degenerate for all $a$, and so we find $\delta$ also for all $t$ due to compactness of $[0,1]$.
\end{proof}

\bibliographystyle{amsalpha}

\end{document}